\documentclass[12pt]{article}

\usepackage[a4paper, top=2cm, bottom=2cm, right=1.5cm, left=1.5cm]{geometry} % To set up page margins and paper size

\usepackage[T1]{fontenc}               % For correct font encoding (essential for hyphenation)
\usepackage[italian, english]{babel}     % For multilingual support (last language is the default)
\usepackage{lmodern}                   % Modern, high-quality vector fonts
\usepackage{textcomp}                  % Additional text symbols
\usepackage{microtype}                 % Improves micro-typography and text justification

\usepackage{mathtools, amssymb, amsthm}

\usepackage{mathrsfs}                  % Provides the calligraphic font with \mathscr{L}
\usepackage{dsfont}                    % Provides blackboard-bold "1" with \mathds{1} for indicator functions
\usepackage{accents}                   % Allows for multiple accents on a single character
\usepackage{aliascnt}                  % Required for the \nuovothm macro that handles theorem counters

\usepackage{xcolor}                    % To define and use colors
\definecolor{myurlcolor}{rgb}{0,0,0.4}
\definecolor{mycitecolor}{rgb}{0,0.5,0}
\definecolor{myrefcolor}{rgb}{0.5,0,0}

\usepackage{graphicx}                  % To include images with \includegraphics
\usepackage{tikz}                      % To create programmable vector graphics
\usepackage{tikz-cd}                   % To easily create commutative diagrams
\usetikzlibrary{graphs, decorations.pathmorphing, decorations.markings}
\usepackage{caption}                   % To customize figure and table captions

\usepackage[sort]{natbib}                    % For bibliography management (e.g., author-year styles)

\usepackage[pagebackref=true, colorlinks=true]{hyperref} 
\hypersetup{
    linkcolor=myrefcolor,
    citecolor=mycitecolor,
    urlcolor=myurlcolor,
}
\usepackage[capitalize, nameinlink]{cleveref}   % For "intelligent" cross-references (\cref)

\usepackage{changepage}                % For custom margins (used in proof environment)
\usepackage{enumitem}                  % To customize lists (enumerate, itemize)
\usepackage{braket}                    % For Dirac notation in physics/mathematics
\usepackage{multirow}
\usepackage{makecell}

\newtheorem{theorem}{Theorem}[section]

\newcommand\nuovothm[3]{
  \newaliascnt{#1}{theorem}
  \newtheorem{#1}[#1]{#2}
  \aliascntresetthe{#1}
  \crefname{#1}{#2}{#3}
}

\nuovothm{lemma}{Lemma}{Lemmas}
\nuovothm{proposition}{Proposition}{Propositions}
\nuovothm{corollary}{Corollary}{Corollaries}
\nuovothm{definition}{Definition}{Definitions}
\nuovothm{remark}{Remark}{Remarks}
\nuovothm{assumption}{Assumption}{Assumptions}
\nuovothm{example}{Example}{Examples}
\nuovothm{question}{Question}{Questions}

\newcommand{\be}{\begin{equation}}
\newcommand{\ee}{\end{equation}}

\newcommand{\dd}{{\rm d}}
\newcommand{\de}{\partial}
\DeclareMathOperator\T{\textup{\textbf{T}}}

\usepackage{todonotes}                 % For inserting notes and TODOs in the margin (disable with the [disable] option)

\usepackage{quiver}
\usepackage{newpxtext}

\newcommand{\Lie}{\pounds}
\renewcommand{\Im}{\operatorname{Im}}

\usepackage{parskip}
\title{A zoo of coisotropic embeddings}

\author{
    \texorpdfstring{M. De Le\'on$^{1,4,5}$ \href{https://orcid.org/0000-0002-8028-2348}{\includegraphics[scale=0.7]{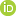}}}{M. De Leon},
    \texorpdfstring{R. Izquierdo-L\'opez$^{2,6}$ \href{https://orcid.org/0009-0007-8747-344X}{\includegraphics[scale=0.7]{ORCID.png}}}{R. Izquierdo-Lopez},
    \texorpdfstring{L. Schiavone$^{3,7}$ \href{https://orcid.org/0000-0002-1817-5752}{\includegraphics[scale=0.7]{ORCID.png}}}{L. Schiavone},
    \texorpdfstring{P. Soto$^{1,8}$ \href{https://orcid.org/0009-0001-6p58-3182}{\includegraphics[scale=0.7]{ORCID.png}}}{P. Soto} \\
    \footnotesize{$^{1}$\textit{Instituto de Ciencias Matemáticas, Campus Cantoblanco, Consejo Superior de Investigaciones Científicas,}} \\
    \footnotesize{\textit{Calle Nicolás Cabrera, 13–15, 28049, Madrid, Spain}} \\
    \footnotesize{$^{2}$\textit{Escuela Superior de Ingeniería y Tecnología, Universidad Internacional de La Rioja,}} \\
    \footnotesize{\textit{Calle Luis de Ulloa, 4, 26004, Logroño, Spain}} \\
    \footnotesize{$^{3}$\textit{Dipartimento di Matematica e Applicazioni Renato Caccioppoli, Università degli Studi di Napoli Federico II,}} \\
    \footnotesize{\textit{Via Cintia, Monte S. Angelo I, 80126, Napoli, Italy}} \\
    \footnotesize{$^{4}$\textit{Real Academia de Ciencias Exactas, Físicas y Naturales de España, C/Valverde, 22, Madrid 28004, Spain}}\\
    \footnotesize{$^{5}$\textit{e-mail: \texorpdfstring{\texttt{mdeleon@icmat.es}}{mdeleon@icmat.es}}} \\
    \footnotesize{$^{6}$\textit{e-mail: \texorpdfstring{\texttt{ruben.izquierdo@unir.net}}{ruben.izquierdo@unir.net}}} \\
    \footnotesize{$^{7}$\textit{e-mail: \texorpdfstring{\texttt{luca.schiavone@unina.it}}{luca.schiavone@unina.it}}} \\
    \footnotesize{$^{8}$\textit{e-mail: \texorpdfstring{\texttt{pablo@soto.es}}{pablo@soto.es}}}\\
}

\begin{document}

\maketitle

\begin{abstract}
The aim of this paper is to extend the coisotropic embedding theorem obtained by \textit{M. J. Gotay} for pre-symplectic manifolds to more general geometric settings: cosymplectic, contact, cocontact, $k$-symplectic, $k$-cosymplectic, $k$-contact, and multisymplectic manifolds. The results are obtained by applying a generic methodology, which gives more relevance to the potential applications. In that sense, this paper gives the fundamental basis to be able to apply the results to the so-called regularization problem of singular Lagrangian systems, both in mechanics and in classical field theories.
\end{abstract}


\tableofcontents

\section{Introduction}
%\addcontentsline{toc}{section}{Introduction}

Symplectic geometry is the natural framework for developing Hamiltonian mechanics. Indeed, the phase space of a classical Hamiltonian system is the cotangent bundle $\T^\ast Q$ of the configuration manifold $Q$, which is endowed with a canonical symplectic structure $\omega_Q$. The solutions of the Hamilton equations are then the integral curves of the Hamiltonian vector field $X_H$ obtained from the Hamiltonian energy $H$ using the canonical symplectic form $\omega_Q$. In the case of Lagrangian mechanics, the velocity space is the tangent bundle $\T Q$ of the configuration manifold. In this case there is no canonical symplectic structure, but we can construct a $2$-form usually denoted by $\omega_L$ from the  Lagrangian function $L$ and using the so-called almost-tangent geometry of the tangent bundle. In the latter case, if the Lagrangian is regular (which is usual for Lagrangians appearing in mechanics), then $\omega_L$ is symplectic and the corresponding Hamiltonian vector field for the energy of $L$ is a Second Order Differential Equation (SODE) whose solutions are just the ones of the Euler-Lagrange equations determined by the Lagrangian (see \cite{abraham,dLR89,godbillon1969geometrie}).
%%%

\noindent This geometrical description of mechanics is valid when the Hamiltonian or the Lagrangian does not explicitly depend on time; if it does, we must consider the spaces $\mathbb{R} \times \T^*Q$ and $\mathbb{R} \times \T Q$, respectively. Now, the geometric structures we must use to obtain the dynamics are the so-called cosymplectic structures \cite{albert,cantrijn,dLR89,cosimpl}. A cosymplectic structure on a manifold $M$ is a pair $(\eta,\,\omega)$ where $\eta$ is a closed $1$-form and $\omega$ is a closed 2-form, such that $\eta \wedge \omega^n \neq 0$, where the dimension of $M$ is precisely $2n+1$.
%%%

\noindent Other interesting systems are the so-called Lagrangian systems with Lagrangian functions depending on the action. To obtain the equations of motion, one needs to extend the usual Hamilton principle to the so-called Herglotz principle \cite{bravetti1,bravetti,mlvmdlmml}. The geometric versions are developed on $\T^*Q \times \mathbb{R}$ and $\T Q \times \mathbb{R}$ for the Hamiltonian and Lagrangian cases, respectively. 
%%%
The geometry behind these systems is provided by a contact form, say, a $1$-form $\eta$ such that $\eta \wedge \dd \eta^n \neq 0$ \cite{mlvmdl1} (see also \cite{multicontact,chicos} for an extension to classical field theories). 
%%%

\noindent Finally, another step is the geometric description of classical field theories, based on the so-called multisymplectic geometry \cite{dLR2,gimmsy,GotayMultisymplecticFramework,binz,narciso,aitor}. Multisymplectic geometry has emerged as a fundamental framework for the geometric description of Field Theories, extending the role that symplectic geometry plays in Classical Mechanics. Alternative approaches to classical Field Theories uses completely different geometric structures, namely, $k$-symplectic, $k$-cosymplectic and $k$-contact structures \cite{silvia2,cata_k_resto,narciso_nuevo}).
%%%

\noindent As well as symplectic geometry has allowed us to obtain new results in mechanics, such as symplectic reduction, the various generalizations of Noether's theorem, study of stability, development of geometric integrators, coisotropic reduction, coisotropic regularization, among many others, something similar occurs in the other geometric scenarios mentioned above, although the extensions are in many cases not as direct as they might seem at first glance. 
%%%

\noindent A very useful tool in symplectic geometry, which is crucial for the problem of coisotropic regularization mentioned above, is the so-called \textit{coisotropic embedding theorem}. 
%%%
If $N$ is a submanifold of a symplectic manifold $(M, \omega)$, we say that it is coisotropic if $\T N^\perp \subset \T N$, where $\T N^\perp$ is the symplectic complement of $\T N$ in the tangent bundle $\T M$. On the other hand, a pre-symplectic manifold is a pair $(M, \omega)$ where $\omega$ is closed but not of maximal rank (if the rank is maximal, then it is a symplectic manifold). 
The coisotropic embedding theorem, proved by \textit{M. J. Gotay} in \cite{gotay}, states that any pre-symplectic manifold can be coisotropically embedded in a symplectic manifold (see also \cite{marle,zambon}). The existence is unique up to symplectomorphisms of a neighborhood of the embedded submanifold (this result lies on previous ones by \textit{A. Weinstein} \cite{weinstein1977lectures,weinstein71symplectic}).
%%%

\noindent The coisotropic embedding theorem, beyond its intrinsic mathematical interest, has proven to be a powerful tool with numerous applications in mathematical physics. For instance, it is crucial for extending the celebrated Noether's theorem to constrained mechanical systems. In this context, it allows for the formalization of a one-to-one correspondence between the system's symmetries and its constants of motion, even in the presence of a pre-symplectic structure \cite{Ciaglia-DiCosmo-Ibort-Marmo-Schiavone-Zampini-Symmetry-2022}. 
%%%

\noindent Another area of great relevance is the study of field theories, where the transition from a pre-symplectic to a symplectic space of solutions—guaranteed by an infinite-dimensional version of the theorem—is a fundamental step in defining a well-posed Poisson structure. 
%%%
This approach has been successfully applied to various theories, including gauge theories and the Palatini formulation of General Relativity \cite{Ciaglia-DiCosmo-Ibort-Marmo-Schiavone-Zampini-Peierls1-2024,Ciaglia-DiCosmo-Ibort-Marmo-Schiavone-Zampini-Peierls2-2022,Ciaglia-DiCosmo-Ibort-Marmo-Schiavone-Zampini-Peierls3-2022}. 
%%%

\noindent The theorem has also been used by one of the authors of the present paper for tackling the inverse problem of the calculus of variations. Indeed, it has been employed to construct variational principles for classes of implicit differential equations for which a Lagrangian formalism is not immediately apparent \cite{Schiavone-InverseProblemElectrodynamics-2024,Schiavone-InverseProblemImplicit-2024}. 
%%%

\noindent The theorem also plays an essential role in the Geometric Quantization program. 
%%%
Since the quantization procedure requires a symplectic manifold as a starting point, the embedding of a pre-symplectic manifold (which typically describes a classical system with constraints) into a larger symplectic one is an indispensable prerequisite for quantizing such systems \cite{Gotay-Sniatycki-CoisotropicEmbeddings-1981}.
%%%

\noindent As a last application of the theorem, we mention the so-called problem of regularization. 
%%%
Indeed, if a Lagrangian function is not regular (this happens in several physical theories), the corresponding 2-form $\omega_L$ is pre-symplectic (assuming some regularity condition on the rank). For this kind of Lagrangian functions, \textit{P. A. M. Dirac} and \textit{P. G. Bergmann} developed a constraint algorithm (the so-called Dirac-Bergmann algorithm \cite{Anderson-Bergmann-Constraints-1951,dirac}) settled respectively in the Hamiltonian and the Lagrangian formalisms, and that has been geometrized by \textit{M. J. Gotay}, \textit{J. N. Nester} and \textit{G. Hinds} \cite{gotay1,GotaySingularLagrangians1,GotaySingularLagrangians2}. The algorithm produces a sequence of constraint submanifolds just to find a final constraint submanifold that, for theories admitting gauge symmetries, is a pre-symplectic manifold. At this stage, one can use the coisotropic embedding theorem to embed the final constraint manifold in a larger symplectic one in such a way that the corresponding dynamics are conveniently related (see \textit{A. Ibort} and \textit{J. Mar{\'i}n-Solano} \cite{Ibort-Solano}).
%%%

\noindent With this last application of the coisotropic embedding theorem in mind, it is clear that, extending both Dirac-Bergmann algorithm and the coisotropic embedding theorem would be crucial to address the problem of classifying and regularizing singular Lagrangians in all the geometric frameworks described in this introduction.
%%%

\noindent An extension of the Dirac-Bergmann constraint algorithm for singular time-dependent Lagrangian systems exists and has been developed by \textit{D. Chinea}, \textit{M. de Le\'on} and \textit{J. C. Marrero} \cite{chinea}. To be able to achieve a regularization for singular Lagrangians depending on time of the type of \cite{Ibort-Solano}, it is clear that the first, obvious, step should be extending the coisotropic embedding theorem to the cosymplectic scenario. This is one of the generalizations achieved in the present paper. 
%%%
The application of the above results to the case of singular time-dependent Lagrangian systems is the subject of a forthcoming paper in elaboration. 
%%%

\noindent On the other hand, singular classical field theories have been recently considered in the multisymplectic setting (see \cite{aitor,singularfields1,arturonarciso} and references therein). 
%%%
We plan to obtain a classification of Lagrangians on the space of 1-jets of the configuration bundle, similar to the one obtained in \cite{Ibort-Solano}, using the operator introduced by \textit{D. Saunders} \cite{saunders} from a chosen volume
form on the space-time manifold (the base of the configuration bundle), and to study the regularization problem for singular classical field theories in a forthcoming paper. 
%%%

\noindent A similar problem arises for singular classical field theories described by means of the so-called $k$-symplectic \cite{ksympsingular}, $k$-cosymplectic \cite{kcosympsingular}, and $k$-contact \cite{skinnerrusk} formulations.
%%%

\noindent In the present paper, we make one of the steps described above that are necessary to classify and regularize singular Lagrangians in all the geometric pictures mentioned, namely we extend the approach introduced in \cite{luca2} (obtained building upon the alternative proof of the classical coisotropic embedding theorem of \cite{luca1}) for the construction of multisymplectic thickenings of pre-multisymplectic manifolds, to all the other geometric scenarios mentioned above. 
%%%
This approach uses the notion of almost product structure (a decomposition of the tangent bundle on complementary distributions) and provides a systematic way to discuss both the existence and the uniqueness of coisotropic embeddings in all the cases mentioned.
%%%

\noindent In the multisymplectic case, the opposite direction, namely the study of coisotropic reduction has been considered in \cite{coisored}.
%%%

\noindent This paper is structured as follows. 
%%%
In \cref{Sec: Preliminaries} we develop all the mathematical background on the geometric notions that are necessary to develop the theory, namely, distributions, Frobenius Theorem and almost product structures. 
%%%
In \cref{Sec: Pre-symplectic manifolds} we recall the main results in the case of pre-symplectic manifolds, but the proofs follow the methodology of \cite{luca2} that we will adopt throughout the whole paper.
%%%
The geometries characterized by the existence of Reeb vector fields, namely, pre-cosymplectic manifolds, pre-contact manifolds, and pre-cocontact manifolds, are discussed in \cref{Sec: cosymplectic manifolds}, \cref{Sec: contact manifolds} and \cref{Sec: cocontact manifolds}.
%%%
The results in \cref{Sec: Pre-symplectic manifolds}, \cref{Sec: cosymplectic manifolds}, and \cref{Sec: contact manifolds} are then generalized to the $k$-pre-symplectic manifolds, $k$-pre-cosymplectic manifolds, and $k$-pre-contact manifolds in \cref{Sec: k-symplectic manifolds}, \cref{Sec: k-cosymplectic manifolds}, and \cref{Sec: k-contact manifolds}. 
%%%
Finally, in \cref{Sec: multisymplectic manifolds} we recall the case of multisymplectic geometry already analyzed in \cite{luca2}, now enriched with a discussion on uniqueness. 
%%%

\noindent The order of presentation of the various sections has been guided by the increasing difficulty in proving the uniqueness of coisotropic embeddings.
%%%
Indeed, uniqueness only holds in the pre-symplectic case; nevertheless, for geometries modelling classical mechanics (time-dependent, dissipative, or both), namely, pre-cosymplectic, pre-contact, and pre-cocontact geometry, we will be able to prove that coisotropic embeddings must have fixed topology (so that any two such embeddings are diffeomorphic, but not necessarily co-symplectomorphic, contactomorphic, or co-contactomorphic). Furthermore, we will prove that it is sufficient and necessary for two embeddings to be neighborhood equivalent that the corresponding Reeb vector fields are proportional. These statements no longer hold for geometric structures modelling field theories, namely, in the cases of $k$-pre-symplectic, $k$-pre-contact, and multisymplectic geometries, as the generality needed translates into the existence of a plethora of different coisotropic embeddings. Our methodology, in these cases, will be to study uniqueness by fixing the local model of the thickening. In all the cases considered, we will give one model that allows for uniqueness, and one where uniqueness does not necessarily hold.
%%%

\noindent Finally, we mention some conclusions and future work.

\section{Preliminaries}
\label{Sec: Preliminaries}

In this section, we will recall the basic notions that we will use in the rest of the paper.

\subsection{Distributions and Frobenius theorem}

\begin{definition}[\textsc{Distribution}]
\label{Def: Distribution}
Let \( M \) be a smooth $d$-dimensional manifold.
%%%
A \textbf{distribution} \( \mathcal{D} \) on \( M \) is an assignment to each point \( m \in M \) of a linear subspace \( \mathcal{D}_m \subset \mathbf{T}_m M \).
%%%
The distribution is said to be \textbf{smooth} if for every \( m \in M \) there exists an open neighborhood \( U_m \subset M \) of \( m \) and a collection of smooth vector fields \( X_1, \dots, X_r \in \mathfrak{X}(U_m) \) satisfying
\[
\mathcal{D}_n = \operatorname{span} \{ X_1(n), \dots, X_r(n) \}, \quad \text{for all } n \in U_m.
\]
The integer $r$ is the \textbf{local rank} of the distribution on $U_m$.
\end{definition}

\begin{remark}
If the local rank of a smooth distribution is not constant across the manifold, the distribution is called a \textbf{singular distribution}.
%%%
In this work, however, we will primarily focus on \textbf{constant-rank distributions}, also referred to as \textbf{regular distributions}, i.e., those for which the integer $r$ is the same for every point $m \in M$.
%%%
In this case,
\be
\mathcal{D} \,:=\, \bigsqcup_{m \in M} \mathcal{D}_m \,,
\ee
defines a smooth rank-$r$ subbundle of the tangent bundle $\mathbf{T}M$.
%%%
\noindent More specifically, all distributions considered in this work will arise as the kernel of a differential form of constant rank
\be
\mathcal{D}_m \,=\, \mathrm{ker}\, \omega_m \,,
\ee
for some $\omega \in \Omega^k (M)$ with constant rank, or as intersections of kernels of differential forms of this type. 
%%%
{Equivalently, it can be interpreted as a smooth section of the $r$-th Grassmann bundle $G^r(TM) \rightarrow M$, that is, the fiber bundle over $M$ whose fiber at each point $m \in M$ is the Grassmannian $G(r, T_m M)$, i.e., the manifold of all $r$-dimensional linear subspaces of the tangent space $T_m M$.}
\end{remark}

\begin{definition}[\textsc{Integral manifold}]
Let \( \mathcal{D}\) be a smooth distribution on a smooth manifold \( M \).
%%%
An \textbf{integral manifold} of \( \mathcal{D} \) is a connected immersed submanifold  
\be
\mathfrak{i}  \colon  N \hookrightarrow M \,, 
\ee
such that
\be
\mathfrak{i}_* (X_n) \in  \mathcal{D}_{\mathfrak{i}(n)}\,,\;\; \forall \, n \in N,\, \text{and } \forall \, X_n \in \mathbf{T}_n N.
\ee
\end{definition}

\begin{definition}[\textsc{Integrable distribution}]
A smooth distribution \( \mathcal{D} \) on a smooth manifold $M$, is said to be \textbf{integrable} if, for every point \( m \in M \), there exists an integral manifold \( N \subset M \) passing through \( m \) such that \( T_n N = \mathcal{D}_n \) for all \( n \in N \).
\end{definition}

\begin{definition}[\textsc{Smooth foliation}] \label{Def: smooth foliation constant rank}
Let $M$ be a smooth $d$-dimensional manifold. 
%%%
A \textbf{smooth foliation of rank} $r$ (or \textbf{codimension} $d - r$) on $M$ is a decomposition of $M$ into a union of connected immersed submanifolds $\{ L_j \}_{j \in \mathscr{J}}$ ($\mathscr{J}$ denoting an index set), called \textbf{leaves}, such that for every point $m \in M$, there exists a coordinate chart $(U_m, \varphi)$
\be
\varphi  \colon  U_m \to \mathbb{R}^d  \colon  m \mapsto \varphi(m) \,=\, (x^1,...,\,x^l,\, z^1,...,\,z^r) \,,
\ee
with $l\,=\, d-r$, satisfying:
\begin{enumerate}
    \item For every leaf $L_j$, the connected components of $L_j \cap U$ are mapped by $\varphi$ to sets of the form
\be
    \varphi(L_j \cap U) = \{ (c^1,...,\,c^l,\, z^1,...,\,z^r) \in \mathbb{R}^d \},
\ee
    for some constants $c^{1}, \dots, c^l \in \mathbb{R}$.
    \item The local coordinate vector fields on $U$ 
    \be
    \left\{\,\frac{\partial}{\partial z^1}, \dots, \frac{\partial}{\partial z^r}\,\right\} \,,
    \ee 
    span a smooth rank-$r$ distribution $\mathcal{D} \subset \mathbf{T}U$, which coincides with the tangent spaces to the leaves.
\end{enumerate}
We refer to such charts as \textbf{foliated charts}.
\end{definition}
%%%

\begin{remark}
Smooth foliations of constant rank described in \cref{Def: smooth foliation constant rank} are also referred to as \textbf{regular foliations}.   
\end{remark}

\begin{definition}[\textsc{Completely integrable distribution}]
A smooth distribution \( \mathcal{D} \)  of constant rank on a smooth manifold $M$ is said to be \textbf{completely integrable} if it is integrable and its maximal integral manifolds define a smooth foliation of \( M \).
\end{definition}
%%%

\begin{theorem}[\textsc{Frobenius Theorem}]
Let $\mathcal{D}$ be a smooth rank-$r$ distribution on a smooth manifold $M$.
%%%
The following statements are equivalent:
\begin{enumerate}
    \item \textbf{Involutivity:} For all vector fields $X, Y \in \mathfrak{X}(M)$ belonging, at each $m \in M$, to $\mathcal{D}_m$, the Lie bracket $[X, Y]$ belongs, at each $m \in M$, to $\mathcal{D}_m$.
    %%%
    \item \textbf{Complete integrability:} The distribution $\mathcal{D}$ is completely integrable.
    %%%
    \item \textbf{Local triviality:} Around each point $m \in M$, there exist local coordinates $(x^1,...,\, x^l,\,z^1,...,\,z^r)$ defined on a neighborhood $U_m \ni m$ such that
\[
    \mathcal{D}_n = \operatorname{span} \left\{ \frac{\partial}{\partial z^1}, \dots, \frac{\partial}{\partial z^r} \right\} \Bigg|_n \quad \text{for all } n \in U_m \,.
\]
\end{enumerate}
\begin{proof}
See \cite{warner}   
\end{proof}
\end{theorem}

\noindent Frobenius Theorem has an alternative version using differential forms.
%%%

\noindent Indeed, let $\mathcal{D}^o$ be the annihilator of $\mathcal{D}$, that is, the space of $1$-forms on $M$ vanishing on $\mathcal{D}$. 
%%%
We can construct an ideal of the algebra of differential forms on $M$ by defining
$$
\mathcal{I}(\mathcal{D}^o) =
\{\alpha \wedge \beta, \hbox{where} \; \beta \in \mathcal{D}^o\}
$$
Then, we have the following result.
\begin{theorem}
$\mathcal{D}$ is involutive if and only if $\mathcal{I}(\mathcal{D}^o)$ is  differential ideal, that is,
$$
d(\mathcal{I}(\mathcal{D}^o)) \subset \mathcal{I}(\mathcal{D}^o).
$$
\begin{proof}
See \cite{warner}   
\end{proof}
\end{theorem}

\begin{remark}
When a smooth distribution $\mathcal{D}$ on a smooth manifold $M$ arises as the kernel of a differential $k$-form $\omega \in \Omega^k(M)$ of constant rank, the involutivity of $\mathcal{D}$ is equivalent to the condition 
\be
\dd \omega \in \mathcal{I}^2(\operatorname{ker}\,\omega) \,, 
\ee
where $\mathcal{I}(\operatorname{ker}\,\omega)$ is the ideal generated by the annihilator of $\operatorname{ker}\,\omega$, and its second power $\mathcal{I}^2(\operatorname{ker}\,\omega)$ is defined as
\be
\mathcal{I}^2(\operatorname{ker}\,\omega) \,=\, \left\{\, \sum_{i,j} \alpha_i \wedge \beta_j \,,\;\;\; \text{for}\,\, \alpha_i,\,\beta_j \in \mathcal{I}(\operatorname{ker}\,\omega) \,\right\} \,.
\ee
%%%
For $k=1$ this condition coincides with
\be
\dd \omega \wedge \omega \,=\, 0 \,.
\ee
%%%
\end{remark}

\noindent The Frobenius theorem provides a complete picture for constant-rank distributions. 
%%%
For completeness, we now briefly consider the more general case of singular distributions. 
%%%
This requires introducing the concept of a \textit{singular foliation}.
%%%

\begin{definition}[\textsc{Singular foliation}]
Let $M$ be a smooth $d$-dimensional manifold.
%%%
A \textbf{singular foliation} on $M$ is a partition of $M$ into a disjoint union of connected, immersed submanifolds $\{ L_j \}_{j \in \mathscr{J}}$, called \textbf{leaves}, which may have varying dimensions.
%%%
This partition defines a singular distribution $\mathcal{D}$ by setting, for each $m \in M$,
\be
\mathcal{D}_m \,=\, \mathbf{T}_m L_j \,, \quad \text{where } m \in L_j\,.
\ee
We say that the foliation is \textbf{generated} by the distribution $\mathcal{D}$.
\end{definition}

\begin{theorem}[\textsc{Stefan-Sussmann's Theorem}]
Let $\mathcal{D}$ be a singular distribution on a smooth manifold $M$, generated by a family of smooth vector fields $\mathcal{F} \subseteq \mathfrak{X}(M)$.
%%%
The following statements are equivalent:
\begin{enumerate}
    \item \textbf{Involutivity:} The set of all smooth vector fields that are sections of $\mathcal{D}$, denoted $\Gamma(\mathcal{D})$, forms a Lie subalgebra of $\mathfrak{X}(M)$. That is, for any two vector fields $X, Y \in \Gamma(\mathcal{D})$, their Lie bracket $[X, Y]$ is also in $\Gamma(\mathcal{D})$.
    %%%
    \item \textbf{Integrability:} For each point $m \in M$, there exists a unique maximal, connected, immersed submanifold (the \textit{leaf}) passing through $m$ that is an integral manifold of $\mathcal{D}$.
    %%%
    \item \textbf{Foliation:} The distribution $\mathcal{D}$ generates a singular foliation on $M$, whose leaves are the maximal integral manifolds from the previous point.
\end{enumerate}
\begin{proof}
See \cite{stefan,sussmann}
\end{proof}
\end{theorem}

\subsection{Moser's trick and relative Poincaré Lemma}

\noindent The following strategy, often referred to as \textit{Moser's trick}, was originally given by Jürgen Moser in 1965 to check when two volume forms are equivalent \cite{moser}, but its main applications are in symplectic geometry. It is the standard argument for the modern proof of Darboux's theorem, as well as for the proof of Darboux-Weinstein theorem and other normal form results \cite{weinstein71symplectic}. It will be relevant to discuss the uniqueness of coisotropic embeddings.
%%%

\begin{theorem}[\textsc{Moser's Trick}]
\label{thm:Moser_trick}
Let $M$ be a manifold, and let $\omega_0, \omega_1 \in \Omega^k(M)$ be two differential $a$-forms. 
%%%
If there exists a complete time-dependent vector field $X_t \in \mathfrak{X}(M)$ such that 
\[
\omega_1 - \omega_0 + \Lie_{X_t} (t \omega_1 + (1-t) \omega_0) = 0\,,
\]
then, there exists a diffeomorphism $\psi\colon M \rightarrow M$ satisfying $\psi^\ast \omega_1 = \omega_0$.
\end{theorem}
\begin{proof}
Let $\psi_t\,:\; M \to M$ be the flow of the vector field $X_t$, such that $\psi_0 = \mathrm{id}_M$. 
%%%
We will show that its time-one map, $\psi := \psi_1$, satisfies $\psi^*\omega_1 = \omega_0$.
%%%

\noindent To do this, we follow the "trick" of showing that the family of forms $\omega_t := \psi_t^*(t \omega_1 + (1-t) \omega_0)$ (that coincides with $\omega_0$ for $t=0$ and with $\omega_1$ for $t=1$) is constant for all $t \in [0, 1]$. 
%%%

\noindent Indeed, 
\begin{align*}
\frac{\dd}{\dd t} \psi_t^\ast(\omega_t) &= \frac{\dd}{\dd t} \left( \psi_t^*(t \omega_1 + (1-t) \omega_0) \right) \\
&= \psi_t^* \left( \mathcal{L}_{X_t}(t \omega_1 + (1-t) \omega_0) + (\omega_1 - \omega_0) \right)\,,
\end{align*}
and the last line vanishes by assumption, so that it satisfies $\psi^ \ast_t \omega_t = \psi^\ast_0 \omega_0 = \omega_0$, for $t \in [0,1]$ Taking $t = 1$ we obtain the desired diffeomorphism.
%%%
\end{proof}

\noindent Lastly, we add the proof of the Relative Poincaré Lemma, which we use several times throught our study:

\begin{theorem}[\textsc{Relative Poincaré Lemma}]
\label{thm:Relative_Poincaré_Lemma}
Let $\mathfrak{i} \colon N \hookrightarrow M$ be an immersed submanifold and let $\omega$ be a closed $k$-form on $M$ that vanishes on $N$. Then, there is a neighborhood $U$ of $N$ in $M$ and an $(k-1)$-form $\theta$ on $M$ that vanishes on $N$ such that $\omega = \dd \theta$.
\end{theorem}

\begin{proof}
Let $U$ be a tubular neighborhood of $N$ in $M$, endowed with a surjective submersion $\pi \colon U \rightarrow N$ and a vector bundle structure. Let $\Delta$ denote the Liouville vector field. Define, for $t \in [0,1]$ the $1$-parameter family of maps \[
\psi_t \colon U \rightarrow U\,, \qquad \psi_t(u) := t \cdot u\,.
\]
Let $\Delta_t := \frac{\dd \psi_t}{\dd t}$. Now, integrating $\frac{\dd}{\dd t}\left(\psi_t^\ast \omega\right)$ we have the following:
\begin{align*}
    \psi_1^\ast \omega - \psi_0^\ast \omega &=  \int_0^1 \frac{\dd}{\dd t} \left(\psi_t^\ast \omega \right)\dd t = \int_0^1 \psi_t^\ast\Lie_{ \Delta_t} \omega \dd t\\
    &=\int \psi_t^\ast\left(\dd i_{ \Delta_t} \omega + i_{ \Delta_t}\dd \omega \right)\dd t =  \int_0^1 \dd \psi_t^\ast i_{\Delta_t} \omega \dd t\\
    &= \dd \int_0^1\psi_t^\ast i_{\Delta_t} \omega \dd t\,.
\end{align*}
Notice that $\psi_1^\ast \omega = \omega$, and that $\psi_0^\ast \omega = \pi^\ast \omega|_N = 0$, since $\omega$ vanishes at $N$. Now, it is clear that if we define 
\[
\theta := \int_0^1 \psi_t^\ast i_{\Delta_t} \omega \dd t\,,
\]
we have $\omega = \dd \theta$. It only remains to notice that $\theta$ also vanishes at $N$, since $\Delta$ does as well.
\end{proof}

\subsection{Differential calculus and almost product structures}

In this section, we will recall the notion of \textit{almost product structure} introduced by \textit{A. G. Walker} \cite{walker} and \textit{T. Fukami} \cite{fukami}, and extensively studied in the sixties of the past century. 
%%%
We will follow the approach in \cite{dLR89}. 
%%%
An interpretation in terms of $G$-structures can be found in \cite{fujimoto}.

\begin{definition}[\textsc{Almost product structure}]
Let $M$ be a smooth $d$-dimensional manifold. 
%%%
An \textbf{almost product structure} on $M$ is a choice of two smooth distributions $\mathcal{H}, \mathcal{V}$ of constant rank, sometimes called the \textbf{horizontal} and \textbf{vertical distributions}, such that
\be
\mathbf{T}_m M = \mathcal{H}_m \oplus \mathcal{V}_m \,, \;\; \forall \,\, m \in M \,,
\ee
or
\be
\mathbf{T} M = \mathcal{H} \oplus \mathcal{V} \,,
\ee
in terms of the {Whitney} sum of vector bundles.
%%%
\end{definition}
%%%

\begin{remark}
In this work, we will only consider almost product structures in which one of the two distributions, say $\mathcal{V}$, arises as the kernel of a differential $k$-form of constant rank, or as the intersection of kernels of $k$-forms, and is completely integrable. 
%%%
The complementary distribution $\mathcal{H}$ is then chosen so that $\mathbf{T}M = \mathcal{H} \oplus \mathcal{V}$.
\end{remark}
%%%

\noindent Given an almost product structure $\mathbf{T}M = \mathcal{H} \oplus \mathcal{V}$, one can define a $(1,1)$-tensor field $P$ such that
\be
\begin{split}
P^2 &= P \,, \\
\operatorname{Im}(P) &= \mathcal{V} \,, \\
\ker(P) &= \mathcal{H} \,,
\end{split}
\ee
%%%
{namely the projection onto $\mathcal{V}$.}
Conversely, any such tensor defines an almost product structure with horizontal and vertical distributions given by its image and kernel, respectively.
%%%

\noindent Thus, given a smooth completely integrable distribution $\mathcal{V}$ on $M$ of rank $r$, there is a one-to-one correspondence between:
\begin{itemize}
    \item smooth $(1,1)$-tensor fields $P$ on $M$ satisfying $P^2 = P$ and $\operatorname{Im}(P) = \mathcal{V}$;
    \item smooth distributions $\mathcal{H}$ on $M$ such that $\mathcal{H}$ and $\mathcal{V}$ provide an almost product structure.
\end{itemize}
%%%
In particular, in the foliated chart associated with the completely integrable distribution $\mathcal{V}$, where
\be
\mathcal{V} \,=\, \operatorname{span}\left\{\, V_A \,:=\, \frac{\de}{\de z^A} \,\right\}_{A\,=\,1,...,\,r} \,,
\ee
the $(1,1)$-tensor $P$ reads
\be
P \,=\, (\dd z^A - P^A_a \dd x^a) \otimes V_A \,=:\, P^A \otimes V_A \,, \;\; a \,=\, 1,...,\,l;\, A\,=\, 1,...,\,r \,,
\ee
where $P^A_a$ are smooth functions on $M$, whose kernel reads
\be
\mathrm{ker}\,P \,=\, \operatorname{span}\left\{\, H_a \,:=\, \frac{\de}{\de x^a} + P^A_a \frac{\de}{\de z^A} \,\right\}_{a\,=\,1,...,\,l} \,=:\, \mathcal{H} \,.
\ee
%%%

\begin{remark}
Note that, given an almost product structure $(\mathcal{V},\, \mathcal{H})$ on a smooth manifold $M$ defined by the $(1,1)$-tensor $P$ as $\mathcal{V} \,=\, \mathrm{Im}\,P$, $\mathcal{H} \,=\, \mathrm{ker}\,P$, one can always define the opposite tensor $R \,:=\, \mathds{1} - P$ providing the opposite almost product structure $\mathbf{T}M \,=\, \mathcal{H} \oplus \mathcal{V}$, since
\be
\begin{split}
\mathcal{V} \,&=\, \mathrm{ker}\,R \,=\, \mathrm{Im}\,P \,,\\
\mathcal{H} \,&=\, \mathrm{Im}\,R \,=\, \mathrm{ker} \,P \,.
\end{split}
\ee
%%%
In the system of foliated charts associated with $\mathcal{V}$, $R$ reads
\be
R \,=\, \dd x^a \otimes H_a \,.
\ee
\end{remark}
%%%

\begin{proposition}[\textsc{Splitting of the cotangent bundle}]
\label{prop:split_cotangent_bundle}
Any almost product structure $(\mathcal{V},\, \mathcal{H})$ provided by the $(1,1)$-tensor $P$ over a smooth manifold $M$, provides a splitting of the cotangent bundle $\mathbf{T}^* M$ of the type
\be
\mathbf{T}^* M \,=\, \mathcal{V}^* \oplus \mathcal{H}^* \,,
\ee
for $\mathcal{V}^*$ and $\mathcal{H}^*$ the subbundles of $\mathbf{T}^* M$ dual to $\mathcal{V}$ and $\mathcal{H}$.
%%%
\begin{proof}
Consider a differential $1$-form $\alpha$ on $M$, namely a section of $\mathbf{T}^* M$.
%%%
Given the almost product structure $(\mathcal{V},\, \mathcal{H})$ given by the $(1,1)$-tensor $P$, $\alpha$ decomposes as
\be \label{Eq: decomposition 1-forms}
\alpha \,=\, \alpha^\parallel_P + \alpha^\perp_P \,,
\ee
where 
\be
\alpha^\parallel_P \,:=\, \alpha \circ P \,,
\ee
and
\be
\alpha^\perp_P \,:=\, \alpha - \alpha^\parallel_P \,.
\ee
%%%
In the system of foliated charts associated with the (completely integrable) distribution $\mathcal{V} \,=\, \mathrm{ker}\,P$, decomposition \eqref{Eq: decomposition 1-forms} for a differential $1$-form
\be
\alpha \,=\, \alpha_a \dd x^a + \alpha_A \dd z^A \,,
\ee
reads
\be
\begin{split}
\alpha^\parallel_P \,&=\, \alpha \circ P \,=\, \alpha_A P^A \,, \\
\alpha^\perp_P \,&=\, \alpha - \alpha^\parallel_P \,=\, \left(\,\alpha_a + P^A_a \alpha_A \,\right) \dd x^a \,.
\end{split}
\ee
%%%

\noindent Note that, by definition, $\alpha^\parallel_P$ belongs to the image of the $(1,1)$-tensor $P$ when it acts upon differential forms.
%%%
We will denote by $P^*$ such action.
%%%
On the other hand, $\alpha^\perp_P$ belongs to the kernel of $P^*$.
%%%
Thus, the discussion above may be resumed by saying that the cotangent bundle $\mathbf{T}^* M$ splits as the direct sum of two subbundles, say
\be
\begin{split}
\mathcal{V}^* \,&:=\, \mathrm{Im}\,P^* \,\subset \, \mathbf{T}^* M \,, \\
\mathcal{H}^* \,&:=\, \mathrm{Ker}\,P^* \,\subset \, \mathbf{T}^* M \,,
\end{split}
\ee
such that 
\be \label{Eq: decomposition T*M}
\mathbf{T}^* M \,=\, \mathcal{V}^* \oplus \mathcal{H}^* \,.
\ee
%%%
It is straightforward to check that $\mathcal{V}^*_m$ can be identified with the dual vector space of $\mathcal{V}_m$ at each $m \in M$, as well as $\mathcal{H}^*_m$ can be identified with the dual vector space of $\mathcal{H}_m$ for each $m \in M$.
%%%
\end{proof}
\end{proposition}

\begin{remark}
Also in this case the opposite $(1,1)$-tensor $R \,=\, \mathds{1} - P$, provides the opposite splitting
\be
\mathbf{T}^* M \,=\, \mathcal{H}^* \oplus \mathcal{V}^* \,,
\ee
since
\be
\begin{split}
\mathcal{V}^* \,&=\, \mathrm{Ker} \, R^* \,=\, \mathrm{Im}\,P^* \,, \\
\mathcal{H}^* \,&=\, \mathrm{Im} \, R^* \,=\, \mathrm{Ker}\,P^* \,.
\end{split}
\ee
%%%
\end{remark}

\begin{proposition}[\textsc{Splitting of the bundle of $k$-forms}]
Any almost product structure $(\mathcal{V},\, \mathcal{H})$ provided by the $(1,1)$-tensor $P$ over a smooth manifold $M$, provides a splitting of the bundle of $k$-forms $\Lambda^k(M)$ of the type
\be
\Lambda^k(M) \,=\, {\Lambda^k}^\parallel_P(M) \oplus {\Lambda^k}^\perp_P(M) \,,
\ee  
for ${\Lambda^k}^\parallel_P(M)$ and ${\Lambda^k}^\perp_P(M)$ two smooth subbundles of $\Lambda^k(M)$.
%%%
\begin{proof}
Decomposition \eqref{Eq: decomposition 1-forms} holds for any $k$-form on $M$, where
\be
\begin{split}
\alpha^\parallel_P \,&:=\, \alpha(P(\,\cdot\,),...,  ,P(\,\cdot\,)) \,, \\
\alpha^\perp_P \,&:=\, \alpha - \alpha^\parallel_P \,.
\end{split}
\ee
%%%
To look at the local expression of such decomposition, consider a differential $k$-form $\alpha$ expressed in the basis
\be
\left\{\, \dd x^a,\, P^A \,\right\}_{a=1,...,l;\,A=1,...,r} \,,
\ee
say
\be
\begin{split}
\alpha \,=\, &\alpha_{A_1...A_k} P^{A_1} \wedge ... \wedge P^{A_k} + \\
&\alpha_{a_1 A_2 ... A_k} \dd x^{a_1} \wedge P^{A_2} \wedge ... \wedge P^{A_k} + \\ 
&\;\; \vdots \\
&\alpha_{a_1 ... a_k} \dd x^{a_1} \wedge ... \wedge \dd x^{a_k} \,.
\end{split}
\ee
%%%
A straightforward computation shows that 
\be
\alpha^\parallel_P \,=\, \alpha_{A_1...A_k} P^{A_1} \wedge ... \wedge P^{A_k} \,,
\ee
and
\be
\alpha^\perp_P \,=\, \alpha_{a_1 A_2 ... A_k} \dd x^{a_1} \wedge P^{A_2} \wedge ... \wedge P^{A_k} + ... + \alpha_{a_1 ... a_k} \dd x^{a_1} \wedge ... \wedge \dd x^{a_k} \,.
\ee
%%%
We will denote the splitting of the bundle $\Lambda^k(M)$ induced by such decomposition of $k$-forms as
\be \label{Eq: decomposition bundle of k-forms P}
\Lambda^k(M) \,=\, {\Lambda^k}^\parallel_P(M) \oplus {\Lambda^k}^\perp_P(M) \,.
\ee
%%%  
\end{proof}
\end{proposition}

\begin{remark}
In this case, the splitting induced by the opposite almost product structure $R \,=\, \mathds{1} - P$ does not coincide with the opposite of \eqref{Eq: decomposition bundle of k-forms P}.
%%%
Indeed, a direct computation shows that
\be
\alpha^\parallel_R \,=\, \alpha_{a_1 ... a_k} \dd x^{a_1} \wedge ... \wedge \dd x^{a_k} \,,
\ee
while
\be
\begin{split}
\alpha^\perp_R \,=\, &\alpha_{A_1...A_k} P^{A_1} \wedge ... \wedge P^{A_k} + \\
&\alpha_{a_1 A_2 ... A_k} \dd x^{a_1} \wedge P^{A_2} \wedge ... \wedge P^{A_k} + \\
&\;\;\vdots \\
&\alpha_{a_1...a_{k-1}A_k} \dd x^{a_1} \wedge... \wedge \dd x^{a_{k-1}} \wedge P^{A_k} \,.
\end{split}
\ee
%%%
We will denote the splitting of the bundle $\Lambda^k(M)$ induced by such decomposition of $k$-forms as
\be \label{Eq: decomposition bundle of k-forms R}
\Lambda^k(M) \,=\, {\Lambda^k}^\parallel_R(M) \oplus {\Lambda^k}^\perp_R(M) \,.
\ee
%%%
\end{remark}

\begin{remark}
\label{remark:bigrading}
Generally, given an almost product structure on a manifold $M$, $\T M  = \mathcal{V} \oplus \mathcal{H}$, the direct sum decomposition proved in \cref{prop:split_cotangent_bundle}, $\T^\ast M = \mathcal{V}^\ast \oplus \mathcal{H}^\ast$ gives a more general decomposition
\[
\Lambda^k (M) = \bigoplus_{j = 0}^{k} \,\Lambda^j \mathcal{V}^\ast \otimes \Lambda^{k-j} \mathcal{H}^\ast\,\,,
\]
where $\Lambda^j \mathcal{V}^\ast$ denotes $j$-times the wedge product of $\mathcal{V}^\ast$ with itself.
%%%
In terms of this decomposition, we may consider the following identifications: 
\[
\Lambda^{k\parallel}_P(M) = \Lambda^k \mathcal{V}^\ast\,, \qquad \Lambda^{k\perp}_P(M) = \bigoplus_{j = 0}^{k-1} \Lambda^j \mathcal{V}^\ast \otimes \Lambda^{k-j} \mathcal{H}^\ast.\]
\end{remark}

\begin{definition}[\textsc{Nijenhuis tensor of an almost product structure}]
Let $P$ be a $(1,1)$-tensor field on a smooth manifold $M$ defining an almost product structure.
%%%
The \textbf{Nijenhuis tensor} associated to $P$ is the $(1,2)$-tensor $\mathcal{N}_P$ defined by
\be
\mathcal{N}_P(X, Y) := [P(X), P(Y)] - P \big(\, [P(X), Y] + [X, P(Y)] - P([X, Y])\, \big),
\ee
for all vector fields $X, Y \in \mathfrak{X}(M)$.
%%%

\noindent Using the foliated charts associated with the distribution $\mathcal{V} \,=\, \mathrm{Im}\,P$, whose coordinates read $(x^a,\,z^A)$, the Nijenhuis tensor $\mathcal{N}_P$ reads
\be
\mathcal{N}_P \,=\, \left(\,\frac{\partial P^A_b}{\partial x^a} - \frac{\partial P^A_a}{\partial x^b} + P^B_a \frac{\partial P^A_b}{\partial z^B} - P^B_b \frac{\partial P^A_a}{\partial z^B} \,\right) \, \dd x^a \wedge \dd x^b \otimes \frac{\de}{\de z^A} + \frac{\partial P^A_a}{\partial z^B} \, \dd x^a \wedge P^B \otimes \frac{\de}{\de z^A}
\ee
\end{definition}

\begin{definition}[\textsc{Integrable Almost Product Structure}]
An almost product structure $P$ is said to be \textbf{integrable} if the two complementary distributions it defines are both involutive. 
%%%
By the Frobenius Theorem, this is equivalent to the existence of local coordinate charts, called \textbf{foliated charts}, adapted to both distributions simultaneously. 
%%%
\end{definition}

\noindent There is an interesting characterization in terms of the Nijenhuis tensor. 

\begin{theorem}
An almost product structure is integrable if and only if its Nijenhuis tensor vanishes.
\begin{proof}
See \cite{dLR89}   
\end{proof}
\end{theorem}

\section{Pre-symplectic manifolds}
\label{Sec: Pre-symplectic manifolds}

\begin{definition}[\textsc{Pre-symplectic manifold}]
\label{Def: Pre-symplectic manifold}
A \textbf{pre-symplectic manifold} is a pair $(M, \omega)$ where $M$ is a smooth $d$-dimensional manifold, and $\omega \in \Omega^2(M)$ is a closed differential $2$-form.
\end{definition}

\begin{remark}
We will always assume the form $\omega$ to have constant rank, i.e., the dimension of $\mathrm{Im}(\omega^\flat_m) \subset T^*_m M$ to be constant for all $m \in M$, where 
\be
\omega^\flat  \colon  \mathbf{T}M \to  \mathbf{T}^* M  \colon  X \mapsto \omega^\flat(X)(\cdot) := \omega(X, \cdot) \,.
\ee
\end{remark}

\begin{definition}
The distribution
\be
\mathcal{V} := \ker \omega \subset \mathbf{T}M
\ee
is called the \textbf{characteristic distribution} of $\omega$.
\end{definition}

\begin{remark}[\textsc{Symplectic manifold}]
When $\mathrm{dim}\,\mathrm{Im}(\omega^\flat_m) \,=\, d \,, \,\, \forall\,\,m\in M$, the map $\omega^\flat$ is a vector bundle isomorphism between $\mathbf{T}M$ and $\mathbf{T}^* M$ whose inverse reads
\be
\omega^\sharp  \colon  \mathbf{T}^* M \to \mathbf{T}M  \colon  \alpha \mapsto \omega^\sharp(\alpha) \;\text{ s.t. } \omega(\omega^\sharp(\alpha), X) \,=\, \alpha(X) \,, \quad \forall\,\alpha \in \Omega^1(M),\, X \in \mathfrak{X}(M) \,,
\ee
and the pair $(M,\,\omega)$ is referred to as a \textbf{symplectic manifold}.
%%%
Note that symplectic manifolds are always even-dimensional.
\end{remark}

\begin{theorem}[\textsc{Darboux theorem for symplectic manifolds}]
\label{Thm: Darboux symplectic}
Let $(M, \omega)$ be a $2n$-dimensional symplectic manifold. 
%%%
Then, for every point $m \in M$, there exists a coordinate chart $(U, \varphi)$ centered at $m$, with coordinates $(q^1, \dots, q^n, p_1, \dots, p_n)$, such that the symplectic form $\omega$ reads
\be
\omega|_U = \dd q^a \wedge \dd p_a \,.
\ee
%%%
Such coordinates are called \textbf{Darboux coordinates} for the symplectic form $\omega$ around the point $m$.   
%%%
\begin{proof}
See \cite{godbillon1969geometrie,abraham,libermarle}.
\end{proof}
\end{theorem}

\begin{theorem}[\textsc{Darboux theorem for pre-symplectic manifolds}]
\label{Thm: Darboux pre-symplectic}
Let $(M, \omega)$ be a $d$-dimensional pre-symplectic manifold, with $\omega \in \Omega^2(M)$ a closed 2-form of constant rank $2r < d$.
%%%
Then, for every point $m \in M$, there exists a coordinate chart $(U, \varphi)$ around $m$, with coordinates $(q^1, \dots, q^r, p_1, \dots, p_r, z^1, \dots, z^{d - 2r})$, such that the 2-form $\omega$ reads
\be
\omega|_U = \dd q^a \wedge \dd p_a \,.
\ee
%%%
Again, such coordinates are called \textbf{Darboux coordinates} for the pre-symplectic form $\omega$ around the point $m$.
%%%
In the system of Darboux coordinates, the characteristic distribution of $\omega$ is spanned by the vector fields
\be
\mathcal{V} \,=\, \mathrm{ker}\,\omega \,=\, \operatorname{span}\,\left\{\, \frac{\de}{\de z^A} \,\right\}_{A=1,...,d-2r} \,.
\ee
%%%
Thus, Darboux coordinates for the pre-symplectic form $\omega$ are foliated charts for the foliation associated with the characteristic distribution.
%%%
\begin{proof}
See \cite{godbillon1969geometrie}.
\end{proof}
\end{theorem}

\subsection{Coisotropic submanifolds}

\begin{definition}[\textsc{Symplectic ortogonal}]
Let $(M, \omega)$ be a symplectic manifold, and let $W_m \subseteq \mathbf{T}_m M$ be a linear subspace of the tangent space at a point $m \in M$.
%%%
The \textbf{symplectic orthogonal} of $W_m$ with respect to $\omega$ is the subspace
\be
W^{\perp_\omega}_m \,:=\, \left\{\, X \in \mathbf{T}_m M  \colon  \omega(X, Y) = 0 \quad \forall\, Y \in W \,\right\}.
\ee    
\end{definition}

\begin{definition}[\textsc{Coisotropic submanifold of a symplectic manifold}]
Let $(M, \omega)$ be a symplectic manifold and let \be
\mathfrak{i}  \colon  N \hookrightarrow M
\ee 
be an immersed submanifold.
%%%
We say that $N$ is a \textbf{coisotropic submanifold} of $(M, \omega)$ if, for every point $n \in N$, the symplectic orthogonal of the tangent space $\mathbf{T}_n N$ satisfies
\be
\left( \mathbf{T}_n N \right)^{\perp_\omega} \,\subseteq\, \mathbf{T}_n N \,.
\ee 
\end{definition}

\begin{remark} 
\label{remark:symplectic_minimal_dimension}
Coisotropic submanifolds appear naturally when we are dealing with the question of finding an embedding of a pre-symplectic manifold $(M, \omega)$ into a symplectic manifold $(\widetilde M, \widetilde \omega)$ of minimal dimension. Indeed, if we denote by $\mathcal{V}|_x$ the characteristic distribution of $\omega$, for every $x \in M$, we have $\mathcal{V} \subseteq (\T_x M)^{\perp}\subseteq \T_x \widetilde M$. Such an embedding has minimal dimension only when $\mathcal{V} = (\T_x M)^{\perp}$, which is precisely the coisotropicity condition.
\end{remark}

\subsection{Coisotropic embeddings}

\subsubsection{Existence}

\begin{theorem}[\textsc{Symplectic thickenings for pre-symplectic manifolds}]
Let $(M,\,\omega)$ be a pre-symplectic manifold.
%%%
There always exists a symplectic manifold $(\widetilde{M},\,\widetilde{\omega})$ and an embedding
\be
\mathfrak{i} \colon  M \hookrightarrow \widetilde{M} \,,
\ee
such that $\mathfrak{i}(M)$ is a closed coisotropic submanifold of $\widetilde{M}$.
%%%
The symplectic manifold $(\widetilde{M},\,\widetilde{\omega})$ is referred to as \textbf{symplectic thickening} of $(M,\,\omega)$.
%%%

\begin{proof}
Let $\mathcal{V}$ be the characteristic distribution on $M$.
%%%
Consider a complementary distribution $\mathcal{H}$ providing an almost product structure on $M$ associated with the $(1,1)$-tensor $P$ on $M$ that, in the system of Darboux coordinates on $M$ as in \cref{Thm: Darboux pre-symplectic}, reads
\be
P \,=\, \left(\,\dd z^A - {P_q}^A_a \dd q^a - {P_p}^{Aa} \dd p_a\,\right) \otimes \frac{\de}{\de z^A} \,.
\ee
%%%
The opposite $(1,1)$-tensor $R$ reads
\be
R \,=\, \mathds{1} - P \,=\, \dd q^a \otimes \left(\, \frac{\de}{\de q^a} + {P_q}^A_a \frac{\de}{\de z^A} \,\right) + \dd p_a \otimes \left(\, 
\frac{\de}{\de p_a} + {P_p}^{Aa} \frac{\de}{\de z^A} \,\right) \,.
\ee
%%%
Consider the bundle ${\Lambda^1}^\perp_R(M)$ over $M$.
%%%
Denote by $\tau$ the projection map onto $M$.
%%%
Its sections are differential $1$-form that, in the system of Darboux coordinates, looks like
\be
\alpha \,=\, \alpha_A P^A \,.
\ee
%%%
Thus, a system of adapted coordinates on ${\Lambda^1}^\perp_R(M)$, reads
\be
\left\{\, q^a,\,p_a,\,z^A,\, \mu_A \,\right\}_{\substack{a=1,...,r \\ A=1,...,l}} \,,
\ee
with $l=d-2r$.
%%%
The projection map $\tau$ thus reads
\be
\tau  \colon  (q^a,\,p_a,\,z^A,\, \mu_A) \mapsto (q^a,\,p_a,\,z^A) \,.
\ee
%%%
The bundle ${\Lambda^1}^\perp_R(M)$ has a distinguished $1$-form, which is the obvious analogous of the tautological $1$-form on $\mathbf{T}^* M$, defined as
\be
\vartheta^P_{\alpha}(X) \,=\, \alpha(P\circ\tau_*(X)) \,, \;\; \forall\,\, X \in \mathbf{T}_\alpha M \,,
\ee
where $\alpha$ has to be considered as a point in ${\Lambda^1}^\perp_R(M)$ on the left hand side, and as a differential $1$-form on $M$ on the right hand side.
%%%
Since the bundle ${\Lambda^1}^\perp_R(M)$ depends on the chosen almost product structure $P$, we stressed the dependence of $\vartheta$ on $P$.
%%%
In the system of Darboux coordinates, $\vartheta^P$ reads
\be
\vartheta^P \,=\, \mu_A P^A \,.
\ee
%%%
Consider the differential $2$-form
\be
\widetilde{\omega} \,=\, \tau^* \omega + \dd \vartheta^P \,,
\ee
on ${\Lambda^1}^\perp_R(M)$.
%%%
It is closed by definition.
%%%
On the other hand, it is non-degenerate in a tubular neighborhood of the zero section of $\tau$.
%%%
Indeed, consider a vector field $X$ written in the basis
\be
\left\{\, {H_q}_a \,:=\,\frac{\de}{\de q^a} + {P_q}^A_a \frac{\de}{\de z^A},\, {H_p}^a\,:=\,\frac{\de}{\de p_a} + {P_p}^{Aa} \frac{\de}{\de z^A}, \, \frac{\de}{\de z^A},\, \frac{\de}{\de \mu_A} \,\right\}_{\substack{a=1,...,r \\ A=1,...,l}} \,,
\ee
as
\be
X \,=\, {X_q}^a \,{H_q}_a + {X_p}_a \,{H_p}^a + {X_z}^A \, \frac{\de}{\de z^A} + {X_\mu}_A\, \frac{\de}{\de \mu_A} \,.
\ee
%%%
The contraction $i_X \widetilde{\omega}$ reads
\be
i_X \widetilde{\omega} = F_a\,dq^a + G^a\,dp_a + H_A\,P^A + K^A\,d\mu_A,
\ee
where:
\begin{align}
F_a = -{X_p}_a
+ \mu_A \Biggl[
{X_q}^b\,&\left(\,\frac{\de {P_q}^A_b}{\de q^a} - \frac{\de {P_q}^A_a}{\de q^b} + {P_q}^B_a \frac{\de {P_q}^A_b}{\de z^B} - {P_q}^B_b \frac{\de {P_q}^A_a}{\de z^B} \,\right)
+ \nonumber \\
{X_p}_b\,&\left(\,\frac{\de {P_p}^{Ab}}{\de q^a} - \frac{\de {P_q}^A_a}{\de p_b} + {P_q}^B_a \frac{\de {P_p}^{Ab}}{\de z^B} - {P_p}^{Bb} \frac{\de {P_q}^A_a}{\de z^B}\,\right) - \nonumber \\
{X_z}^B &\frac{\de {P_q}^A_a}{\de z^B}
\Biggr], \\
G^a = {X_q}^a
+ \mu_A \Biggl[\,
{X_q}^b\,&\left(\,\frac{\de {P_q}^A_b}{\de p_a} - \frac{\de {P_p}^{Aa}}{\de q^b} + {P_p}^{Ba} \frac{\de {P_q}^A_b}{\de z^B} - {P_q}^B_b \frac{\de {P_p}^{Aa}}{\de z^B} \,\right) + \nonumber \\
{X_p}_b\,&\left(\,\frac{\de {P_p}^{Ab}}{\de p_a} - \frac{\de {P_p}^{Aa}}{\de p_b} + {P_p}^{Ba} \frac{\de {P_p}^{Ab}}{\de z^B} - {P_p}^{Bb} \frac{\de {P_p}^{Aa}}{\de z^B} \,\right) - \nonumber \\
{X_z}^B & \frac{\de {P_p}^{Aa}}{\de z^B}
\Biggr], \\
H_A = {X_z}_A
+ \mu_B  \Biggl[
{X_q}^b\, & \frac{\de {P_q}^B_b}{\de z^A} + {X_p}_b\,\frac{\de {P_p}^{Bb}}{\de z^A}
\Biggr], \\
K^A = -{X_\mu}^A \,. \hspace{1,1cm} &
\end{align}
%%%
With this in mind, one can easily prove that 
\be
i_X \widetilde{\omega} \,=\, 0 \,\implies\, X \,=\, 0 \,,
\ee
if and only if, either 
\be
\mu_A \,= \, 0 \,,
\ee
or
\begin{align}
\left(\,\frac{\de {P_q}^A_b}{\de q^a} - \frac{\de {P_q}^A_a}{\de q^b} + {P_q}^B_a \frac{\de {P_q}^A_b}{\de z^B} - {P_q}^B_b \frac{\de {P_q}^A_a}{\de z^B} \,\right) \,&=\, 0 \\
\left(\,\frac{\de {P_p}^{Ab}}{\de q^a} - \frac{\de {P_q}^A_a}{\de p_b} + {P_q}^B_a \frac{\de {P_p}^{Ab}}{\de z^B} - {P_p}^{Bb} \frac{\de {P_q}^A_a}{\de z^B}\,\right) \,&=\,0 \\
\left(\,\frac{\de {P_q}^A_b}{\de p_a} - \frac{\de {P_p}^{Aa}}{\de q^b} + {P_p}^{Ba} \frac{\de {P_q}^A_b}{\de z^B} - {P_q}^B_b \frac{\de {P_p}^{Aa}}{\de z^B} \,\right) \,&=\,0 \\
\left(\,\frac{\de {P_p}^{Ab}}{\de p_a} - \frac{\de {P_p}^{Aa}}{\de p_b} + {P_p}^{Ba} \frac{\de {P_p}^{Ab}}{\de z^B} - {P_p}^{Bb} \frac{\de {P_p}^{Aa}}{\de z^B} \,\right) \,&=\, 0 \\
\frac{\de {P_q}^A_b}{\de z^B} \,&=\,0 \,, \\
\frac{\de {P_p}^{Ab}}{\de z^B} \,&=\,0 \,.
\end{align}
for all $a,\,b \,=\, 1,...,r$ and for all $A,\, B \,=\, 1,...,l$.
%%%
This proves that $\widetilde{\omega}$ is non-degenerate in a tubular neighborhood of the zero section of $\tau$ (i.e. $\mu_A \,=\, 0$), or, if all the components of $\mathcal{N}_P$ vanish, on the whole ${\Lambda^1}^\perp_R(M)$.
%%%

\noindent Thus, the manifold $\widetilde{M}$ reads a tubular neighborhood of the zero section of the bundle ${\Lambda^1}^\perp_R$ (or the whole bundle ${\Lambda^1}^\perp_R$ if the almost product structure chosen has vanishing Nijenhuis tensor), the embedding map $\mathfrak{i}$ reads the zero section of $\tau$, and the symplectic structure $\widetilde{\omega}$ is
\be
\widetilde{\omega} \,=\, \tau^* \omega + \dd \vartheta^P \,.
\ee
%%%

\noindent To prove that $M$ is a coisotropic submanifold of $\widetilde{M}$, we shall prove that $\T_m M^\perp \subset T_m M$ for all $m \in \mathfrak{i}(M)$.
%%%
Elements of $\T_m M^\perp$ are tangent vectors to $m \in \mathfrak{i}(M)$
\be
X \,=\, {X_q}^a {H_q}_a\bigr|_{m} + {X_p}_a {H_p}^a\bigr|_{m} + {X_z}^A \frac{\de}{\de z^A}\bigr|_{m} + {X_\mu}^A \frac{\de}{\de \mu^A}\bigr|_{m} \,,
\ee
(now ${X_q}^a,\, {X_p}_a,\, {X_z}^A,$ and ${X_\mu}_A$ are real numbers) such that 
\be \label{Eq: coisotropicity pre-symplectic}
\widetilde{\omega}_m(X,\, W) \,=\, 0 \,,\;\;\; \forall \,\, W \in \T_m M \,, 
\ee
namely for any tangent vector $W$ to $M$ in $m$
\be
W \,=\, {W_q}^a {H_q}_a\bigr|_{m} + {W_p}_a {W_p}^a\bigr|_{m} + {W_z}^A \frac{\de}{\de z^A}\bigr|_{m} \,.
\ee
%%%
A straightforward computation shows that
\be
\widetilde{\omega}_m(X,\, W) \,=\, {X_q}^a {W_p}_a - {W_q}^a {X_p}_a + {X_\mu}_A {W_z}^A \,, 
\ee
and, thus, \eqref{Eq: coisotropicity pre-symplectic} implies 
\be
{X_q}^a \,=\, {X_p}_a \,=\, {X_\mu}_A \,=\, 0 \,,\;\;\; \forall\,\, a\,=\, 1,...,r;\,A\,=\,1,...,l \,.
\ee
%%%
Consequently $X \in \T_m M$, which proves that $\T_m M^\perp \subset \T_m M$.
%%%
\end{proof}
\end{theorem}

\subsubsection{Uniqueness}
Uniqueness of coisotropic embeddings of pre-symplectic manifolds of constant rank has been proved by \textit{M. J. Gotay} in \cite{gotay}. 
%%%

\noindent Let $(M,\omega)$ be a pre-symplectic manifold of constant rank and let $\mathfrak{i}_1\colon (M, \omega) \hookrightarrow (\widetilde M_1, \widetilde\omega_1)$, $\mathfrak{i}_2\colon (M, \omega) \hookrightarrow (\widetilde M_2, \widetilde\omega_2)$ be two different coisotropic embeddings. 
%%%

\noindent The proof of uniqueness can be split into two main steps: an algebraic one, and a geometric one.

\begin{itemize}
    \item Find a symplectic vector bundle isomorphism 
    $$\phi \colon  (\T \widetilde M_1 |_{M}, \widetilde\omega_1) \longrightarrow (\T \widetilde M_2 |_M, \widetilde\omega_2) \,,$$
    which is the identity when restricted to $\T M$.
\end{itemize}

\noindent This first step will be obtained by simple algebra on symplectic vector spaces, point-wise.
%%%

\noindent Then, such an isomorphism is integrated to a neighborhood diffeomorphism $\psi\colon U_1 \longrightarrow U_2$ where $M \subseteq U_1 \subseteq\widetilde M_1$ and $M \subseteq U_2 \subseteq \widetilde M_2$ are open neighborhoods of $M$ on $\widetilde M_1$ and $\widetilde M_2$, such that $\psi^\ast \widetilde\omega_2 |_M = \widetilde\omega_1$. 
%%%
Lastly, we apply \textit{Moser's trick}:

\begin{itemize}
    \item Look for a vector field $X_t$ vanishing on $M$ satisfying the hypotheses of \cref{thm:Moser_trick}.
\end{itemize}

%%%
This last step uses the Relative Poincaré Lemma (Theorem \ref{thm:Relative_Poincaré_Lemma}).

\noindent The condition that $X_t$ vanishes on $M$ means that its flow at time one, $\psi_1$, will be the identity on $M$, so that we have found a symplectomorphism of the two (possible smaller) neighborhoods $U_1$ and $U_2$ that is the identity over $M$.
%%%

\noindent Let us recall here the notion of \textit{neighborhood equivalence}:

\begin{definition}[\textsc{Neighborhood equivalence}]
Let $\mathfrak{i}_1\colon (M, \omega) \hookrightarrow (\widetilde M_1, \widetilde\omega_1)$ and $\mathfrak{i}_2\colon (M, \omega) \hookrightarrow (\widetilde M_2, \widetilde\omega_2)$ be two coisotropic embeddings. They are said to be \textbf{neighborhood equivalent} if there exist open neighborhoods $U_1$ of $\mathfrak{i}_1(M)$ in $\widetilde M_1$ and $U_2$ of $\mathfrak{i}_2(M)$ in $\widetilde M_2$, and a diffeomorphism $\psi\colon U_1 \longrightarrow U_2$ such that:
\begin{enumerate}
    \item $\psi$ is a symplectomorphism, i.e., $\psi^\ast \widetilde\omega_2 = \widetilde\omega_1$;
    \item $\psi$ is the identity on $M$, i.e., $\psi \circ \mathfrak{i}_1 = \mathfrak{i}_2$.
\end{enumerate}
\end{definition}

\noindent The two steps above are made precise in the following theorem.
%%%

\begin{theorem}[\textsc{Coisitropic embeddings uniqueness}] 
\label{Thm: uniqueness pre-symplectic}
Two different coisotropic embeddings of $(M, \omega)$ are neighborhood equivalent.
\end{theorem}

\begin{proof} 
Let $\mathcal{V}$ the characteristic distribution of $\omega$ on $M$, and let $\mathcal{H}$ be any complementary distribution so that $\T M = \mathcal{V} \oplus \mathcal{H}$. 
%%%
Let $\mathfrak{i}\colon (M, \omega) \hookrightarrow (\widetilde M, \widetilde\omega)$ be a coisotropic embedding. 
%%%
The distribution $\mathfrak{i}_\ast \mathcal{H}$ provides a symplectic vector subbundle of $(\T\widetilde M |_M, \widetilde\omega)$ since $\widetilde \omega|_{\mathcal{H}}$ is non degenerate, $\mathcal{H}$ being complementary to $\mathcal{V}$.
%%%

\noindent Consequently, $\T\widetilde M |_{M} = \mathcal{H} \oplus \mathcal{H}^{\perp}$ and $\mathcal{V} \subseteq \mathcal{H}^{\perp}$. 
%%%
\noindent On the other hand, $\mathcal{V}$ is a Lagrangian subbundle of $\mathcal{H}^\perp$. 
%%%
Indeed
\[
\mathrm{dim} \mathcal{H}^\perp \,=\, \underbrace{\mathrm{dim} \T \widetilde{M}}_{=\,2r + 2l} - \underbrace{\mathrm{dim} \mathcal{H}}_{=\,2r} \,=\, 2l \,, 
\]
where $r \,=\, \mathrm{rank} \mathcal{H}$ and $l \,=\, \mathrm{rank} \mathcal{V}$.
%%%
Therefore, we necessarily have $\mathcal{H}^{\perp} \cong \mathcal{V} \oplus \mathcal{V}^\ast$, so that 
\[
\T\widetilde M |_M \cong \mathcal{H} \oplus \mathcal{V} \oplus \mathcal{V}^\ast  = \T \mathcal{V}^\ast|_{M}\,.
\]
It is easy to check that all this chain of vector bundle isomorphisms preserve symplectic forms, where the symplectic form on $\mathcal{V}^\ast$ (which, following \cref{remark:bigrading}, can be identified with ${\Lambda^1}^\perp_R (M)$) is defined as in the previous section.
%%%

\noindent Following the steps defined at the beginning of this section, it only remains to deal with the case where $M$ is embedded in $U \subseteq \widetilde M$ and there are two symplectic forms $\widetilde\omega_1$, $\widetilde\omega_2$ such that $\widetilde\omega_1 = \widetilde\omega_2$ on $M$. 
%%%
Using the relative Poincaré Lemma, and shrinking $U$ further if necessary, we take a $1$-form, $\widetilde\theta$ vanishing on $M$ such that 
\[
\widetilde\omega_2 - \widetilde\omega_1 = \dd \widetilde\theta\,.
\]
Now, let us define the following family of forms 
\[
\widetilde\omega_t := t \widetilde\omega_2 + (1-t) \widetilde\omega_1\,,
\]
which is symplectic on $U$ (if necessary, we shrink this neighborhood further to achieve this). Now, define $X_t$ to be the unique time dependent vector field satisfying $i_{X_t} \widetilde\omega_t = -\widetilde\theta$. Then, 
\begin{align*}
    \widetilde\omega_2 - \widetilde\omega_1 +\Lie_{X_t} \widetilde\omega_t &= 
    \widetilde\omega_2 - \widetilde\omega_1 + \dd i_{X_t} \widetilde\omega_t = 0\,,
\end{align*}
by definition. 
%%%
This leaves us in the hypotheses of \cref{thm:Moser_trick}. Furthermore, $X_t$ vanishes on $M$, since $\widetilde\theta$ does, so that the time-one flow of $X_t$, say $\psi_1$, is the identity on $M$, and it defines a symplectomorphism between $(U,\, \widetilde{\omega}_1)$ and $(U,\, \widetilde{\omega}_2)$, proving the result.
\end{proof}

\section{Pre-cosymplectic manifolds}
\label{Sec: cosymplectic manifolds}

\begin{definition}[\textsc{Pre-cosymplectic manifold}]
A \textbf{pre-cosymplectic manifold} is a triple $(M, \eta, \omega)$, where $M$ is a smooth $(2d+1)$-dimensional manifold, $\eta \in \Omega^1(M)$ is a closed differential $1$-form, and $\omega \in \Omega^2(M)$ is a closed differential 2-form.
\end{definition}

\begin{remark}
We will always assume both $\eta$ and $\omega$ to have constant rank, i.e., $\eta$ to be nowhere vanishing and the dimension of $\mathrm{Im}(\omega^\flat_m) \subset T^*_m M$ to be constant for all $m \in M$, where 
\be
\omega^\flat  \colon  \mathbf{T}M \to  \mathbf{T}^* M  \colon  X \mapsto \omega^\flat(X)(\cdot) := \omega(X, \cdot) \,.
\ee
%%%
Moreover, we will always assume $\mathrm{ker}\,\eta \cap \mathrm{ker}\,\omega$ to have constant dimension, and, thus, to define an integrable distribution amounting to a subbundle of $\mathbf{T}M$.
\end{remark}

\begin{definition}[\textsc{Characteristic distribution}]
The distribution
\be
\mathcal{V} \,:=\, \mathrm{ker}\,\eta \cap \mathrm{ker}\, \omega \,,
\ee
is called the \textbf{characteristic distribution} of $(\eta,\, \omega)$.
\end{definition}

\begin{definition}[\textsc{Reeb vector field}]
Let $(M, \eta, \omega)$ be a pre-cosymplectic manifold. 
%%%
A \textbf{Reeb vector field} is a vector field $R \in \mathfrak{X}(M)$ satisfying the conditions
\be
i_R \omega = 0 \,, \qquad i_R \eta = 1 \,.
\ee
\end{definition}

\begin{remark}[\textsc{Cosymplectic manifold}]
When $\eta \wedge \omega^d$ is nowhere vanishing on $M$, i.e., it is a volume form, the pair $(\eta, \omega)$ defines a \textbf{cosymplectic structure} on $M$, and the triple $(M, \eta, \omega)$ is referred to as a \textbf{cosymplectic manifold}. Then, there is a unique Reeb vector field.
%%%
\end{remark}

\begin{theorem}[\textsc{Darboux theorem for cosymplectic manifolds}]
Let $(M, \eta, \omega)$ be a $(2d+1)$-dimensional cosymplectic manifold. 
%%%
Then, for every point $m \in M$, there exists a coordinate chart $(U, \varphi)$ centered at $m$, with coordinates $(t,\, q^1, \dots, q^d,\, p_1, \dots, p_d)$, such that:
\[
\eta|_U = \dd t \,, \qquad \omega|_U = \dd q^a \wedge \dd p_a \,.
\]
%%%
These coordinates are called \textbf{Darboux coordinates} for the cosymplectic structure.
%%%
\begin{proof}
See \cite{godbillon1969geometrie}
\end{proof}
\end{theorem}

\begin{theorem}[\textsc{Darboux theorem for pre-cosymplectic manifolds}]
Let $(M, \eta, \omega)$ be a pre-cosymplectic manifold of dimension $2d + 1$, where $\eta \in \Omega^1(M)$ and $\omega \in \Omega^2(M)$ are closed forms with constant rank, and such that $\eta \wedge \omega^r \neq 0$ for some $r \leq d$.
%%%
Then, for every point $m \in M$, there exists a coordinate chart $(U, \varphi)$ centered at $m$, with coordinates $(t,\, q^1, \dots, q^r,\, p_1, \dots, p_r,\, z^1, \dots, z^{2(d - r)})$, such that:
\be
\eta|_U = \dd t \,, \qquad \omega|_U = \dd q^a \wedge \dd p_a \,.
\ee
%%%
Again, such coordinates are called \textbf{Darboux coordinates} for the pre-cosymplectic structure $(\eta,\,\omega)$ around the point $m$.
%%%
In the system of Darboux coordinates, the characteristic distribution is spanned by the vector fields
\be
\mathcal{V} \,=\,  \mathrm{ker} \,\eta \cap \mathrm{ker}\, \omega \,=\, \operatorname{span}\,\left\{\, \frac{\de}{\de z^A} \,\right\}_{A=1,...,2(d-r)} \,.
\ee
%%%
Thus, Darboux coordinates for the pre-cosymplectic structure $(\eta,\,\omega)$ are foliated charts for the foliation associated with the characteristic distribution.
%%%
\begin{proof}
See \cite{godbillon1969geometrie,chinea}.
\end{proof}
\end{theorem}

\subsection{Coisotropic submanifolds}

\begin{definition}[\textsc{Cosymplectic orthogonal}]
\label{Def: cosymplectic ortogonal}
Let $(M, \eta, \omega)$ be a cosymplectic manifold, and let $W_m \subseteq \mathbf{T}_m M$ be a linear subspace of the tangent space at a point $m \in M$.
%%%
The \textbf{cosymplectic orthogonal} of $W_m$ with respect to $(\eta, \omega)$ is the subspace
\be
W^{\perp_{(\eta, \omega)}}_m \,:=\, \left\{\, X \in \mathbf{T}_m M  \colon  \eta(X) = 0 \;\text{ and }\; \omega(X, Y) = 0 \quad \forall\, Y \in W_m \,\right\}.
\ee    
\end{definition}

\begin{remark}
The orthogonality condition is imposed only with respect to the restriction of $\omega$ to $\ker \eta_m$, where $\omega$ defines a symplectic structure.
%%%
Outside of $\ker \eta$, $\omega$ is non-degenerate; this motivates \cref{Def: cosymplectic ortogonal}.
\end{remark}

\begin{definition}[\textsc{Coisotropic submanifold of a cosymplectic manifold}]
Let $(M, \eta, \omega)$ be a cosymplectic manifold, and let 
\[
\mathfrak{i}  \colon  N \hookrightarrow M
\]
be an immersed submanifold.
%%%
We say that $N$ is a \textbf{coisotropic submanifold} of $(M, \eta, \omega)$ if, for every point $n \in N$, the cosymplectic orthogonal of the subspace $\mathbf{T}_n N$ satisfies
\be
{\mathbf{T}_n N}^{\perp_{(\eta, \omega)}} \,\subseteq\, \mathbf{T}_n N \,.
\ee
\end{definition}

\begin{remark}
    Together with a cosymplectic manifold, there is an induced Poisson structure defined by the corresponding Poisson brackets (see \cite{dLR89}). Furthermore, every Poisson manifold also has the notion of orthogonal (see \cite{vai1994}). It is easy to see that this notion coincides with the one presented in this paper.
\end{remark}

\subsection{Coisotropic embeddings}

\subsubsection{Existence}
\label{Subsubsec: Existence co-symplectic}

\begin{theorem}[\textsc{Cosymplectic thickenings for pre-cosymplectic manifolds}] \label{Thm: cosymplectic tickening}
Let $(M,\,\eta,\, \omega)$ be a pre-cosymplectic manifold.
%%%
There exists a cosymplectic manifold $(\widetilde{M},\,\widetilde{\eta},\, \widetilde{\omega})$ and an embedding
\[
\mathfrak{i} \colon  M \hookrightarrow \widetilde{M}
\]
such that $\mathfrak{i}(M)$ is a coisotropic submanifold of $\widetilde{M}$.
%%%
The cosymplectic manifold $(\widetilde{M},\,\widetilde{\eta},\, \widetilde{\omega})$ is referred to as a \textbf{cosymplectic thickening} of $(M,\,\eta,\, \omega)$.
%%%
\end{theorem}
%%%

\begin{proof}
Consider the characteristic distribution $\mathcal{V}$ on $M$.
%%%
Consider a complementary distribution $\mathcal{H}$ providing an almost product structure on $M$ associated with the $(1,\,1)$-tensor $P$ on $M$ that, in the system of Darboux coordinates on $M$, reads
\be \label{Eq: almost product structure cosymplectic}
P \,=\, \left(\,\dd z^A - {P_t}^A \dd t - {P_q}^A_a \dd q^a - {P_p}^{Aa} \dd p_a\,\right) \otimes \frac{\de}{\de z^A} \,=:\, P^A \otimes \frac{\de}{\de z^A} \,.
\ee
%%%
The opposite $(1,\,1)$-tensor $R$ reads
\be
R \,=\, \mathds{1} - P \,=\, \dd t \otimes \left(\, \frac{\de}{\de t} + {P_t}^A \frac{\de}{\de z^A} \,\right) + \dd q^a \otimes \left(\, \frac{\de}{\de q^a} + {P_q}^A_a \frac{\de}{\de z^A} \,\right) + \dd p_a \otimes \left(\, 
\frac{\de}{\de p_a} + {P_p}^{Aa} \frac{\de}{\de z^A} \,\right) \,.
\ee
%%%
Consider the bundle ${\Lambda^1}^\perp_R(M)$ over $M$.
%%%
Denote by $\tau$ the projection map onto $M$.
%%%
Its sections are differential $1$-forms that, in the system of Darboux coordinates look like
\be
\alpha \,=\, \alpha_A P^A \,.
\ee
%%%
Thus, a system of adapted coordinates on ${\Lambda^1}^\perp_R(M)$, reads
\be
\left\{\, t,\, q^a,\,p_a,\,z^A,\, \mu_A \,\right\}_{\substack{a=1,...,r \\ A=1,...,l}} \,,
\ee
with $l=d-2r$.
%%%
The projection map $\tau$ thus reads
\be
\tau  \colon  (t,\,q^a,\,p_a,\,z^A,\, \mu_A) \mapsto (t,\,q^a,\,p_a,\,z^A) \,.
\ee
%%%
The bundle ${\Lambda^1}^\perp_R(M)$ has a distinguished $1$-form, which is the obvious analogous of the tautological $1$-form on $\mathbf{T}^* M$, defined as
\be
\vartheta^P_{\alpha}(X) \,=\, \alpha(\tau_*(X)) \,, \;\; \forall\,\, X \in \mathbf{T}_\alpha M \,,
\ee
where $\alpha$ has to be considered as a point in ${\Lambda^1}^\perp_R(M)$ on the left hand side, and as a differential $1$-form on $M$ on the right hand side.
%%%
In the system of Darboux coordinates $\vartheta^P$ reads
\be \label{Eq: thetaP cosymplectic}
\vartheta^P \,=\, \mu_A P^A \,.
\ee
%%%
Consider the differential $2$-form
\be
\widetilde{\omega} \,=\, \tau^* \omega + \dd \vartheta^P \,,
\ee
on ${\Lambda^1}^\perp_R(M)$.
%%%
It is closed by definition.
%%%
On the other hand, $(\,(\tau^* \eta),\, \widetilde{\omega}^d\,)$ form a cosymplectic structure on a tubular neighborhood of the zero section of $\tau$.
%%%
Indeed, consider a vector field $X$ written in the basis
\be
\left\{\, H_t \,:=\, \frac{\de}{\de t} + {P_t}^A \frac{\de}{\de z^A},\, {H_q}_a \,:=\,\frac{\de}{\de q^a} + {P_q}^A_a \frac{\de}{\de z^A},\, {H_p}^a\,:=\,\frac{\de}{\de p_a} + {P_p}^{Aa} \frac{\de}{\de z^A}, \, \frac{\de}{\de z^A},\, \frac{\de}{\de \mu_A} \,\right\}_{\substack{a=1,...,r \\ A=1,...,l}} \,,
\ee
as
\be
X \,=\, X_t \, H_t + {X_q}^a \,{H_q}_a + {X_p}_a \,{H_p}^a + {X_z}^A \, \frac{\de}{\de z^A} + {X_\mu}_A\, \frac{\de}{\de \mu_A} \,.
\ee
%%%
The contraction $i_X \widetilde{\omega}$ reads
\be
i_X \widetilde{\omega} = E\, \dd t + F_a\,\dd q^a + G^a\,\dd p_a + H_A\,P^A + K^A\,\dd \mu_A,
\ee
where:
\begin{align}
E \,=\, \mu_A \Biggl[\, {X_q}^a &\left(\, \frac{\de {P_q}^A_a}{\de t} - \frac{\de {P_t}^A}{\de q^a} + {P_t}^B \frac{\de {P_q}^A_a}{\de z^B} - {P_q}^B_a \frac{\de {P_t}^A}{\de z^B} \,\right) + \nonumber \\ 
{X_p}_a &\left(\, \frac{\de {P_p}^{Aa}}{\de t} - \frac{\de {P_t}^A}{\de p_a} + {P_t}^B \frac{\de {P_p}^{Aa}}{\de z^B} - {P_p}^{Ba} \frac{\de {P_t}^A}{\de z^B} \,\right) + \nonumber \\
{X_z}^B & \frac{\de {P_t}^A}{\de z^B}\,\Biggr] \,, \\
F_a = -{X_p}_a
+ \mu_A \Biggl[ X_t\,&\left(\, \frac{\de {P_t}^A}{\de q^a} - \frac{\de {P_q}^A_a}{\de t} + {P_q}^B_a \frac{\de {P_t}^A}{\de z^B} - {P_t}^B \frac{\de {P_q}^A_a}{\de z^B} \,\right) + \nonumber \\
{X_q}^b\,&\left(\,\frac{\de {P_q}^A_b}{\de q^a} - \frac{\de {P_q}^A_a}{\de q^b} + {P_q}^B_a \frac{\de {P_q}^A_b}{\de z^B} - {P_q}^B_b \frac{\de {P_q}^A_a}{\de z^B} \,\right)
+ \nonumber \\
{X_p}_b\,&\left(\,\frac{\de {P_p}^{Ab}}{\de q^a} - \frac{\de {P_q}^A_a}{\de p_b} + {P_q}^B_a \frac{\de {P_p}^{Ab}}{\de z^B} - {P_p}^{Bb} \frac{\de {P_q}^A_a}{\de z^B}\,\right) - \nonumber \\
{X_z}^B &\frac{\de {P_q}^A_a}{\de z^B}
\Biggr], \\
G^a = {X_q}^a
+ \mu_A \Biggl[\, X_t\,& \left(\, \frac{\de {P_t}^A}{\de p_a} - \frac{\de {P_p}^{Aa}}{\de t} + {P_p}^{Ba}\frac{\de {P_t}^A}{\de z^B} - {P_t}^B \frac{\de {P_p}^{Aa}}{\de z^B} \,\right) + \nonumber \\
{X_q}^b\,&\left(\,\frac{\de {P_q}^A_b}{\de p_a} - \frac{\de {P_p}^{Aa}}{\de q^b} + {P_p}^{Ba} \frac{\de {P_q}^A_b}{\de z^B} - {P_q}^B_b \frac{\de {P_p}^{Aa}}{\de z^B} \,\right) + \nonumber \\
{X_p}_b\,&\left(\,\frac{\de {P_p}^{Ab}}{\de p_a} - \frac{\de {P_p}^{Aa}}{\de p_b} + {P_p}^{Ba} \frac{\de {P_p}^{Ab}}{\de z^B} - {P_p}^{Bb} \frac{\de {P_p}^{Aa}}{\de z^B} \,\right) - \nonumber \\
{X_z}^B & \frac{\de {P_p}^{Aa}}{\de z^B}
\Biggr], \\
H_A = {X_z}_A
+ \mu_B  \Biggl[ {X_t} \,& \frac{\de {P_t}^B}{\de z^A} + {X_q}^b\, \frac{\de {P_q}^B_b}{\de z^A} + {X_p}_b\,\frac{\de {P_p}^{Bb}}{\de z^A}
\Biggr], \\
K^A = -{X_\mu}^A \,. \hspace{1,1cm} &
\end{align}
%%%
With this in mind, one can easily prove that 
\be
i_X \widetilde{\omega} \,=\, 0 \,\implies\, X \,=\, X_t \, H_t \,,
\ee
if and only if, either 
\be
\mu_A \,= \, 0 \,,
\ee
or
\begin{align}
\left(\, \frac{\de {P_q}^A_a}{\de t} - \frac{\de {P_t}^A}{\de q^a} + {P_t}^B \frac{\de {P_q}^A_a}{\de z^B} - {P_q}^B_a \frac{\de {P_t}^A}{\de z^B} \,\right) \,&=\, 0 \\
\left(\, \frac{\de {P_p}^{Aa}}{\de t} - \frac{\de {P_t}^A}{\de p_a} + {P_t}^B \frac{\de {P_p}^{Aa}}{\de z^B} - {P_p}^{Ba} \frac{\de {P_t}^A}{\de z^B} \,\right) \,&=\, 0 \\
\left(\, \frac{\de {P_t}^A}{\de q^a} - \frac{\de {P_q}^A_a}{\de t} + {P_q}^B_a \frac{\de {P_t}^A}{\de z^B} - {P_t}^B \frac{\de {P_q}^A_a}{\de z^B} \,\right) \,&=\,0 \\
\left(\,\frac{\de {P_q}^A_b}{\de q^a} - \frac{\de {P_q}^A_a}{\de q^b} + {P_q}^B_a \frac{\de {P_q}^A_b}{\de z^B} - {P_q}^B_b \frac{\de {P_q}^A_a}{\de z^B} \,\right) \,&=\, 0 \\
\left(\,\frac{\de {P_p}^{Ab}}{\de q^a} - \frac{\de {P_q}^A_a}{\de p_b} + {P_q}^B_a \frac{\de {P_p}^{Ab}}{\de z^B} - {P_p}^{Bb} \frac{\de {P_q}^A_a}{\de z^B}\,\right) \,&=\,0 \\
\left(\, \frac{\de {P_t}^A}{\de p_a} - \frac{\de {P_p}^{Aa}}{\de t} + {P_p}^{Ba}\frac{\de {P_t}^A}{\de z^B} - {P_t}^B \frac{\de {P_p}^{Aa}}{\de z^B} \,\right)\,&=\, 0 \\
\left(\,\frac{\de {P_q}^A_b}{\de p_a} - \frac{\de {P_p}^{Aa}}{\de q^b} + {P_p}^{Ba} \frac{\de {P_q}^A_b}{\de z^B} - {P_q}^B_b \frac{\de {P_p}^{Aa}}{\de z^B} \,\right) \,&=\,0 \\
\left(\,\frac{\de {P_p}^{Ab}}{\de p_a} - \frac{\de {P_p}^{Aa}}{\de p_b} + {P_p}^{Ba} \frac{\de {P_p}^{Ab}}{\de z^B} - {P_p}^{Bb} \frac{\de {P_p}^{Aa}}{\de z^B} \,\right) \,&=\, 0 \\
\frac{\de {P_t}^B}{\de z^A} \,&=\, 0 \,, \\
\frac{\de {P_q}^A_b}{\de z^B} \,&=\,0 \,, \\
\frac{\de {P_p}^{Ab}}{\de z^B} \,&=\,0 \,.
\end{align}
for all $a,\,b \,=\, 1,...,r$ and for all $A \,=\, 1,...,l$.
%%%
This proves that $\mathrm{ker}\,\tau^* \eta \cap \mathrm{ker} \, \widetilde{\omega} \,=\, \{0\}$ in a tubular neighborhood of the zero section of $\tau$, or, if all the components of $\mathcal{N}_P$ vanish, on the whole ${\Lambda^1}^\perp_R(M)$.
%%%

\noindent Thus, the manifold $\widetilde{M}$ reads a tubular neighborhood of the zero section of the bundle ${\Lambda^1}^\perp_R$ (or the whole bundle ${\Lambda^1}^\perp_R$ if the almost product structure chosen has vanishing Nijenhuis tensor), the embedding map $\mathfrak{i}$ reads the zero section of $\tau$, and the cosymplectic structure is given by the pair $(\widetilde{\eta},\, \widetilde{\omega})$ where
\be
\widetilde{\eta} \,=\, \tau^* \eta \,,
\ee
and
\be
\widetilde{\omega} \,=\, \tau^* \omega + \dd \vartheta^P \,.
\ee
%%%

\noindent To prove that $M$ is a coisotropic submanifold of $\widetilde{M}$, we shall prove that $\T_m M^{\perp_{(\widetilde{\eta}, \widetilde{\omega})}} \subset \T_m M$ for all $m \in \widetilde{M}$.
%%%
Elements of $\T_m M^{\perp_{(\widetilde{\eta}, \widetilde{\omega})}}$ are tangent vectors to $m \in \widetilde{M}$, belonging to $\operatorname{ker}\, \eta_m$
\be
X \,=\, {X_q}^a {H_q}_a\bigr|_{m} + {X_p}_a {H_p}^a\bigr|_{m} + {X_z}^A \frac{\de}{\de z^A}\bigr|_{m} + {X_\mu}^A \frac{\de}{\de \mu^A}\bigr|_{m} \,,
\ee
(now ${X_q}^a,\, {X_p}_a,\, {X_z}^A,$ and ${X_\mu}_A$ are real numbers) such that 
\be \label{Eq: coisotropicity pre-cosymplectic}
\widetilde{\omega}_m(X,\, W) \,=\, 0 \,,\;\;\; \forall \,\, W \in \T_m M \cap \operatorname{ker}\,\widetilde{\eta}_m \,, 
\ee
namely for any tangent vector $W$ to $M$ in $m$ of the type
\be
W \,=\, {W_q}^a {H_q}_a\bigr|_{m} + {W_p}_a {H_p}^a\bigr|_{m} + {W_z}^A \frac{\de}{\de z^A}\bigr|_{m} \,.
\ee
%%%
A straightforward computation shows that
\be
\widetilde{\omega}_m(X,\, W) \,=\, {X_q}^a {W_p}_a - {W_q}^a {X_p}_a + {X_\mu}_A {W_z}^A \,, 
\ee
and, thus \eqref{Eq: coisotropicity pre-contact} implies 
\be
{X_q}^a \,=\, {X_p}_a \,=\, {X_\mu}_A \,=\, 0 \,,\;\;\; \forall\,\, a\,=\, 1,...,r;\,A\,=\,1,...,l \,.
\ee
%%%
Consequently $X \in \T_m M$, which proves that $\T_m M^{\perp_{(\widetilde{\eta}, \widetilde{\omega})}} \subset \T_m M$.
%%%
\end{proof}

\begin{remark} \label{Rem: family embeddings Reeb pre-cosy}
On the pre-cosymplectic manifold $(M, \eta, \omega)$, the Reeb vector field, which is defined by the conditions
\begin{eqnarray} 
i_R \omega &\,=\, 0 \,, \label{Eq: Reeb pre-cosy 1}\\
i_R \eta &\,=\, 1 \,. \label{Eq: Reeb pre-cosy 2}
\end{eqnarray}
is not uniquely determined. 
%%%
Indeed, there exists a whole family of vector fields on $M$ satisfying \eqref{Eq: Reeb pre-cosy 1}-\eqref{Eq: Reeb pre-cosy 2}:
\be \label{Eq: family Reeb pre-cosy}
R = \frac{\de}{\de t} + R^A \frac{\de}{\de z^A} \,,
\ee
parametrized by $l = \mathrm{dim}(\mathcal{V})$ arbitrary functions $R^A$ on $M$.
%%%
With this in mind, \cref{Thm: cosymplectic tickening} can be refined in such a way that the Reeb vector field $\widetilde{R}$ on the cosymplectic manifold $(\widetilde{M}, \widetilde{\eta}, \widetilde{\omega})$ is an extension of one of the Reeb fields $R$ of the family defined above. 
%%%
In other words, for any given field $R$ from the family above, we can construct an embedding $\mathfrak{i}$ such that $\widetilde{R}$ is $\mathfrak{i}$-related to $R$, i.e., $\widetilde{R}|_{\mathfrak{i}(M)} = R$.
%%%

\noindent To show this, one first sees that the Reeb vector field on $\widetilde{M}$ reads 
$$
\widetilde{R} \,=\, \frac{\partial}{\partial t} + P_t^A \frac{\partial}{\partial z^A} \,.
$$
%%%
Note that this is actually tangent to $M$, so it makes sense to ask whether it is an extension of any of the vector fields \eqref{Eq: family Reeb pre-cosy}.
%%%
For $\widetilde{R}|_M$ to coincide with a given $R$, it is sufficient to impose
$$
P_t^A = R^A \,.
$$
%%%
This condition has a clear geometric interpretation: it is equivalent to choosing an almost product structure $P$ such that the pre-selected Reeb field $R$ belongs to the horizontal distribution $\mathcal{H} = \ker P$. 
%%%
A direct computation shows:
\be
    P(R) \,=\, (R^A - P_t^A) \otimes \frac{\de}{\de z^A} \,.
\ee
%%%
The condition $P(R) = 0$ is therefore satisfied if and only if $P_t^A = R^A$. 
%%%
Since one can always construct a horizontal distribution $\mathcal{H}$ (and, thus, an almost product structure $P$) that contains a given vector field transverse to the vertical distribution $\mathcal{V}$, we can conclude that for every Reeb field $R$ on $M$, there exists a coisotropic embedding that extends it.
%%%
\end{remark}

\subsubsection{Uniqueness}

As \cref{Rem: family embeddings Reeb pre-cosy} clearly shows, differently from the pre-symplectic case, coisotropic embeddings are not unique in this context, as we get a continuous family of possible coisotropic embeddings, parametrized by each possible choice of Reeb vector field. 
%%%

\noindent What we can guarantee, however, is that the topology of the embeddings must coincide, as there is a vector bundle isomorphism of $\T \widetilde{M} |_M$ with $\T \mathcal{V}^\ast |_M$, where $\mathfrak{i} \colon M \hookrightarrow \widetilde{M}$ is an arbitrary coisotropic embedding, and $\mathcal{V}$ is the characteristic distribution on $M$.

\begin{theorem} Let $(M, \omega, \eta)$ be a pre-cosymplectic manifold, and $ \mathfrak{i}\colon (M, \omega, \eta) \longrightarrow (\widetilde{M}, \widetilde{\omega}, \widetilde{\eta})$ be a coisotropic embedding. Then, $(M, \widetilde{M})$ and $(M, \mathcal{V}^\ast)$ are neighborhood diffeomorphic (not cosymplectomorphic).
\end{theorem}

\begin{proof} Decompose $\T M$ as $\T M = \mathcal{V} \oplus \langle R \rangle \oplus \mathcal{H}$, where $\mathcal{V}$ is the characteristic distribution, $R$ is a choice of Reeb vector field, and $\mathcal{H}$ is an arbitrary complement to $\mathcal{V}$ in $\ker \eta$. Then, an argument similar to the one followed in \cref{Thm: uniqueness pre-symplectic} shows that $\mathcal{H}$ is a symplectic bundle in $\T \widetilde{M} |_M$, and so is $\mathcal{H}^\perp$, which satisfies 
\[
\T \widetilde{M} = \mathcal{H} \oplus \mathcal{H}^\perp \oplus \langle\widetilde R \rangle\,,
\]
where $\widetilde R$ denotes the Reeb vector field on $\widetilde M$. Since $\mathcal{V} \subseteq \mathcal{H}^\perp$ is Lagrangian, we have 
\[
\T \widetilde{M} |_M  = \mathcal{H} \oplus \mathcal{H}^\perp \oplus \langle \widetilde R\rangle \cong \mathcal{H} \oplus \mathcal{V} \oplus \mathcal{V}^\ast \oplus \langle R\rangle \cong \T \mathcal{V}^\ast |_{M}\,.
\]
\end{proof}

\begin{remark} It is easy to show that under the aditional condition that two coisotropic embeddings induce the same Reeb vector field on $M$, the the previous diffeomorphism preserves the cosymplectic structure on $M$.
\end{remark}

\noindent A natural question to ask is whether we can guarantee the uniqueness of the embedding when a possible Reeb vector field is fixed on $M$. This is not the case with full generality, as the following example shows:

\begin{example}
\label{ex:cosymplectic_nonuniqueness}
Let $M = \mathbb{S}^1 \times \mathbb{S}^1$ be the two-dimensional torus, together with the pre-symplectic structure $\omega = 0$, $\eta = \dd \theta_1$, where $\dd \theta_1$ denotes the canonical volume form on the first copy of $\mathbb{S}^1$. Let $\widetilde{M} := \mathbb{R} \times \mathbb{S}^1 \times \mathbb{S}^1$ together with the following two possible cosymplectic structures:
\[
(\widetilde{\omega}_1 = \dd t \wedge \dd \theta_2, \eta_1 = \dd \theta_1)\,, \qquad (\widetilde{\omega}_2 = \dd t \wedge \dd \theta_2 + t \dd \theta_1 \wedge \dd t\,, \eta = \dd \theta_1).
\]
It is clear that $M$, identified as the zero section of $\widetilde{M} = \mathbb{R} \times M$, is coisotropic for both structures. Let us check that these are not neighborhood isomorphic. Indeed, their respective Reeb vector fields are written as follows:
\[
R_1 = \frac{\partial}{\partial\theta^1}\,,\qquad R_2 =  \frac{\partial}{\partial\theta^1} + t  \frac{\partial}{\partial\theta^2}\,,
\]
and the orbits that these two vector fields define lie within the tori $\{a\} \times \mathbb{S}^1 \times \mathbb{S}^1$, where $a \in \mathbb{R}$ is arbitrary. Clearly, these two vector fields coincide on $M$ (where $t = 0$). However, for each neighborhood of $M$ in $\widetilde{M}$, there is a smaller neighborhood such that $R_1$ is only comprised of periodic orbits. However, $R_2$ does not satisfy this same property, showing that they are not neighborhood equivalent.
\end{example}

\noindent However, if we have a coisotropic embedding of a pre-cosymplectic manifold $(M, \omega, \eta)$ into one manifold $\widetilde{M}$ with two different cosymplectic structures $(\widetilde{\omega}_1, \widetilde{\eta}_1)$, $(\widetilde{\omega}_2, \widetilde{\eta}_2)$, such that they coincide on $M$ and their corresponding Reeb vector fields are proportional, they are neighborhood equivalent:

\begin{theorem} Let $(M, \omega, \eta)$ be a pre-cosymplectic manifold and $\mathfrak{i}\colon M \hookrightarrow \widetilde{M}$ be a coisotropic embedding for two different cosymplectic structures on $\widetilde{M}$, $(\widetilde{\omega}_1, \widetilde{\eta}_1)$ and $(\widetilde{\omega}_2, \widetilde{\eta}_2)$. Suppose that these two structures satisfy the following hypotheses:
\begin{enumerate}
    \item They coincide on $M$;
    \item Their Reeb vector fields are proportional.
\end{enumerate}
Then, these two embeddings are neighborhood equivalent.
\end{theorem}
\begin{proof} We just apply Moser's trick and look for a time-dependent vector field $X_t$ satisfying 
\[
\Lie_{X_t} \widetilde{\omega}_t = \widetilde{\omega}_2 - \widetilde{\omega}_1\,,\qquad \Lie_{X_t} \widetilde{\eta}_t = \widetilde{\eta}_2 - \widetilde{\eta}_1\,,
\]
where $\widetilde \omega_t = t \widetilde \omega_2 + (1-t) \widetilde \omega_1$, and $\widetilde \eta_t = t \widetilde \eta_2 + (1-t) \widetilde \eta_1$. Applying Poincar\'e Lemma, let us denote by $\theta$ and $f$ a $1$-form and a function vanishing on $M$, respectively, satisfying 
\[
\widetilde{\omega}_2 - \widetilde{\omega}_1  = \dd \theta\,, \qquad \widetilde \eta_1 - \widetilde \eta_1 = \dd f\,.
\]
%%%
Then, it is enough to look for a time-dependent vector field satisfying 
\[
i_{X_t} \widetilde \omega_t = \theta\,, \qquad i_{X_t} \theta_t = f\,.
\]
For this, it is enough to show that we may choose $\theta$ such that $i_{\widetilde R_t} \theta = 0$, for every $t$, where $\widetilde R_t$ denotes the Reeb vector field associated to the cosymplectic structure $(\widetilde \omega_t, \widetilde \eta_t)$. Indeed, using the Relative Poincaré Lemma (\cref{thm:Relative_Poincaré_Lemma}), $\theta$ can be explicitely defined as 
\[
\theta = \int_0^1 t^2i_{\Delta_t} (\widetilde \omega_2 - \widetilde \omega_1)\dd t\,,
\]
where $\Delta$ is a Liouville vector field for a tubular neighborhood. From this construction, and using that $\widetilde R_t$ is propostional to $\widetilde R_1$ and $\widetilde R_2$, it follows that $i_{\widetilde R_t} \theta = 0$, so that we may apply Moser's trick, and the theorem follows.
\end{proof}

\section{Pre-contact manifolds}
\label{Sec: contact manifolds}

\begin{definition}[\textsc{Pre-contact manifold}]
A \textbf{pre-contact manifold} is a pair $(M, \eta)$, where $M$ is a smooth $(2d+1)$-dimensional manifold, and $\eta \in \Omega^1(M)$ is a differential 1-form of constant rank.
\end{definition}

\begin{remark}
We always assume that $\eta$ is nowhere vanishing and that the distribution 
\[
\mathcal{V} := \ker \eta \subset \mathbf{T}M
\]
has constant rank $2d$, and defines a smooth subbundle of the tangent bundle.
%%%
We also assume that the restriction of $\dd \eta$ to $\mathcal{V}$ has constant rank.
\end{remark}

\begin{definition}[\textsc{Characteristic distribution}]
Let $(M, \eta)$ be a pre-contact manifold.
%%%
We define the \textbf{characteristic distribution} of $\eta$ as
\[
\mathcal{V} \,:=\, \ker \eta \cap \ker(\dd \eta|_{\ker \eta}) \,.
\]
%%%
It is the maximal integrable distribution contained in $\ker \eta$ on which $\dd \eta$ vanishes.
\end{definition}

\begin{definition}[\textsc{Reeb vector field}]
Let $(M, \eta)$ be a pre-contact manifold. 
%%%
A \textbf{Reeb vector field} is a vector field $R \in \mathfrak{X}(M)$ satisfying the conditions
\be
i_R \dd\eta = 0 \,, \qquad i_R \eta = 1 \,.
\ee
\end{definition}

\begin{remark}[\textsc{Contact manifold}]
When $\eta \in \Omega^1(M)$ is such that
\[
\eta \wedge (\dd \eta)^d \,\neq\, 0
\]
everywhere on $M$, the 1-form defines a \textbf{contact structure}, and the pair $(M, \eta)$ is called a \textbf{contact manifold}.
%%%
In this case, $\ker \eta$ has rank $2d$, and the restriction $\dd \eta|_{\ker \eta}$ defines a symplectic structure.
%%%
\end{remark}

\begin{theorem}[\textsc{Darboux theorem for contact manifolds}]
Let $(M, \eta)$ be a contact manifold of dimension $2d+1$.
%%%
Then, for every point $m \in M$, there exists a coordinate chart $(U, \varphi)$ centered at $m$, with coordinates $(t,\, q^1, \dots, q^d,\, p_1, \dots, p_d)$, such that
\[
\eta|_U = \dd t - p_a \dd q^a \,.
\]
Such coordinates are called \textbf{Darboux coordinates} for the contact form $\eta$.
\begin{proof}
See \cite{godbillon1969geometrie}.
\end{proof}
\end{theorem}

\begin{theorem}[\textsc{Darboux theorem for pre-contact manifolds}]
Let $(M, \eta)$ be a pre-contact manifold of dimension $2d+1$, with $\eta \in \Omega^1(M)$ a 1-form of constant rank, and such that the rank of $\dd \eta|_{\ker \eta}$ is constant and equal to $2r < 2d$.
%%%
Then, for every point $m \in M$, there exists a coordinate chart $(U, \varphi)$ around $m$, with coordinates
\[
(t,\, q^1, \dots, q^r,\, p_1, \dots, p_r,\, z^1, \dots, z^{2(d-r)}) \,,
\]
such that the 1-form $\eta$ reads
\[
\eta|_U = \dd t - p_a \dd q^a \,.
\]
%%%
These coordinates are called \textbf{Darboux coordinates} for the pre-contact form $\eta$.
%%%
In these coordinates, the characteristic distribution of $\eta$ is
\[
\mathcal{V} = \operatorname{span} \left\{\, \frac{\partial}{\partial z^A} \,\right\}_{A = 1, \dots, 2(d - r)} \,.
\]
Thus, Darboux coordinates are foliated charts for the foliation associated with the characteristic distribution $\mathcal{V}$.
\begin{proof}
See \cite{godbillon1969geometrie}.
\end{proof}
\end{theorem}

\subsection{Coisotropic submanifolds}

\begin{definition}[\textsc{Contact orthogonal}]
\label{Def: contact orthogonal}
Let $(M, \eta)$ be a contact manifold, and let $W_m \subseteq \mathbf{T}_m M$ be a linear subspace of the tangent space at a point $m \in M$.
%%%
The \textbf{contact orthogonal} of $W_m$ with respect to $\eta$ is defined as
\[
W^{\perp_\eta}_m := \left\{\, X \in \ker \eta_m  \colon  \dd \eta(X, Y) = 0 \quad \forall \, Y \in W_m \,\right\}.
\]
\end{definition}

\begin{remark}
Notice that $W^{\perp_\eta}_m$ is nothing but the symplectic orthogonal of $W_m$ inside $\ker \eta_m$.
\end{remark}

\begin{definition}[\textsc{Coisotropic submanifold of a contact manifold}]
Let $(M, \eta)$ be a contact manifold, and let
\[
\mathfrak{i}  \colon  N \hookrightarrow M
\]
be an immersed submanifold.
%%%
We say that $N$ is a \textbf{coisotropic submanifold} of $(M, \eta)$ if, for every point $n \in N$
\[
{\mathbf{T}_n N}^{\perp_\eta} \,\subseteq\, \mathbf{T}_n N\,.
\]
\end{definition}

\begin{remark}
    Similar to the cosymplectic setting, for each contact manifold there is an associated Jacobi structure, which have a notion of orthogonal. Again, it can be easily shown that both definitions coincide (see \cite{mlvmdl1}).
\end{remark}

\subsection{Coisotropic embeddings}

\subsubsection{Existence}

\begin{theorem}[\textsc{Contact thickenings of pre-contact manifolds}]
\label{Thm: existence contact}
Let $(M,\,\eta)$ be a pre-contact manifold.
%%%
There exists a contact manifold $(\widetilde{M},\,\widetilde{\eta})$ and an embedding
\[
\mathfrak{i} \colon  M \hookrightarrow \widetilde{M}
\]
such that $\mathfrak{i}(M)$ is a coisotropic submanifold of $\widetilde{M}$.
%%%
The contact manifold $(\widetilde{M},\,\widetilde{\eta})$ is referred to as a \textbf{contact thickening} of $(M,\,\eta)$.
%%%
\begin{proof}
Consider the characteristic distribution $\mathcal{V}$ on $M$.
%%%
Consider a complementary distribution $\mathcal{H}$ providing an almost product structure on $M$ associated with the $(1,1)$-tensor $P$ on $M$ that, in the system of Darboux coordinates on $M$, reads
\be
P \,=\, \left(\,\dd z^A - {P_t}^A \dd t - {P_q}^A_a \dd q^a - {P_p}^{Aa} \dd p_a\,\right) \otimes \frac{\de}{\de z^A} \,.
\ee
%%%
The opposite $(1,1)$-tensor $R$ reads
\be
R \,=\, \mathds{1} - P \,=\, \dd t \otimes \left(\, \frac{\de}{\de t} + {P_t}^A \frac{\de}{\de z^A} \,\right) + \dd q^a \otimes \left(\, \frac{\de}{\de q^a} + {P_q}^A_a \frac{\de}{\de z^A} \,\right) + \dd p_a \otimes \left(\, 
\frac{\de}{\de p_a} + {P_p}^{Aa} \frac{\de}{\de z^A} \,\right) \,.
\ee
%%%
Consider the bundle ${\Lambda^1}^\perp_R(M)$ over $M$.
%%%
Denote by $\tau$ the projection map onto $M$.
%%%
Its sections are differential $1$-form that, in the system of Darboux coordinates looks like
\be
\alpha \,=\, \alpha_A P^A \,.
\ee
%%%
Thus, a system of adapted coordinates on ${\Lambda^1}^\perp_R(M)$, reads
\be
\left\{\,t,\, q^a,\,p_a,\,z^A,\, \mu_A \,\right\}_{\substack{a=1,...,r \\ A=1,...,l}} \,,
\ee
where $l=1,...,2(d-r)$.
%%%
The projection map $\tau$ thus reads
\be
\tau  \colon  (t,\,q^a,\,p_a,\,z^A,\, \mu_A) \mapsto (t,\,q^a,\,p_a,\,z^A) \,.
\ee
%%%
The bundle ${\Lambda^1}^\perp_R(M)$ has a distinguished $1$-form, which is the obvious analogous of the tautological $1$-form on $\mathbf{T}^* M$, defined as
\be
\vartheta^P_{\alpha}(X) \,=\, \alpha(\tau_*(X)) \,, \;\; \forall\,\, X \in \mathbf{T}_\alpha M \,,
\ee
where $\alpha$ has to be considered as a point in ${\Lambda^1}^\perp_R(M)$ on the left hand side, and as a differential $1$-form on $M$ on the right hand side.
%%%
In the system of Darboux coordinates $\vartheta^P$ reads
\be
\vartheta^P \,=\, \mu_A P^A \,.
\ee
%%%
Consider the differential $1$-form
\be
\widetilde{\eta} \,=\, \tau^* \eta - \vartheta^P \,,
\ee
on ${\Lambda^1}^\perp_R(M)$.
%%%
It is a contact structure in a tubular neighborhood of the zero section of $\tau$.
%%%
Indeed, consider a vector field $X$ written in the basis
\be
\left\{\, {H_t} \,:=\, \frac{\de}{\de t} + {P_t}^A \frac{\de}{\de z^A},\, {H_q}_a \,:=\,\frac{\de}{\de q^a} + {P_q}^A_a \frac{\de}{\de z^A},\, {H_p}^a\,:=\,\frac{\de}{\de p_a} + {P_p}^{Aa} \frac{\de}{\de z^A}, \, \frac{\de}{\de z^A},\, \frac{\de}{\de \mu_A} \,\right\}_{\substack{a=1,...,r \\ A=1,...,l}} \,,
\ee
as
\be
X \,=\, {X_t} \, H_t + {X_q}^a \,{H_q}_a + {X_p}_a \,{H_p}^a + {X_z}^A \, \frac{\de}{\de z^A} + {X_\mu}_A\, \frac{\de}{\de \mu_A} \,.
\ee
%%%
The contraction $i_X \dd \widetilde{\eta}$ reads
\be
i_X \dd\widetilde{\eta} = E \, \dd t +  F_a\,\dd q^a + G^a\,\dd p_a + H_A\,P^A + K^A\,\dd\mu_A,
\ee
where:
\begin{align}
E \,=\, \mu_A \Biggl[\, {X_q}^a &\left(\, \frac{\de {P_t}^A}{\de q^a} - \frac{\de {P_q}^A_a}{\de t} + {P_q}^B_a \frac{\de {P_t}^A}{\de z^B} - {P_t}^B \frac{\de {P_q}^A_a}{\de z^B}
 \,\right) + \nonumber \\ 
 {X_p}_a &\left(\, \frac{\de {P_t}^A}{\de p_a} - \frac{\de {P_p}^{Aa}}{\de t} + {P_p}^{Ba} \frac{\de {P_t}^A}{\de z^B} - {P_t}^B \frac{\de {P_p}^{Aa}}{\de z^B} \,\right) + \nonumber \\
 {X_z}^B &\frac{\de {P_t}^A}{\de z^B}\,\Biggr] \,, \\
F_a = -{X_q}_a
+ \mu_A \Biggl[ {X_t} &\left(\, \frac{\de {P_q}^A_a}{\de t} - \frac{\de {P_t}^A}{\de q^a} + {P_t}^B \frac{\de {P_q}^A_a}{\de z^B} - {P_q}^B_a \frac{\de {P_t}^A}{\de z^B} \,\right) + \nonumber \\
{X_q}^b\,&\left(\,\frac{\de {P_q}^A_b}{\de q^a} - \frac{\de {P_q}^A_a}{\de q^b} + {P_q}^B_a \frac{\de {P_q}^A_b}{\de z^B} - {P_q}^B_b \frac{\de {P_q}^A_a}{\de z^B} \,\right)
+ \nonumber \\
{X_p}_b\,&\left(\,\frac{\de {P_p}^{Ab}}{\de q^a} - \frac{\de {P_q}^A_a}{\de p_b} + {P_q}^B_a \frac{\de {P_p}^{Ab}}{\de z^B} - {P_p}^{Bb} \frac{\de {P_q}^A_a}{\de z^B}\,\right) + \nonumber \\
{X_z}^B &\frac{\de {P_q}^A_a}{\de z^B}
\Biggr], \\
G^a = {X_q}^a
+ \mu_A \Biggl[\, X_t &\left(\, \frac{\de {P_p}^{Aa}}{\de t} - \frac{\de {P_t}^A}{\de p_a} + {P_t}^B \frac{\de {P_p}^{Aa}}{\de z^B} - {P_p}^{Ba} \frac{\de {P_t}^A}{\de z^B} \,\right) + \nonumber \\
{X_q}^b\,&\left(\,\frac{\de {P_q}^A_b}{\de p_a} - \frac{\de {P_p}^{Aa}}{\de q^b} + {P_p}^{Ba} \frac{\de {P_q}^A_b}{\de z^B} - {P_q}^B_b \frac{\de {P_p}^{Aa}}{\de z^B} \,\right) + \nonumber \\
{X_p}_b\,&\left(\,\frac{\de {P_p}^{Ab}}{\de p_a} - \frac{\de {P_p}^{Aa}}{\de p_b} + {P_p}^{Ba} \frac{\de {P_p}^{Ab}}{\de z^B} - {P_p}^{Bb} \frac{\de {P_p}^{Aa}}{\de z^B} \,\right) + \nonumber \\
{X_z}^B & \frac{\de {P_p}^{Aa}}{\de z^B}
\Biggr], \\
H_A = - {X_z}_A
- \mu_B  \Biggl[ {X_t} &\frac{\de {P_t}^B}{\de z^A}
{X_q}^b\, \frac{\de {P_q}^B_b}{\de z^A} + {X_p}_b\,\frac{\de {P_p}^{Bb}}{\de z^A}
\Biggr], \\
K^A = {X_\mu}^A \,. \hspace{1,1cm} &
\end{align}
%%%
With this in mind, one can easily prove that 
\be
i_X \dd \widetilde{\eta} \,=\, 0 \,\implies\, {X_q}^a \,=\, {X_p}_a \,=\, {X_z}^A \,=\, {X_\mu}_A \,=\, 0 \,,
\ee
if and only if, either 
\be
\mu_A \,= \, 0 \,,
\ee
or
\begin{align}
\left(\, \frac{\de {P_t}^A}{\de q^a} - \frac{\de {P_q}^A_a}{\de t} + {P_q}^B_a \frac{\de {P_t}^A}{\de z^B} - {P_t}^B \frac{\de {P_q}^A_a}{\de z^B}
 \,\right) \,&=\, 0 \\
\left(\, \frac{\de {P_t}^A}{\de p_a} - \frac{\de {P_p}^{Aa}}{\de t} + {P_p}^{Ba} \frac{\de {P_t}^A}{\de z^B} - {P_t}^B \frac{\de {P_p}^{Aa}}{\de z^B} \,\right) \,&=\, 0 \\
\left(\, \frac{\de {P_q}^A_a}{\de t} - \frac{\de {P_t}^A}{\de q^a} + {P_t}^B \frac{\de {P_q}^A_a}{\de z^B} - {P_q}^B_a \frac{\de {P_t}^A}{\de z^B} \,\right) \,&=\, 0 \\
\left(\,\frac{\de {P_q}^A_b}{\de q^a} - \frac{\de {P_q}^A_a}{\de q^b} + {P_q}^B_a \frac{\de {P_q}^A_b}{\de z^B} - {P_q}^B_b \frac{\de {P_q}^A_a}{\de z^B} \,\right) \,&=\, 0 \\
\left(\,\frac{\de {P_p}^{Ab}}{\de q^a} - \frac{\de {P_q}^A_a}{\de p_b} + {P_q}^B_a \frac{\de {P_p}^{Ab}}{\de z^B} - {P_p}^{Bb} \frac{\de {P_q}^A_a}{\de z^B}\,\right) \,&=\, 0 \\
\left(\, \frac{\de {P_p}^{Aa}}{\de t} - \frac{\de {P_t}^A}{\de p_a} + {P_t}^B \frac{\de {P_p}^{Aa}}{\de z^B} - {P_p}^{Ba} \frac{\de {P_t}^A}{\de z^B} \,\right) \,&=\, 0 \\
\left(\,\frac{\de {P_q}^A_b}{\de p_a} - \frac{\de {P_p}^{Aa}}{\de q^b} + {P_p}^{Ba} \frac{\de {P_q}^A_b}{\de z^B} - {P_q}^B_b \frac{\de {P_p}^{Aa}}{\de z^B} \,\right) \,&=\, 0 \\
\left(\,\frac{\de {P_p}^{Ab}}{\de p_a} - \frac{\de {P_p}^{Aa}}{\de p_b} + {P_p}^{Ba} \frac{\de {P_p}^{Ab}}{\de z^B} - {P_p}^{Bb} \frac{\de {P_p}^{Aa}}{\de z^B} \,\right) \,&=\, 0 \\
\frac{\de {P_t}^B}{\de z^A} \,&=\, 0 \\
\frac{\de {P_q}^B_b}{\de z^A} \,&=\, 0 \\
\frac{\de {P_p}^{Bb}}{\de z^A} \,&=\, 0 \,.
\end{align}
for all $a,\,b \,=\, 1,...,r$ and for all $A, B \,=\, 1,...,l$.
%%%
This proves that $\dd \widetilde{\eta}$ has kernel spanned by
\be
\mathrm{ker} \,\dd \widetilde{\eta} \,=\, \langle \left\{\, \frac{\de}{\de t} \,\right\}\rangle
\ee
in a tubular neighborhood of the zero section of $\tau$ (i.e. $\mu_A \,=\, 0$), or, if all the components of $\mathcal{N}_P$ vanish, on the whole ${\Lambda^1}^\perp_R(M)$.
%%%

\noindent Thus, the manifold $\widetilde{M}$ reads a tubular neighborhood of the zero section of the bundle ${\Lambda^1}^\perp_R$ (or the whole bundle ${\Lambda^1}^\perp_R$ if the almost product structure chosen has vanishing Nijenhuis tensor), the embedding map $\mathfrak{i}$ reads the zero section of $\tau$, and the contact structure $\widetilde{\eta}$ is
\be
\widetilde{\eta} \,=\, \tau^* \eta - \vartheta^P \,.
\ee
%%%

\noindent To prove that $M$ is a coisotropic submanifold of $\widetilde{M}$, we shall prove that $\T_m M^{\perp_{\widetilde{\eta}}} \subset \T_m M$ for all $m \in \widetilde{M}$.
%%%
Elements of $\T_m M^{\perp_{\widetilde{\eta}}}$ are tangent vectors to $m \in \widetilde{M}$, belonging to $\operatorname{ker}\, \widetilde{\eta}_m$
\be
X \,=\, {p_a}_m {W_q}^a {H_t}\bigr|_{m} + {X_q}^a {H_q}_a\bigr|_{m} + {X_p}_a {H_p}^a\bigr|_{m} + {X_z}^A \frac{\de}{\de z^A}\bigr|_{m} + {X_\mu}^A \frac{\de}{\de \mu^A}\bigr|_{m} \,,
\ee
(now ${X_q}^a,\, {X_p}_a,\, {X_z}^A,$ and ${X_\mu}_A$ are real numbers) such that 
\be \label{Eq: coisotropicity pre-contact}
\dd \widetilde{\eta}_m(X,\, W) \,=\, 0 \,,\;\;\; \forall \,\, W \in \T_m M \cap \operatorname{ker}\,\widetilde{\eta}_m \,, 
\ee
namely for any tangent vector $W$ to $M$ in $m$ of the type
\be
W \,=\, {p_a}_m {W_q}^a {H_t}\bigr|_{m} + {W_q}^a {H_q}_a\bigr|_{m} + {W_p}_a {H_p}^a\bigr|_{m} + {W_z}^A \frac{\de}{\de z^A}\bigr|_{m} \,.
\ee
%%%
A straightforward computation shows that
\be
\dd \widetilde{\eta}_m(X,\, W) \,=\, {X_q}^a {W_p}_a - {W_q}^a {X_p}_a + {X_\mu}_A {W_z}^A \,, 
\ee
and, thus \eqref{Eq: coisotropicity pre-contact} implies 
\be
{X_q}^a \,=\, {X_p}_a \,=\, {X_\mu}_A \,=\, 0 \,,\;\;\; \forall\,\, a\,=\, 1,...,r;\,A\,=\,1,...,l \,.
\ee
%%%
Consequently $X \in \T_m M$, which proves that $\T_m M^{\perp_{\widetilde{\eta}}} \subset \T_m M$.
%%%
\end{proof}
\end{theorem}

\begin{remark} \label{Rem: family embeddings Reeb contact}
Similarly to the pre-cosymplectic case, one can show that a pre-contact manifold $(M,\, \eta)$ has a whole family of Reeb vector fields
\be
R \,=\, \frac{\de}{\de t} + R^A \frac{\de}{\de z^A} \,,
\ee
and, following the same proof, one can show that for each element of such a family, a coisotropic embedding $\mathfrak{i}$ from $(M,\, \eta)$ to $(\widetilde{M},\, \widetilde{\eta})$ can be constructed (the one given by the almost product structure such that $P(R) \,=\, 0$) such that the Reeb vector field of $\widetilde{R}$ of $(\widetilde{M},\, \widetilde{\eta})$ is $\mathfrak{i}$-related with $R$.
%%%
\end{remark}

\subsubsection{Uniqueness}

As for pre-cosymplectic manifolds, \cref{Rem: family embeddings Reeb contact} immediately implies also in the pre-contact case that coisotropic embeddings are not unique, as different choices of Reeb vector fields are possible. However, and similar to the case of pre-cosymplectic embeddings, the topology of the embeddings is fixed:

\begin{theorem} Let $(M, \eta)$ be a pre-contact manifold, and let $\mathfrak{i} \colon (M, \eta) \hookrightarrow (\widetilde M, \widetilde \eta)$ be a coisotropic embedding. Then, it is neighborhood diffeomorphic to the one built in \cref{Thm: existence contact}.
\end{theorem}

\noindent Throughout the proof, we use the term symplectic and Lagrangian subbundle of a contact manifold $(M, \eta)$ in the following sense: When dealing with a vector subbundle $D$ contained in $\ker\, \eta$, we say that $D$ is Lagrangian (resp. symplectic) when $D$ is a Lagrangian (resp. symplectic) vector subbundle of the symplectic vector bundle $(\ker \,\eta, \dd \eta)$.

\begin{proof} Denote by $\mathcal{V} = \ker\, \eta \cap \ker\, \dd \eta$ the characteristic distribution on $M$, and let $\mathcal{H}$ be any complement to $\mathcal{V} \oplus \langle R \rangle$, where $R$ is an arbitrary choice of vector field satisfying 
\[ \eta(R) = 1\,, \qquad i_R \dd \eta = 0\,,
\] so that $\T M = \mathcal{V} \oplus \mathcal{H}\oplus \langle R \rangle$. 
%%%
Following the techniques used in \cref{Sec: cosymplectic manifolds}, we have that 
\[
\T \widetilde M|_M  = \mathcal{H}^\perp\oplus \mathcal{H} \oplus \langle\widetilde R \rangle\,,
\]
where $\widetilde R$ denotes the Reeb vector field on $\widetilde M$. Furthermore, since $\mathcal{V}\subseteq \mathcal{H}^\perp$ is a Lagrangian subbundle from a symplectic subbundle (since $\mathcal{H}$ is symplectic in the pre-symplectic vector bundle $(\T \widetilde M |_M, \dd \eta)$), we have $\mathcal{H}^\perp \cong \mathcal{V} \oplus \mathcal{V}^\ast$. Hence, we have that 
\[
\T \widetilde M |_M = \mathcal{H}\oplus \mathcal{H}^\perp \oplus \langle \widetilde R\rangle \cong \mathcal{H} \oplus \mathcal{V}\oplus \mathcal{V}^\ast \oplus \langle\widetilde R \rangle\,, 
\]
which is clearly isomorphic to $\mathcal{H} \oplus \mathcal{V} \oplus \langle R\rangle \oplus \mathcal{V}^\ast \cong \mathcal{V}^\ast \oplus \T M \cong  \T \mathcal{V}^\ast |_M$. 
%%%
Taking a tubular neighborhood of the zero-section finishes the proof.
\end{proof}

\begin{remark} Furthermore, when two embeddings are such that their Reeb vector fields are equal on $M$, then it is easy to check that the previous diffeomorphism actually preserves the contact structure (both $\eta$ and $\dd \eta$) on $M$.
\end{remark}

\noindent As for the pre-cosymplectic scenario, fixing a Reeb vector field on $M$ is not enough to guarantee uniqueness, as the following example shows:

\begin{example}
\label{example:contact_nonuniqueness}
Let $M = \mathbb{S}^1 \times \mathbb{S}^1$ be the two-dimensional torus, together with the pre-contact structure given by $\eta = \dd \theta_2$, where $\dd \theta_2$ denotes the canonical volume form on the second copy of $\mathbb{S}^1$. Let $\widetilde{M} := (-1,1) \times \mathbb{S}^1 \times \mathbb{S}^1$ together with the following two possible contact structures:
\[
\widetilde \eta_1 = \dd \theta_2 - t \dd \theta_1\,, \qquad \widetilde \eta_2 = \left(1+\frac{t^2}{2}\right) \dd \theta_2 - t \dd \theta_1.
\]
It is clear that $M$, identified as the zero section of $\widetilde M = (-1,1) \times M$, is a coisotropic submanifold of $\widetilde M$, with respect to both contact structures. Let us check that these are not neighborhood isomorphic. Indeed, their respective Reeb vector fields are written as follows:
\[
\widetilde R_1 = \frac{\partial}{\partial\theta^2}\,,\qquad \widetilde R_2 = \frac{2}{2-t^2} \left(t\,\frac{\partial}{\partial\theta^1} + \frac{\partial}{\partial\theta^2}\right)\,.
\]
%%%
The argument now proceeds as in the cosymplectic scenario.
\end{example}

\noindent What we can guarantee, however, that if the orbits of the Reeb vector fields coincide on a neighborhood of $M$, then the contact structures are neighborhood equivalent.

\begin{theorem} Let $\widetilde M$ be a manifold and $\mathfrak{i} \colon M \hookrightarrow \widetilde M$ be a coisotropic submanifold. Suppose $\widetilde \eta_1$ and $\widetilde \eta_2$ are two contact forms on $\widetilde M$ such that $\widetilde \eta_1 = \widetilde \eta_2$ and $\dd \widetilde \eta_1 = \dd \widetilde \eta_2$ on $M$, and such that $M$ is a coisotropic submanifold of the contact structure induced by both of them. If the Reeb vector fields of the two contact structures are proportional, then the contact structures are neighborhood equivalent.
\end{theorem}

\begin{proof} 
 Notice first that we may assume that both vector fields are equal, as it is easy to find a diffeomorphism that transforms one into another. Let us apply Moser's trick, and look for a time-dependent vector field $X_t$ such that 
\[
\dot{\widetilde \eta}_t + \Lie_{X_t} \widetilde \eta_t= 0\,,
\]
where $\widetilde \eta_t = t \widetilde \eta_2 + (1-t) \widetilde \eta_1$, that is, we look for a vector field satifying 
\[
\widetilde \eta_1 - \widetilde \eta _2 = \dd i_{X_t} \widetilde \eta_t + i_{X_t} \dd \widetilde \eta_t\,.
\]
Notice that the Reeb vector field of $\widetilde \eta_t$, $\widetilde R_t$, is the Reeb vector field of $\widetilde \eta_1$ and $\widetilde \eta_2$, which we denote by $\widetilde R$. Hence, we may choose $X_t$ to lie in the horizontal distribution $\ker \widetilde \eta_t$, and look for a solution of 
\[
i_{X_t} \dd \widetilde \eta_t = \widetilde \eta_1 - \widetilde \eta_2\,,
\]
which has a unique solution, as $\widetilde R_t$ annihilates the right term. As usual, the time one flow of this vector field (which exists on a neighborhodd since $X_t = 0$ on $M$) defines the desired contactomorphism.
\end{proof}

\begin{remark}
(Another definition of contact manifold). Some authors define contact manifolds $(M, \xi)$ as odd-dimensional manifolds
$M$ with a contact distribution $\xi$, that is, a maximally nonintegrable codimension 1 distribution. By the Frobenius theorem, this means that the distribution $\xi$ is
given locally as the kernel of a contact form $\eta$. Of course, every contact manifold $(M, \eta)$ is a contact manifold in this sense by taking $\xi = \ker \eta$.
Conversely, a contact distribution $\xi$ is globally the kernel of a contact form if and only if $\xi$ is co-orientable.
Although some of the concepts we will be working with only depend on the contact distribution, such as those of isotropic, coisotropic,
and Legendrian submanifold, the choice of contact forms plays a crucial role on the dynamics of the contact Hamiltonian system, and this is the reason for our choice.
We refer to \cite{vitagliano} where an embedding coisotropic theorem has been proved in this alternative context (see also \cite{tortorella}).
\end{remark}

\section{Pre-cocontact manifolds}
\label{Sec: cocontact manifolds}

\begin{definition}[\textsc{Pre-cocontact manifold}]
A \textbf{pre-cocontact manifold} is a triple $(M, \xi, \eta)$, where $M$ is a smooth $d$-dimensional manifold and $\xi, \eta \in \Omega^1(M)$ are differential 1-forms such that $\xi$ is closed, i.e., $\dd\xi = 0$.
\end{definition}

\begin{definition}[\textsc{Characteristic distribution}]
The \textbf{characteristic distribution} of a pre-cocontact structure $(M, \xi, \eta)$ is the distribution
\be
\mathcal{V} := \ker\xi \cap \ker\eta \cap \ker(d\eta) \subset \mathbf{T}M.
\ee
We assume $\mathcal{V}$ has constant rank.
\end{definition}

\begin{remark}
    From now on, we will assume that for $l = \operatorname{corank} \mathcal{V} - 2$, we have that $\xi \wedge \eta \wedge (\dd \eta)^l$ is nowhere zero. 
\end{remark}

\begin{definition}[\textsc{Reeb vector fields}]
Let $(M, \xi, \eta)$ be a pre-cocontact manifold. 
%%%
The structure admits two (non unique) characteristic vector fields, often called \textbf{Reeb vector fields}:
\begin{itemize}
    \item A vector field $R_\xi \in \mathfrak{X}(M)$ satisfying
    $$
    i_{R_\xi} \eta = 0 \,, \quad i_{R_\xi} \dd\eta = 0 \,, \quad i_{R_\xi} \xi = 1 \,.
    $$
    \item A vector field $R_\eta \in \mathfrak{X}(M)$ satisfying
    $$
    i_{R_\eta} \xi = 0 \,, \quad i_{R_\eta} \dd\eta = 0 \,, \quad i_{R_\eta} \eta = 1 \,.
    $$
\end{itemize}
\end{definition}

\begin{remark}[\textsc{Cocontact manifold}]
A pre-cocontact manifold is called a \textbf{cocontact manifold} if it is non-degenerate. This means that if the dimension of the manifold is $d=2n+2$, the form
\be
\xi \wedge \eta \wedge (d\eta)^n
\ee
is a volume form on $M$. This structure provides a geometric framework for time-dependent contact mechanics \cite{cocontact}. In a cocontact manifold, Reeb vector fields are unique.
\end{remark}

\begin{theorem}[\textsc{Darboux theorem for cocontact manifolds}]
Let $(M, \xi, \eta)$ be a $(2n+2)$-dimensional cocontact manifold. For every point $m \in M$, there exists a coordinate chart $(U, \varphi)$ centered at $m$, with coordinates $(s, t, q^1, \dots, q^n, p_1, \dots, p_n)$, such that:
\[
\xi|_U = \dd s, \qquad \eta|_U = \dd t - p_a \dd q^a.
\]
Such coordinates are called \textbf{Darboux coordinates} for the cocontact structure.
\begin{proof}
See \cite{cocontact}.
\end{proof}
\end{theorem}

\begin{theorem}[\textsc{Darboux theorem for pre-cocontact manifolds}]
Let $(M, \xi, \eta)$ be a pre-cocontact manifold such that its characteristic distribution $\mathcal{V}$ has constant rank $l$. For every point $m \in M$, there exists a coordinate chart $(U, \varphi)$ centered at $m$, with coordinates $(s, t, q^1, \dots, q^r, p_1, \dots, p_r, z^1, \dots, z^l)$, where $2+2r+l = \dim(M)$, such that:
\be
\xi|_U = \dd s, \qquad \eta|_U = \dd t - p_a \dd q^a.
\ee
In these coordinates, the characteristic distribution is 
\be
\mathcal{V} = \operatorname{span}\left\{\,\frac{\partial}{\partial z^A}\,\right\}_{A=1,\dots,l} \,.
\ee
\begin{proof}
The proof follows by a slight modification of that in \cite{cocontact} for cocontact manifolds.
\end{proof}
\end{theorem}

\subsection{Coisotropic submanifolds}

\begin{definition}[\textsc{Cocontact orthogonal}]
Let $(M, \xi, \eta)$ be a cocontact manifold, and let $W_m \subseteq \mathbf{T}_m M$ be a linear subspace of the tangent space at a point $m \in M$. The \textbf{cocontact orthogonal} of $W_m$ is the subspace
\be
W^{\perp_{(\xi, \eta)}}_m := \left\{\, X \in \operatorname{ker} \,\xi_m \cap \operatorname{ker}\,\eta_m  \colon  d\eta_m(X, Y) = 0 \quad \forall\, Y \in W_m \,\right\}.
\ee
\end{definition}

\begin{definition}[\textsc{Coisotropic submanifold of a cocontact manifold}]
Let $(M, \xi, \eta)$ be a cocontact manifold, and let $\mathfrak{i} \colon N \hookrightarrow M$ be an immersed submanifold. We say that $N$ is a \textbf{coisotropic submanifold} if, for every point $n \in N$, its cocontact orthogonal is contained within its tangent space:
\be
(\mathbf{T}_n N)^{\perp_{(\xi, \eta)}} \subseteq \mathbf{T}_n N.
\ee
\end{definition}

\subsection{Coisotropic embeddings}

\subsubsection{Existence}

\begin{theorem}[\textsc{Cocontact thickenings for pre-cocontact manifolds}]
\label{thm: cocontact_existence}
Let $(M, \xi, \eta)$ be a pre-cocontact manifold. There exists a cocontact manifold $(\widetilde{M}, \widetilde{\xi}, \widetilde{\eta})$ and an embedding
\[
\mathfrak{i} \colon  M \hookrightarrow \widetilde{M}
\]
such that $\mathfrak{i}(M)$ is a coisotropic submanifold of $\widetilde{M}$.
\begin{proof}
Consider the characteristic distribution $\mathcal{V} = \ker\xi \cap \ker\eta \cap \ker(d\eta)$ and an associated almost product structure $P$. 
%%%
The thickening space $\widetilde{M}$ is a tubular neighborhood of the zero section of the vector bundle ${\Lambda^1}^\perp_R(M)$. 
%%%
Let $\tau\colon \widetilde{M} \to M$ be the bundle projection. 
%%%
In a Darboux chart for the pre-cocontact structure, the adapted coordinates on $\widetilde{M}$ are $(s, t, q^a, p_a, z^A, \mu_A)$.
%%%

\noindent We define the new structure $(\widetilde{\xi}, \widetilde{\eta})$ on $\widetilde{M}$ by:
\begin{align}
\widetilde{\xi} &:= \tau^* \xi, \\
\widetilde{\eta} &:= \tau^* \eta + \vartheta^P,
\end{align}
where $\vartheta^P = \mu_A P^A$ is the tautological 1-form. 
%%%
In local coordinates, we have 
\begin{align}
\widetilde{\xi} \,&=\, \dd s \,, \\
\widetilde{\eta} \,&=\, \dd t - p_a \dd q^a + \mu_A P^A \,.
\end{align}
%%%

\noindent The form $\widetilde{\xi}$ is clearly closed. 
%%%
The new $1$-form $\widetilde{\eta}$ is constructed such that the term $d\vartheta^P$ introduces the necessary non-degenerate pairings between the new fiber coordinates $\mu_A$ and the base coordinates $z^A$, which were the source of the degeneracy. 
%%%
A direct computation analogous to the one performed in the case of pre-contact manifolds shows that $\widetilde{\xi} \wedge \widetilde{\eta} \wedge (d\widetilde{\eta})^r$ is non-vanishing in a neighborhood of the zero section, making $(\widetilde{M}, \widetilde{\xi}, \widetilde{\eta})$ a cocontact manifold.
%%%

\noindent The proof that $M$ is a coisotropic submanifold follows by showing that any vector $X \in (\mathbf{T}_m M)^{\perp_{(\widetilde{\xi}, \widetilde{\eta})}}$ must be tangent to $M$, which can be proved, again, by an analogous computation to the one performed in the pre-contact case.
%%%
\end{proof}
\end{theorem}

\subsubsection{Uniqueness}

Uniqueness of coisotropic embeddings in the pre-cocontact case does not hold, similar to the pre-contact and pre-cosymplectic scenario. However, as in the previous cases, topological uniqueness still holds:

\begin{theorem}
    Let $(M, \xi,\eta)$ be a pre-cocontact manifold and let $\mathfrak{i} \colon (M, \xi, \eta) \hookrightarrow (\widetilde M, \widetilde \xi, \widetilde \eta)$ be a coisotropic embedding. Then, it is neighborhood diffeomorphic to the one of \cref{thm: cocontact_existence}.
\end{theorem}

\begin{proof} The proof goes as in the contact case, taking a complement $\mathcal{H}$ to $\mathcal{V} \oplus \langle R_\xi\rangle \oplus \langle R_\eta \rangle$, where $R_\xi$ and $R_\eta$ are arbitrary Reeb vector fields on $M$.
\end{proof}

\begin{remark} Again, assuming that two coisotropic embeddings induce the same pair of Reeb vector fields on $M$ we may show that the diffeomorphism preserves the cocontact structure on $M$.
\end{remark}

Now, as for neighborhood equivalence goes, \cref{ex:cosymplectic_nonuniqueness} and \cref{example:contact_nonuniqueness} may be easily generalized to include a pre-cocontact manifold. Nevertheless, we can still prove a version of uniqueness, when both pairs of Reeb vector fields are tangent and proportional to each other. 

\begin{theorem} Let $(M, \xi, \eta)$ be a pre-cocontact manifold, and let $\mathfrak{i}_i \colon M \rightarrow \widetilde M$, be coisotropic embeddings for two different cocontact structures, $(\widetilde \xi_i, \widetilde \eta_i)$, for $i = 1,2$, such that $\widetilde \xi_1 = \widetilde \xi_2$, $\widetilde \eta_1 = \widetilde \eta_2$, and $\dd \widetilde \eta_1 = \dd  \widetilde \eta_2$ on $M$. Further suppose that both pairs of Reeb vector fields $(R^i_t, R^i_s)$ are proportional pair-wise. Then, the embeddings are neighborhood equivalent to each other. 
\end{theorem}

\begin{proof}

Define $\widetilde \xi_t := t \widetilde \xi_2 +(1-t) \widetilde \xi_1$ and $\widetilde \eta_t := t  \widetilde \eta_2 - \widetilde \eta_1$, for $t \in [0,1]$. Then, reducing neighborhoods further if necessary, we have that $(\widetilde \xi_t, \widetilde \eta_t)$ defines a cocontact manifold, for every $t \in [0,1]$. In order to apply Moser's trick, let us look for a time dependent vector field $X_t \in \mathfrak{X}(\widetilde M_1)$ satisfying the following
\begin{align*}
    &\widetilde \xi_2 - \widetilde \xi_1 + \Lie_{X_t} \widetilde \xi_t = 0\,,\\
    & \widetilde\eta_2 - \widetilde \eta_1+ \Lie_{X_t} \widetilde \eta_t = 0\,.
\end{align*}
Hence, we need to find a solution to
\begin{align*}
    \widetilde \xi_2 - \widetilde \xi_1 + \dd i_{X_t}  \xi_t = 0\,,\\
     \widetilde\eta_2 - \widetilde \eta_1+ \dd i_{X_t} \widetilde \eta_t+ i_{X_t} \dd \widetilde \eta_t = 0\,.
\end{align*}
Notice that the $1$-form $\widetilde \xi_2 - \widetilde \xi_1$ is closed and vanishes on $M$, so that we may apply the relative Poincaré Lemma and (again, reducing neighborhoods if necessary), and write $\widetilde \xi_2 - \widetilde \xi_1 = \dd f$, for certain function $f$ vanishing on $M$. Hence, it suffices to solve 
\begin{align*}
    i_{X_t}\widetilde \xi_t &=  f\,,\\
     \widetilde\eta_2 - \widetilde \eta_1+ \dd i_{X_t} \widetilde \eta_t+ i_{X_t} \dd \widetilde \eta_t &= 0\,.
\end{align*}

\noindent Recall that the corresponding Reeb vector fields are proportional, so that Reeb vector fields of the cocontact structure defined by $(\widetilde \xi_t, \widetilde \eta_t)$ remain proportional to the original. This implies that we may look for $X_t \in \ker \eta_t$ and solve for
\begin{align*}
    i_{X_t}\widetilde \xi_t &=  f\,,\\
      i_{X_t} \dd \widetilde \eta_t &= -\widetilde\eta_2 + \widetilde \eta_1\,,\\
      \widetilde \eta_t(X_t) &= 0\,.
\end{align*}
With a mix of the arguments employed in the cosymplectic and contact scenario, we see that such a vector field exists, and is unique. Furthermore, since the right hand side of all three condition vanishes on $M$, so does $X_t$, which implies that we may appy Moser's Trick and obtain the desired neighborhood equivalence.
\end{proof}

\section{k-Pre-symplectic manifolds}
\label{Sec: k-symplectic manifolds}

To our knowledge, the theory of $k$-symplectic manifolds is usually done by requiring a Darboux theorem, namely the existence of coordinate charts that transform the differential forms defining the structure into forms having constant coefficients. However, we work with a general family of $2$-forms such that their kernels intersect trivially. 
%%%
This allows for more general embeddings, but when working with manifolds which are locally isomorphic to \[\underbrace{\T^\ast Q \oplus \cdots \oplus \T^\ast Q}_{k \text{ times}}\,,\]
up to a kernel, or more general manifolds admitting a Darboux theorem (see \cite{darbouxvarios}), we do not recover another manifold with such a $k$-symplectic structure when performing the thickening procedure developed here.
%%%

\noindent Hence, we prove a general coisotropic embedding theorem for general $k$-pre-symplectic manifolds, at the risk of loosing structure on the thickening. 

\noindent Nevertheless, we deal with the question of classifying the coisotropic embeddings of a $k$-pre-symplectic manifold with certain local model into one where the model is \[\underbrace{\T^\ast Q \oplus \cdots \oplus \T^\ast Q}_{k \text{ times}}\] in the uniqueness section. We mention that, as explained in the introduction, uniqueness in geometries modelling field theories (if ti exists) will be only found topologically, through a diffeomorphism that preserves the structures on the submanifold.

\begin{definition}[\textsc{k-Pre-symplectic manifold}]
\label{Def: k-Pre-symplectic manifold}
A \textbf{k-pre-symplectic manifold} is a pair $(M, \{\omega_1, \dots, \omega_k\})$ where $M$ is a smooth $d$-dimensional manifold, and $\{\omega_1, \dots, \omega_k\}$ is a collection of $k$ closed differential $2$-forms, $\omega_j \in \Omega^2(M)$ for $j=1,\dots,k$.
\end{definition}

\begin{remark}
As in \cref{Sec: Pre-symplectic manifolds}, we will always assume that each form $\omega_j$ has constant rank. 
%%%
Furthermore, we assume that the intersection of the kernels has constant rank.
\end{remark}

\begin{definition}[\textsc{Characteristic distribution}]
The distribution
\be
\mathcal{V} := \bigcap_{j=1}^k \ker \omega_j \subset \mathbf{T}M
\ee
is called the \textbf{characteristic distribution} of the k-pre-symplectic structure. 
%%%
Since each $\omega_j$ is closed, one easily proves that $\mathcal{V}$ is an involutive distribution and thus, it is integrable.
\end{definition}

\begin{remark}[\textsc{k-Symplectic manifold}]
When the intersection of the kernels is trivial, i.e., $\mathcal{V} = \{0\}$, the collection of forms $\{\omega_1, \dots, \omega_k\}$ is said to be non-degenerate, and the pair $(M, \{\omega_1, \dots, \omega_k\})$ is referred to as a \textbf{k-symplectic manifold}.
\end{remark}

\begin{remark} As we mentioned, the notion of $k$-symplectic manifold present in the literature is more restrictive. We prefer this general approach because of the wider possibilities of applications in classical field theories.
\end{remark}

\noindent As we have established, our notion of $k$-symplectic manifold is more general than the one that may be found in the literature. For completeness, we state the Darboux theorem in the latter scenario for both $k$-symplectic and $k$-pre-symplectic manifolds.

\begin{definition}[$k$-\textsc{Lagrangian distribution}]
A distribution $\mathcal{L}$ on a $k$-symplectic manifold $(M, \{\omega_1, \dots, \omega_k\})$ is called {$k$-\bf Lagrangian} if for every $m \in M$
\[
\mathcal{L}|_m = \{v \in \T_m M \colon \,,\omega_j(v, {\mathcal{L}|_m}) = 0\,, \text{ for } j = 1, \dots k\}\,.
\]
\end{definition}

\begin{theorem}[\textsc{Darboux theorem for k-symplectic manifolds}]
\label{Thm: Darboux k-symplectic}
Let $(M, \{\omega_1, \dots, \omega_k\})$ be a k-symplectic manifold of dimension $n + nk$, for certain $k$. If there exists a Lagrangian involutive distribution on $M$ of rank $nk$, then for every point $m \in M$, there exists a coordinate chart $(U, \varphi)$ with coordinates $(q^a, p_a^j)$, with $a = 1, \dots, n$ such that the $k$-symplectic forms read
\be
\omega_j|_U = \sum_{a=1}^n \dd q^a \wedge \dd p_a^j, \quad \text{for } j=1,\dots,k.
\ee
\begin{proof}
See \cite{silvia1,silvia2}.
\end{proof}
\end{theorem}

\begin{remark} \cref{Thm: Darboux k-symplectic} implies that under the assumption that $\dim M = n + nk$ and that there exists a Lagrangian involutive distribution of rank $nk$, $(M, \{\omega_1, \dots, \omega_k\})$ is locally isomorphic to $\underbrace{\T^\ast Q\oplus \cdots \oplus \T^\ast Q}_{k \text{ times}}$, together with its canonical $k$-symplectic structure (see \cite{silvia1, silvia2}). We would like to mention that less restrictive conditions on a $k$-pre-symplectic manifold have been obtained in order to guarantee a Darboux theorem in \cite{darbouxvarios}.
\end{remark}

\noindent In full generality, we can only guarantee the following:

\begin{theorem}[\textsc{Darboux theorem for k-pre-symplectic manifolds}]
\label{Thm: Darboux k-pre-symplectic}
Let $(M, \{\omega_1, \dots, \omega_k\})$ be a k-pre-symplectic manifold, such that its characteristic distribution $\mathcal{V}$ has constant rank $l$.
Then, for every point $m \in M$, there exists a coordinate chart $(U, \varphi)$ around $m$, with coordinates $(x^1, \dots, x^{d-l}, z^1, \dots, z^l)$, such that the characteristic distribution is locally spanned by $\{\frac{\partial}{\partial z^A}\}_{A=1,\dots,l}$, and the 2-forms $\omega_j$ depend only on the coordinates $x^a$:
\be
\omega_j|_U = \omega_j(x).
\ee
Thus, Darboux coordinates for the k-pre-symplectic structure are foliated charts for the foliation associated with the characteristic distribution.
\end{theorem}
\begin{proof}
Since we are assuming that the characteristic distribution is regular, it is completely integrable (as it arises as the kernel of closed forms). It is enough to take a coordinate chart $(U, \varphi)$ around an arbitrary point $m \in M$ such that $(x^1, \dots, x^{d-l}, z^1, \dots, z^{l})$ is a foliated chart for the integrable distribution $\mathcal{V}$. The fact that $\omega_j|_U$ is closed and only depends on $(x^1, \dots, x^{d-l})$ follows from the fact that $\{\frac{\partial}{\partial z^A}\}_{A = 1, \dots, l}$ generate the characteristic distribution.
\end{proof}

\subsection{Coisotropic submanifolds}

\begin{definition}[$\ell$-\textsc{Coisotropic submanifold of a $k$-symplectic manifold}]
Let $(M, \{\omega_1, \dots, \omega_k\})$ be a $k$-symplectic manifold and let $\mathfrak{i} \colon N \hookrightarrow M$ be an immersed submanifold.
%%%
We say that $N$ is a $\ell$-\textbf{coisotropic submanifold} of $(M, \{\omega_j\})$ if, for every point $n \in N$
\be
\left( \mathbf{T}_n N \right)^{{\ell,\,\perp}_{\{\omega\}}} \,\subseteq\, \mathbf{T}_n N \,,
\ee
where the $\ell$-orthogonal is defined as
\be
\left( \mathbf{T}_n N \right)^{{\ell,\,\perp}_{\{\omega\}}} \,=\, \left\{\, X \in \T_nM  \colon  i_X \omega_1(Y) \,=\, ...\,=\, i_X \omega_\ell(Y) \,=\, 0 \,, \;\; \forall \,\, Y \in \T_n N \,\right\} \,.
\ee
\end{definition}

\begin{remark}
Again, coisotropic embeddings are related to the search of minimal embeddings of a $k$-pre-symplectic manifold into a $k$-symplectic manifold. Indeed, in general, the $k$-th orthogonal must contain the characteristic distribution, so that when they are equal (the coisotropic case), there is not a coisotropic embedding into a smaller subspace.

\noindent Notice the difference between the statement in \cref{remark:symplectic_minimal_dimension} and the $k$-symplectic setting. Indeed, we cannot guarantee that a particular $k$-coisotropic embedding is the one with minimal dimension (unlike in the symplectic setting), only that there is no proper subspace which contains the $k$-pre-symplectic subspace and is $k$-symplectic.

\end{remark}

\subsection{Coisotropic embeddings}

\subsubsection{Existence}

\begin{theorem}[\textsc{k-Symplectic thickenings for k-pre-symplectic manifolds}]
\label{Thm: existence k-pre-symplectic}
Let $(M, \{\omega_1, \dots, \omega_k\})$ be a k-pre-symplectic manifold.
There always exists a k-symplectic manifold $(\widetilde{M},\{\widetilde{\omega}_1, \dots, \widetilde{\omega}_k\})$ and an embedding
\be
\mathfrak{i} \colon  M \hookrightarrow \widetilde{M} \,,
\ee
such that $\mathfrak{i}(M)$ is a $k$-coisotropic submanifold of $\widetilde{M}$.
The k-symplectic manifold $(\widetilde{M},\{\widetilde{\omega}_j\})$ is referred to as a \textbf{k-symplectic thickening} of $(M,\{\omega_j\})$.
\begin{proof}
The proof strategy is a direct generalization of the pre-symplectic case. 
%%%

\noindent We consider the characteristic distribution $\mathcal{V} := \bigcap_{j=1}^k \ker \omega_j$, which is assumed to have constant rank, and a complementary distribution $\mathcal{H}$ defining an almost product structure on $M$ via a $(1,\,1)$-tensor $P$ with $\operatorname{Im}(P) = \mathcal{V}$.
%%%
The thickening space $\widetilde{M}$ is constructed as a tubular neighborhood of the zero section of the vector bundle ${\Lambda^1}^\perp_R(M)$ over $M$, where $R = \mathds{1}-P$. 
%%%
Let $\tau\colon\widetilde{M} \to M$ be the bundle projection.
%%%

\noindent The key idea is to apply the construction from the pre-symplectic case to each form $\omega_j$ using the same thickening space $\widetilde{M}$ and the same tautological form $\vartheta^P$ built from the common characteristic distribution $\mathcal{V}$.
%%%
Let $\vartheta^P$ be the tautological 1-form on $\widetilde{M}$ defined as in the pre-symplectic case. In local Darboux coordinates $(x^a, z^A)$ for the foliation and adapted coordinates $(x^a, z^A, \mu_A)$ on $\widetilde{M}$, it reads $\vartheta^P = \mu_A P^A$.
%%%
We then define a collection of $k$ 2-forms on $\widetilde{M}$ as follows:
\be
\widetilde{\omega}_j := \tau^* \omega_j + \dd \vartheta^P, \quad \text{for } j=1, \dots, k.
\ee
Each $\widetilde{\omega}_j$ is closed since both $\tau^* \omega_j$ and $\dd\vartheta^P$ are closed. The term $\dd\vartheta^P$ is independent of $j$ and is precisely the term that "repairs" the degeneracy of all the $\omega_j$ along the directions of $\mathcal{V}$.
%%%
An analogous computation to the one performed in the pre-symplectic case shows that the intersection of the kernels of the new forms is trivial, $\bigcap_{j=1}^k \ker \widetilde{\omega}_j = \{0\}$, in a tubular neighborhood of the zero section. 
%%%
Thus, $(\widetilde{M}, \{\widetilde{\omega}_j\})$ is a $k$-symplectic manifold.
%%%

\noindent Regarding coisotropicity, also in this case, an analogous computation to the one performed in the pre-symplectic case shows that $(M,\, \{\,\omega_1,\,...,\,\omega_k\,\})$ is an $\ell$-coisotropic submanifold of $(\widetilde{M},\, \{\,\widetilde{\omega}_1,\,...,\,\widetilde{\omega}_k\,\})$ for $\ell \,=\, k$.
%%%
\end{proof}
\end{theorem}

\subsubsection{Uniqueness}

In this case, and in the subsequent ones, the uniqueness of coisotropic embeddings is not guaranteed, not even topologically. 
%%%
In order to discuss uniqueness, given the generality of the geometric structures we deal with, we need to restrict to particular cases, where, given the $k$-pre-symplectic manifold $(M, \{\omega_1, \dots, \omega_k\})$ (with certain point-wise structure arising), we demand that its $k$-symplectic thickening $(\widetilde M, \{\widetilde \omega_1, \dots, \widetilde \omega_k\})$ has certain point-wise structure when restricted to $M$.
%%%

\noindent In particular, we focus on two different cases, one in which we find topological uniqueness of the embeddings and the other in which we do not even find topological uniqueness. 
%%%

\noindent Let us begin with the second case.
%%%
Let $Q \hookrightarrow P$ be an embedded submanifold of a given manifold $P$, and define the canonical $k$-symplectic manifold \[\widetilde M := \underbrace{\T^\ast P \oplus \cdots \oplus \T^\ast P}_{k \text{ times}}\,,\] together with the $k$-coisotropic submanifold \[M:= \underbrace{\T^\ast P |_Q \oplus \cdots \oplus \T^\ast P |_Q}_{k \text{ times}}\,.\] Denote by $\{\widetilde\omega_1, \dots, \widetilde \omega_k\}$ the canonical $2$-forms on $\widetilde M$, and by $\{\omega_1, \dots, \omega_k\}$ the canonical $2$-forms on $M$, induced by the natural inclusion 
\[
\mathfrak{i} \colon M \hookrightarrow \widetilde M\,.
\]
We will show that in certain cases, there exists a different embedding 
\[
\mathfrak{i}_2 \colon M \hookrightarrow \widetilde M_2 \,,
\]
where $\widetilde M_2$ is $k$-symplectomorphic to $\widetilde M$ but such that the embedding is not neighborhood equivalent to the canonical $ \mathfrak{i} \colon M \hookrightarrow \widetilde M$. In fact, we may prove a more general statement, characterizing all embeddings topologically, where equivalence is given by diffeomorphism $\psi \colon U \rightarrow U_2$, where $U$ and $U_2$ are neighborhoods of $M$ in $\widetilde M$ and $\widetilde M_2$, respectively; and such that $\psi$ preserves the $k$-symplectic structures on $M$.

\begin{theorem}
\label{thm:k_symplectic_nonuniqueness}
Let $(M, \omega_1, \dots, \omega_k)$ be a $k$-pre-symplectic manifold which is locally $k$-symplectomorphic to $\underbrace{\T^\ast P |_Q \oplus \cdots \oplus \T^\ast P |_Q}_{k \text{ times}}$. Let $\mathcal{V}$ denote its characteristic distribution. Then, coisotropic embeddings up to topological equivalence as described above $M \hookrightarrow \widetilde M$, where the $k$-symplectic structure on $\widetilde M$ is isomorphic to the canonical on $M$, are in bijection with vector bundles
\[
\mathcal{K}^\ast \rightarrow M\,,
\]
such that $\underbrace{\mathcal{K} \oplus \cdots \oplus \mathcal{K}}_{k \text{ times}} \cong \mathcal{V}$.
\end{theorem}
\begin{proof} Let us first show that given a vector bundle $\mathcal{K}^\ast \rightarrow M$ together with an isomorphism of vector bundles $\phi \colon \underbrace{\mathcal{K} \oplus \cdots \oplus \mathcal{K}}_{k \text{ times}} \rightarrow \mathcal{V}$ gives rise to a coisotropic embedding. Indeed, let $\mathcal{H}$ be any complement to $\mathcal{V}$ and identify 
\[\underbrace{\mathcal{K}^\ast \oplus \cdots \oplus \mathcal{K}^\ast}_{k \text{ times}} \cong \mathcal{V}^\ast \cong {\Lambda^1}^\perp (M)\,.\]
Denote by 
$\mathfrak{i}_A \colon \mathcal{K}^\ast \rightarrow \underbrace{\mathcal{K}^\ast \oplus \cdots \oplus \mathcal{K}^\ast}_{k \text{ times}}\,, \quad \mathfrak{i}_A(k) := (0, \cdots,0,k,0, \cdots 0)\,.$
Using the identifactions above, let $\vartheta$ denote the canonical $1$-form on $\underbrace{\mathcal{K}^\ast \oplus \cdots \oplus \mathcal{K}^\ast}_{k \text{ times}}$ and define $\vartheta_A := \mathfrak{i}_A^\ast\vartheta$ together with the forms 
\[\widetilde \omega_A := \omega_A + \dd \vartheta_A\,.\] Then, some local computations as above show that in a neighborhood of $M$ in $\mathcal{K}^\ast$, say $\widetilde M:= U$ the previous forms define a $k$-symplectic structure for which $M$ is $k$-coisotropic. Furthermore, $\{\widetilde \omega_1, \dots, \widetilde \omega_k\}$ define a $k$-symplectic structure such that $(\T_m M, \{\widetilde \omega_1, \dots, \widetilde \omega_k\})$ is $k$-symplectomorphic to the canonical $k$-symplectic structure. 

Conversely, let $\mathfrak{i} \colon M \rightarrow \widetilde M$ be a $k$-coisotropic embedding and define $\mathcal{K} := \T \widetilde M |_M$. Let
\[\phi \colon \mathcal{V} \rightarrow  \underbrace{\mathcal{K}^\ast \oplus \cdots \oplus \mathcal{K}^\ast}_{k \text{ times}}\]
be given by 
\[
\phi(v) := (i_v \widetilde \omega_1, \dots, i_v \widetilde \omega_k)\,.
\]
Using the local model, it is easy to check that this is a well defined vector bundle isomorphism. We now clearly may identify $\mathcal{K}$ as a tubular neighborhood of $M$ in $\widetilde M$, $U$, so that there is a diffeomorphism
\[
\phi \colon \mathcal{K} \rightarrow U\,,
\]
such that $\phi_\ast$ is the identity on $M$. This last condition implies that $\widetilde \omega_1, \dots, \widetilde \omega_k$ coincide with the forms built above on $M$.
\end{proof}

\begin{remark}[\textsc{Counting solutions via K-theory}]
\label{rem:counting_solutions_ktheory}
\cref{thm:k_symplectic_nonuniqueness} frames the classification of topologically non-equivalent coisotropic embeddings as a problem of finding "k-th roots" of the vector bundle. The subject can be said to begin with A. Grothendieck in 1957, and developed in terms of vector bundles by M. Atiyah and F. Hirzebruch (see \cite{karoubi,atiyah}).
%%%
The natural setting to analyze this question is topological K-theory, and the fundamental tool is the so-called \textbf{Grothendieck group} of real vector bundles over $M$. 
%%%

\noindent The set of isomorphism classes of real vector bundles over $M$, which we denote by $\mathrm{Vect}(M)$, forms a commutative semigroup under the direct sum operation $\oplus$. The \textbf{Grothendieck group} $K_0(M)$ is the group completion of this semigroup. Its elements are equivalence classes of formal pairs of vector bundles $(E, F)$, called \textbf{virtual bundles}, which are thought of as formal differences $[E] - [F]$. Two pairs $(E,F)$ and $(G,H)$ are equivalent if there exists a trivial vector bundle $L$ over $M$ such that $E \oplus H \oplus L \cong F \oplus G \oplus L$.
%%%
The addition in this group is defined by $([E] - [F]) + ([G] - [H]) := [E \oplus G] - [F \oplus H]$. Any genuine vector bundle $\mathcal{K} \in \mathrm{Vect}(M)$ is represented in $K_0(M)$ by the class of the pair $(\mathcal{K}, 0)$, which we denote simply by $[\mathcal{K}]$.
%%%
With this algebraic structure, the geometric condition on the vector bundle $\mathcal{K}$,
\[
\mathcal{V} \cong \mathcal{K} \oplus \cdots \oplus \mathcal{K} \quad (k \text{ times})\,,
\]
is translated into the algebraic equation within the ring $K_0(M)$:
\be \label{Eq: K-theory}
[\mathcal{V}] = k \cdot [\mathcal{K}]\,,
\ee
where the multiplication by an integer $k$ arises from the repeated addition of $[\mathcal{K}]$ to itself.
%%%

\noindent Since the existence of at least one embedding is guaranteed by \cref{Thm: existence k-pre-symplectic}, let us denote by $\mathcal{K}_0$ a vector bundle corresponding to one solution. Then its class $[\mathcal{K}_0]$ is a solution to \eqref{Eq: K-theory}. If $\mathcal{K}_1$ is another such bundle, its class $[\mathcal{K}_1]$ must also satisfy \eqref{Eq: K-theory}, which implies
\[
k \cdot [\mathcal{K}_0] = k \cdot [\mathcal{K}_1] \quad \implies \quad k \cdot \left( [\mathcal{K}_1] - [\mathcal{K}_0] \right) = 0\,.
\]
%%%
This means that the difference between any two solutions, $[\mathcal{K}_1] - [\mathcal{K}_0]$, must be an element of the so-called \textbf{k-torsion subgroup} of $K_0(M)$, which is precisely the set of solutions $[\mathcal{K}]$ of the algebraic equation
\[
k \cdot [\mathcal{K}] \,=\, 0 \,.
\]
%%%

\noindent Consequently, the number of non-equivalent vector bundles $\mathcal{K}$ satisfying the hypothesis (and thus, the number of non-equivalent embeddings) is given by the cardinality of the $k$-torsion subgroup of $K_0(M)$. 
%%%
This number is a well-known topological invariant of the manifold $M$. 
%%%
\end{remark}

\noindent The subsequent example for $M \cong \mathbb{S}^1 \times \mathbb{R}^{2}$ illustrates this principle, as the two solutions (the orientable and non-orientable line bundles) correspond to the two elements of the 2-torsion subgroup of $K_0(M)$.

\begin{example}
\label{example:non_uniqueness_k_symplectic}
Let $Q = \mathbb{S}^1 \hookrightarrow P = \mathbb{R}^2$ and $k = 2l$ be even. Then, \[M = \underbrace{\T^\ast P|_Q \oplus \cdots \oplus \T^\ast P|_Q}_{2l \text{ times}}\] is a $2l$-coisotropic submanifold and the characteristic distribution $\mathcal{V}$ is a trivial bundle over $M$. Is not hard to check that $\operatorname{rank} \mathcal{V} = 2l$, and using that $M \cong \mathbb{S}^1 \times \mathbb{R}^{2l}$, we can see that there are two vector bundles $\mathcal{K} \rightarrow M$ such that $\mathcal{V} \cong \mathcal{K}^\ast \oplus \cdots \oplus \mathcal{K}^\ast$, an orientable (trivial) line bundle and a non-orientable line bundle.
\end{example}

\noindent Now let us move on to an example where we can find topological uniqueness. 
%%%
The local model will be the following. Let $Q$ be a manifold, and let $\mathcal{D} \subseteq \T Q$ be a completely integrable distribution. Define $\widetilde M := \underbrace{\T^\ast Q \oplus \cdots \oplus \T^\ast Q}_{k \text{ times}}$ to be the canonical $k$-symplectic manifold, and let 
$ M := \underbrace{\mathcal{D}^\circ \oplus \cdots \oplus \mathcal{D}^\circ}_{k \text{ times}} \subseteq \widetilde M$ the be the $k$-coisotropic submanifold. Then, the canonical embedding is the unique possible coisotropic embedding, up to a diffeomorphism that preserves the form on $M$.

\begin{theorem} Let $(M, \omega_1, \dots, \omega_k)$ be a $k$-pre-symplectic manifold that is locally $k$-symplectomorphic to $\mathcal{D}^\circ \oplus \cdots \oplus \mathcal{D}^\circ$. Then, all $k$-coisotropic embeddings of $M$ into a $k$-symplectic manifold locally $k$-symplectomorphic to the canonical one are neighborhood equivalent to each other.
\end{theorem}
\begin{proof} Let $M \hookrightarrow \widetilde M$ be a $k$-coisotropic embedding and let $\mathcal{V}$ be the characteristic distribution on $M$. It is easy to check that the map 
\[
\phi\colon\T \widetilde M \big /\T M \longrightarrow \underbrace{\mathcal{V^\ast} \oplus \cdots \oplus \mathcal{V^\ast}}_{k \text{ times}}\,, 
\]
given by $\phi([v]) := (i_v \widetilde \omega_1, \dots, i_v \widetilde\omega_k)$ is well defined and defines an isomorphism (essentially restricting to the local case). Now, as usual, to end the proof we choose a diffeomorphism between two neighborhoods that coincides with $\phi$ on $M$, so that we may assume that we have two different k-symplectic structures $\{\widetilde \omega_1^{(1)}, \dots, \widetilde \omega_k^{(1)}\}$ and $\{\widetilde \omega_1^{(2)}, \dots, \widetilde \omega_k^{(2)}\}$ defined on a neighborhood $U$ of $M$, and such that they coincide on $M$.

\end{proof}

\section{\texorpdfstring{$k$}{k}-Pre-cosymplectic manifolds}
\label{Sec: k-cosymplectic manifolds}

As in the $k$-symplectic setting, we deal with more general structures that one may find in the literature. For the sake of completeness, we cite the Darboux theorem for the more restrictive version.

\begin{definition}[\textsc{$k$-Pre-cosymplectic manifold}]
A \textbf{$k$-pre-cosymplectic manifold} is a tuple $$(M, \{\eta_1, \dots, \eta_k\}, \{\omega_1, \dots, \omega_k\})\,,$$ where $M$ is a smooth manifold, $\eta_1, \dots, \eta_k \in \Omega^1(M)$ are $k$ closed differential 1-forms, and $\omega_1, \dots, \omega_k \in \Omega^2(M)$ are $k$ closed differential 2-forms.
\end{definition}

\begin{definition}[\textsc{Characteristic distribution}]
The \textbf{characteristic distribution} of a $k$-pre-cosymplectic structure is the distribution
\be
\mathcal{V} \,:=\, \left( \bigcap_{i=1}^k \ker\,\eta_i \right) \cap \left( \bigcap_{j=1}^k \ker\,\omega_j \right).
\ee
\end{definition}

\begin{remark}
As in the previous sections, we will always assume that all forms have constant rank, that the $1$-forms $\{\eta_1, \dots, \eta_k\}$ are nowhere zero, and that their common characteristic distribution is a smooth distribution of constant rank.
\end{remark}

\begin{remark}[\textsc{$k$-Cosymplectic manifold}]
A $k$-pre-cosymplectic manifold is a \textbf{$k$-cosymplectic manifold} when its characteristic distribution is trivial, i.e., $\mathcal{V} = \{0\}$. This structure appears in the description of certain classical field theories and time-dependent systems with multiple conserved quantities \cite{silvia2}.
\end{remark}

Before citing the Darboux theorem found in the literature, let us first introduce what we mean by a Lagrangian distribution. A distribution $\mathcal{L}$ on a $k$-cosymplectic manifold $(M, \{\eta_1, \dots, \eta_k\}, \{\omega_1, \dots, \omega_k\})$ is called {\bf $k$-Lagrangian} if for every $m \in M$, 
\[
\mathcal{L}|_m = \{v \in \T_m M \colon \, \eta_j(v) = \omega_j(v, \mathcal{L}|_m) = 0\,, \text{ for }j = 1, \dots, k\}\,.
\]

\begin{theorem}[\textsc{Darboux theorem for $k$-cosymplectic manifolds}]
\label{thm: Darboux_for_k_cosymplectic}
Let $(M, \{\eta_i\}, \{\omega_j\})$ be a $k$-cosymplectic manifold of dimension $d = k + n + nk$. Assume that there is an involutive Lagrangian distribution of rank $nk$ and that $\eta_1 \wedge \cdots \wedge \eta_k$ defines a volume form on $\bigcap_{j = 1}^k \ker\, \omega_j$ (so in particular $\bigcap_{j = 1}^{k} \ker\, \omega_k$ has rank $k$).
%%%
Then, for every point $m \in M$, there exists a coordinate chart $(U, \varphi)$ centered at $m$, with coordinates $(t^1, \dots, t^k, q^{a}, p_{a}^j)_{j=1,\dots,k; a=1,\dots,n}$, such that:
\[
\eta_i|_U = \dd t^i, \qquad \omega_j|_U = \dd q^{a} \wedge \dd p_{a}^j\,.
\]
%%%
We will call such coordinates \textbf{Darboux coordinates} for the $k$-cosymplectic structure.

\begin{proof}
See \cite{silvia2}.
\end{proof}
\end{theorem}

\begin{remark} The canonical example of $k$-cosymplectic manifold with the previous local structure is $\mathbb{R}^k \times \underbrace{\T^\ast Q \oplus \cdots \oplus \T^\ast Q}_{k \text{ times}}$. Again, a more general result can be found in \cite{darbouxvarios}.
\end{remark}

\begin{remark}
\label{remark:Reeb_vector_fields_k_cosymplectic}
Under the hypotheses of \cref{thm: Darboux_for_k_cosymplectic}, a family of \textbf{Reeb vector fields}, $R_1, \dots, R_k$, can be defined uniquely by
\[
i_{R_i} \eta_j = \delta_{ij}\,, \quad i_{R_i} \omega_j = 0\,.
\]
In Darboux coordinates, these vector fields read $R_i = \frac{\partial}{\partial t^i}$. 
%%%
We would like to mention that Reeb vector fields are not well defined in general.
\end{remark}

As for $k$-pre-symplectic manifolds in general, one can only prove the following:

\begin{theorem}[\textsc{Darboux theorem for $k$-pre-cosymplectic manifolds}] 
Let $(M, \{\eta_1, \dots,\eta_k\}, \{\omega_1, \dots, \omega_k\})$ be a $k$-pre-cosymplectic manifold, such that its characteristic distribution $\mathcal{V}$ has constant rank $l$. Then, for every point $m \in M$, there exists a coordinate chart $(U, \varphi)$ around $m$, iwht coordinates $(x^1, \dots, x^{d-l}, z^1, \dots, z^l)$, such that the characteristic distribution is locally spanned by $\{\frac{\partial}{\partial z^A}\}_{A = 1, \dots, l}$ and the $1$-forms $\eta_j$ and the $2$-forms $\omega_j$ only depend on the coordinates $x^a$.
\end{theorem}
\begin{proof} The same strategy as in \cref{Thm: Darboux pre-symplectic} works.
\end{proof}

\subsection{Coisotropic submanifolds}

\begin{definition}[$\ell$-\textsc{Coisotropic submanifold of a $k$-cosymplectic manifold}] \label{Def: coisotropic submanifold k-cosymplectic}
Let $(M, \{\eta_i\}, \{\omega_j\})$ be a $k$-cosymplectic manifold, and let $\mathfrak{i} \colon 
N \hookrightarrow M$ be an immersed submanifold. For an integer $\ell \in \{1, \dots, k\}$, we say that $N$ is an \textbf{$\ell$-coisotropic submanifold} if, for every point $n \in N$,
\[
{\mathbf{T}_n N}^{\ell, \perp} \,\subseteq\, \mathbf{T}_n N\,,
\]
where the \textbf{$\ell$-$k$-cosymplectic orthogonal} is defined as
\[
{\mathbf{T}_n N}^{\ell, \perp} := \left\{\, X \in \mathbf{T}_n M \mid \eta_i(X)=0, \, \omega_j(X, Y) = 0, \; \forall i,j \in \{1,\dots,\ell\}, \; \forall Y \in \mathbf{T}_n N \,\right\}.
\]
\end{definition}

\subsection{Coisotropic embedding}

\subsubsection{Existence}

\begin{theorem}[\textsc{$k$-Cosymplectic thickenings for $k$-pre-cosymplectic manifolds}]
Let $(M, \{\eta_i\}, \{\omega_j\})$ be a $k$-pre-cosymplectic manifold.
There exists a $k$-cosymplectic manifold $(\widetilde{M}, \{\widetilde{\eta}_i\}, \{\widetilde{\omega}_j\})$ and an embedding
\be
\mathfrak{i} \colon  M \hookrightarrow \widetilde{M} \,,
\ee
such that $\mathfrak{i}(M)$ is a $k$-coisotropic submanifold of $\widetilde{M}$.
\begin{proof}
Consider the characteristic distribution $\mathcal{V}$ and a complementary distribution $\mathcal{H}$ defining an almost product structure via a projector $P$ with $\operatorname{Im}(P) = \mathcal{V}$.
%%%
The thickening space $\widetilde{M}$ is a tubular neighborhood of the zero section of the vector bundle ${\Lambda^1}^\perp_R(M)$, where $R = \mathds{1}-P$. Let $\tau\colon \widetilde{M} \to M$ be the bundle projection.
%%%

\noindent As in the $k$-pre-symplectic case, we define a single tautological 1-form $\vartheta^P$ on $\widetilde{M}$ and the new $k$-cosymplectic structure is then defined by pulling back all the original forms and applying the same regularization term for the $\omega_j$:
\begin{align}
\widetilde{\eta}_i &:= \tau^* \eta_i \,, \quad \text{for } i=1, \dots, k, \\
\widetilde{\omega}_j &:= \tau^* \omega_j + \dd\vartheta^P, \quad \text{for } j=1, \dots, k.
\end{align}
All new forms are closed by construction. The correction terms $\dd\vartheta^P$ are designed to introduce non-degeneracy along the directions of $\mathcal{V}$. An analogous computation to the previous sections confirms that the resulting structure $(\widetilde{M}, \{\widetilde{\eta}_i\}, \{\widetilde{\omega}_j\})$ has a trivial characteristic distribution in a neighborhood of the zero section, and is thus a $k$-cosymplectic manifold.
%%%

\noindent Also the proof of $k$-coisotropicity follows from an analogous computation to the ones made in the previous sections by taking into account \cref{Def: coisotropic submanifold k-cosymplectic}.
\end{proof}
\end{theorem}

\subsubsection{Uniqueness}
Uniqueness in the $k$-pre-cosymplectic setting is not guaranteed, as in the $k$-pre-symplectic scenario, and the same counter example may be easily generalized to include the previous case:

%%%
\noindent Let $Q \hookrightarrow P$ be a submanifold and define the canonical $k$-cosymplectic manifold \[\widetilde M := \mathbb{R}^k \times\underbrace{\T^\ast P \oplus \cdots \oplus \T^\ast P}_{k \text{ times}}\] with the $k$-coisotropic submanifold defined as\[M := \mathbb{R}^k \times \underbrace{\T^\ast P|_Q \oplus \cdots \oplus \T^\ast P|_Q}_{k \text{ times}} \subset \widetilde M\,.\] We would like to study the possible $k$-coisotropic embeddings of $M$ or, more generally, the possible $k$-coisotropic embeddings of a $k$-cosymplectic manifold $(M, \{\eta_j\}, \{\omega_j\})$ that is point-wise isomorphic to that of $\mathbb{R}^k \times \underbrace{\T^\ast P|_Q \oplus \cdots \oplus \T^\ast P|_Q}_{k \text{ times}}$.

\begin{theorem} Let $(M, \{\eta_j\}, \{\omega_j\})$ be a $k$-cosymplectic manifold which is point-wise $k$-cosymplectomorphic to $\mathbb{R}^k \times \underbrace{\T^\ast P|_Q \oplus \cdots \oplus \T^\ast P|_Q}_{k \text{ times}}$. Denote by $\mathcal{V}$ its characteristic distribution. Then, coisotropic embeddings up to neighborhood diffemorphism preserving the structure on $M$, $\mathfrak{i} \colon M \hookrightarrow \widetilde M$ into a $k$-cosymplectic manifold $(\widetilde M , \{\widetilde \eta\}, \{\widetilde \omega_j\})$ which is isomorphic to the canonical $k$-cosymplectic structure on $M$ are characterized by vector bundles $\mathcal{K} \rightarrow M$ such that 
\[
\underbrace{\mathcal{K}^\ast \oplus \cdots \mathcal{K}^\ast}_{k \text{ times}} \cong \mathcal{V}\,.
\]
\end{theorem}
\begin{proof} The proof follows as in the $k$-symplectic case. The $1$-forms $\widetilde \eta_j$ are just given by the pullback of $\eta_j$.
\end{proof}

\begin{example} Now it is easy to generalize \cref{example:non_uniqueness_k_symplectic} to the $k$-cosymplectic setting by defining $Q:= \mathbb{S}^1\hookrightarrow P := \mathbb{R}^2$, with $k = 2l$, showing that we do not have uniqueness of embeddings.
\end{example}

\noindent Again, following the same ideas of the $k$-pre-symplectic case, we may prove uniqueness in certain particular cases. 
%%%
However, in the case of $k$-cosymplectic geometry, as in cosymplectic, we need to be careful with the orbits of the Reeb vector fields. In full generality, we may only prove that any two embeddings with the local structure mentioned in the statement of the following theorem are topologically equivalent.

\begin{theorem} Let $(M, \eta_1, \dots, \eta_k, \omega_1, \dots, \omega_k)$ be a $k$-pre-cosymplectic manifold which is locally $\mathcal{D}^\circ \oplus \dots \oplus \mathcal{D}^\circ \times \mathbb{R}^k$, for some completely integrable distribution $\mathcal{D}$ on certain $P$. Then any pair of coisotropic embeddings $M \hookrightarrow \widetilde M$ into a $k$-cosymplectic manifold which is point-wise isomorphic to $\T^\ast P \oplus \cdots \oplus \T^\ast P \times \mathbb{R}^k$ on $M$ are diffeomorphic.
\end{theorem}

\noindent For completeness, let us mention that employing the same tecniques as in the cosymplectic and $k$-symplectic we may guarantee that the diffeomorphism preserves the structure on $M$, provided that the Reeb vector fields (defined in \cref{remark:Reeb_vector_fields_k_cosymplectic}) are equal on $M$. So that we may obtain the following refinement:

\begin{theorem} Let $(M, \eta_1, \dots, \eta_k, \omega_1, \dots, \omega_k)$ be a pre-$k$-cosymplectic manifold which is locally $\mathcal{D}^\circ \oplus \dots \oplus \mathcal{D}^\circ \times \mathbb{R}^k$, coisotropically embedded into $\widetilde M$, equipped with two different $k$-cosymplectic structures, $(\widetilde \eta_1^{(1)}, \dots, \widetilde \eta_k^{(1)}, \widetilde \omega_1^{(1)}, \dots, \widetilde \omega_k^{(1)})$ and $(\widetilde \eta_1^{(2)}, \dots, \widetilde \eta_k^{(2)}, \widetilde \omega_1^{(2)}, \dots, \widetilde \omega_k^{(2)})$, which are point-wise isomorphic to the canonical $k$-cosymplectic structure, equal on $M$ and such that the induced Reeb vector fields are equal on $M$. Then, they are diffeomorphic through a diffeomorphism that preserves the structures on $M$.
\end{theorem}

\section{\texorpdfstring{$k$}{k}-Pre-contact manifolds}
\label{Sec: k-contact manifolds}

\begin{definition}[\textsc{k-Pre-contact manifold}]
A \textbf{k-pre-contact manifold} is a pair $(M, \{\eta_1, \dots, \eta_k\})$, where $M$ is a smooth manifold and $\eta_1, \dots, \eta_k \in \Omega^1(M)$ are $k$ differential 1-forms.
\end{definition}

\begin{remark}
Throughout this section, we will assume that the distributions $\mathcal{D}_j := \ker(\eta_j)$ and their intersection $\mathcal{D} := \bigcap_{j=1}^k \ker(\eta_j)$ have constant rank. We also assume that the restriction of each $d\eta_j$ to $\mathcal{D}$ has constant rank.
\end{remark}

\begin{definition}[\textsc{Characteristic distribution}]
Let $(M, \{\eta_j\})$ be a k-pre-contact manifold. We define its \textbf{characteristic distribution} as
\[
\mathcal{V} \,:=\, \left( \bigcap_{j=1}^k \ker \eta_j \right) \cap \left( \bigcap_{j=1}^k \ker(d \eta_j|_{\bigcap_{l=1}^k \ker \eta_l}) \right) \,.
\]
\end{definition}

\begin{remark}[\textsc{k-Contact manifold}]
When $\mathcal{V} \,=\, \left\{\,0\,\right\}$, or, equivalently, the set of forms $\{\eta_1, \dots, \eta_k\}$ is such that
\[
\eta_1 \wedge \dots \wedge \eta_k \wedge d\eta_1 \wedge \dots \wedge d\eta_k \,\neq\, 0
\]
everywhere on $M$, the collection defines a \textbf{k-contact structure}, and the pair $(M, \{\eta_j\})$ is called a \textbf{k-contact manifold}.
\end{remark}

Again, we prefer a more general approach but, for the sake of completeness, we add here the Darboux theorem found in the literature. 

\begin{definition}[$k$-\textsc{Lagrangian distribution}]
A distribution $\mathcal{L}$ on $M$ is called {\bf $k$-Lagrangian} if for every $m \in M$,
\[
\mathcal{L}|_m = \{v \in \T_m M \colon \, \eta_j(v) = \dd \eta_j(v, \mathcal{L}|_m) = 0\, \text{ for }j = 1, \dots, k\}\,.
\]
\end{definition}

\begin{theorem}[\textsc{Darboux theorem for k-contact manifolds}]\label{thm:Darboux_k_contact}
Let $(M, \{\eta_j\})$ be a k-contact manifold of dimension $d = k + n + nk$. Suppose there is a $k$-Lagrangian distribution of rank $nk$ and that $\eta_1 \wedge \cdots \wedge \eta_k$ defines a volume form on $\bigcap_{j= 1}^{k} \ker \dd\eta_j$.
%%%
Then, for every point $m \in M$, there exists a coordinate chart $(U, \varphi)$ centered at $m$, with coordinates $(t^1, \dots, t^k, q^{a}, p_{aj})_{j=1,\dots,k; a=1,\dots,n_j}$, such that
\[
\eta_j|_U = \dd t^j - p_{a}^j \dd q^{a} \quad (\text{no sum over } j).
\]
We will call such coordinates \textbf{Darboux coordinates} for the k-contact structure.
\begin{proof}
See \cite{narciso_nuevo}.
\end{proof}
\end{theorem}

\begin{remark} The canonical model for these coordinates is $M := \mathbb{R}^k \times \underbrace{\T^\ast Q \oplus \cdots \oplus \T^\ast Q}_{k \text{ times}}$.
\end{remark}

\begin{remark} As in the $k$-cosymplectic scenario, for $k$-contact manifolds under the hypotheses of \cref{thm:Darboux_k_contact}, there is a distinguished family of vector fields, called {\bf Reeb vector fields}, $R_1, \dots, R_k$ defined as the unique vector fields satisfying 
\[
i_{R_i} \eta_j = \delta_{ij}\,, \quad i_{R_i} \dd \eta_j = 0\,.
\]
Locally, in Darboux coordinates, $R_i = \frac{\partial}{\partial t^i}$.
\end{remark}

Again, in full generality, we may only guarantee the following:

\begin{theorem}[\textsc{Darboux theorem for k-pre-contact manifolds}]

Let $(M, \{\eta_j\})$ be a $k$-pre-contact manifold, such that its characteristic distribution $\mathcal{V}$ has constant rank $l$. Then, for every point $m \in M$, there exists a coordinate chart $(U, \varphi)$ around $m$, with cordinates $(x^1, \dots, x^{d-l}, z^1, \dots, z^{l})$ such that the characteristic distribution is locally spanned by $\{\frac{\partial}{\partial z^A}\}_{A= 1, \dots, l}$, and such that the $1$-forms $\eta_j$ only depend on the coordinates $x^a$.
\begin{proof}
The proof follows as a direct generalization of the proof of \cref{Thm: Darboux pre-symplectic}.
\end{proof}
\end{theorem}

\subsection{Coisotropic submanifolds}

\begin{definition}[$\ell$-\textsc{Coisotropic submanifold of a k-contact manifold}]
Let $(M, \{\eta_1, \dots, \eta_k\})$ be a k-contact manifold, and let
\[
\mathfrak{i}  \colon  N \hookrightarrow M
\]
be an immersed submanifold. For an integer $\ell \in \{1, \dots, k\}$, we say that $N$ is an \textbf{$\ell$-coisotropic submanifold} of $(M, \{\eta_j\})$ if, for every point $n \in N$,
\[
{\mathbf{T}_n N}^{\ell, \perp_{\{\eta\}}} \,\subseteq\, \mathbf{T}_n N\,,
\]
where the \textbf{$\ell$-k-contact orthogonal} of the tangent space $\mathbf{T}_n N$ is defined as the subspace
\[
{\mathbf{T}_n N}^{\ell, \perp_{\{\eta\}}} := \left\{\, X \in \bigcap_{j=1}^\ell \ker(\eta_j)_n  \colon  d\eta_j(X, Y) = 0 \quad \forall j \in \{1,\dots,\ell\}, \quad \forall \, Y \in \mathbf{T}_n N \,\right\}.
\]
\end{definition}

\subsection{Coisotropic embedding}

\subsubsection{Existence}

\begin{theorem}[\textsc{k-Contact thickenings for k-pre-contact manifolds}]
Let $(M, \{\eta_1, \dots, \eta_k\})$ be a k-pre-contact manifold.
%%%
There exists a k-contact manifold $(\widetilde{M}, \{\widetilde{\eta}_1, \dots, \widetilde{\eta}_k\})$ and an embedding
\[
\mathfrak{i} \colon  M \hookrightarrow \widetilde{M}
\]
such that $\mathfrak{i}(M)$ is a k-coisotropic submanifold of $\widetilde{M}$.
%%%
The k-contact manifold $(\widetilde{M}, \{\widetilde{\eta}_j\})$ is referred to as a \textbf{k-contact thickening} of $(M, \{\eta_j\})$.
%%%
\begin{proof}

\noindent Consider the characteristic distribution $\mathcal{V}$ of the k-pre-contact structure and an associated almost product structure $P$ with $\operatorname{Im}(P) = \mathcal{V}$.
%%%
The thickening space $\widetilde{M}$ is constructed as a tubular neighborhood of the zero section of the vector bundle ${\Lambda^1}^\perp_R(M)$, where $R = \mathds{1}-P$. Let $\tau\colon \widetilde{M} \to M$ be the bundle projection. As in the previous sections, the adapted coordinates on $\widetilde{M}$ include the base coordinates from $M$ and a single set of fiber coordinates, which we denote by $\mu_A$.
%%%

\noindent On this space, we define a single tautological 1-form $\vartheta^P := \mu_A P^A$. The new k-contact structure on $\widetilde{M}$ is then defined by applying the same correction to each pre-contact form:
\[
\widetilde{\eta}_j := \tau^* \eta_j - \vartheta^P, \quad \text{for } j=1, \dots, k.
\]
By taking the exterior derivative, we get $d\widetilde{\eta}_j = \tau^* d\eta_j - d\vartheta^P$. The term $-d\vartheta^P$, common to all forms, is designed to resolve the degeneracy along the characteristic distribution $\mathcal{V}$. A direct computation, analogous to the one for a single pre-contact form, shows that the resulting structure $(\widetilde{M}, \{\widetilde{\eta}_j\})$ is non-degenerate in a neighborhood of the zero section, thus defining a k-contact manifold.
%%%

\noindent To prove that $M$ is a k-coisotropic submanifold, we must show that $\mathbf{T}_m M^{\perp_{\{\widetilde{\eta}\}}} \subset \mathbf{T}_m M$, which can be proved via a computation analogous to the one made in the pre-contact case.
%%%
\end{proof}
\end{theorem}

\subsubsection{Uniqueness}
Much in vain of the $k$-pre-cosymplectic case, and employing the same tecniques developed, uniqueness is not garanteed, and a generalization of the example described in the $k$-symplectic case works.

\noindent Nevertheless, uniqueness may follow by requiring the thickening $\widetilde M$ to have certain geometry. Indeed, we have the following

\begin{theorem} Let $(M,\eta_1, \dots, \eta_k)$ be a $k$-pre-contact manifold locally diffeomorphic to $\mathcal{D}^\circ \oplus \cdots \oplus \mathcal{D}^\circ \times \mathbb{R}^k$, for certain completely integrable distribution $\mathcal{D}$ on a particular manifold $P$. Then, any pair of coisotropic embeddings $M \hookrightarrow \widetilde M$ into a $k$-contact manifolds which are locally diffeomorphic to $\T^\ast P \oplus \cdots \oplus \T^\ast P \times\mathbb{R}^k$ are diffeomorphic through a diffeomorphism preserving the structures on $M$.
\end{theorem}

\noindent Again, this theorem may be refined if we take into account the orbits of the Reeb vector fields:

\begin{theorem}Let $(M,\eta_1, \dots, \eta_k)$ be a $k$-pre-contact manifold locally diffeomorphic to $\mathcal{D}^\circ \oplus \cdots \oplus \mathcal{D}^\circ \times \mathbb{R}^k$, for certain completely integrable distribution $\mathcal{D}$ on a particular manifold $P$. Let $M$ be coisotropically embedded into $\widetilde M$, where $\widetilde M$ is endowed with two different $k$-contact structures, $(\widetilde \eta_1^{(1)}, \dots, \widetilde \eta_k^{(1)})$ and $(\widetilde \eta_1^{(2)}, \dots, \widetilde \eta_k^{(2)})$ such that the forms and their exterior differentials coincide on $M$. If the Reeb vector fields are proportional, both embeddings are neighborhood equivalent.
\end{theorem}

\section{Pre-multisymplectic manifolds}
\label{Sec: multisymplectic manifolds}

\begin{definition}[\textsc{Pre-multisymplectic manifold}]
A \textbf{pre-multisymplectic manifold} is a pair $(M, \omega)$, where $M$ is a smooth $d$-dimensional manifold and $\omega \in \Omega^{k+1}(M)$ is a closed differential $(k+1)$-form.
\end{definition}

\begin{remark}
We will always assume that $\omega$ has \textbf{constant rank}, in the sense that the map
\[
\omega^\flat_m \colon \mathbf{T}_m M \longrightarrow \Lambda^k \mathbf{T}^*_m M \;:\; X \mapsto i_X \Omega_m
\]
has image of constant dimension as $m$ varies over $M$. 
%%%
That is, $\dim(\operatorname{Im}\, \omega^\flat_m)$ is independent of $m \in M$.
\end{remark}

\begin{definition}[\textsc{Characteristic distribution}]
The \textbf{characteristic distribution} of a pre-multisymplectic form $\omega$ is defined by
\[
\mathcal{V} := \ker \omega \,.
\]
%%%
\end{definition}

\begin{remark}[\textsc{Multisymplectic manifold}]
If $\omega^\flat$ is injective (i.e., $\ker \omega^\flat = \{0\}$), then $\omega$ is said to be \textbf{non-degenerate}, and the pair $(M, \omega)$ is called a \textbf{multisymplectic manifold} (or sometimes \textbf{$(k+1)$-plectic}).
%%%
In this case, $\omega$ defines a fiberwise isomorphism
\[
\omega^\flat \colon \mathbf{T}M \to \Lambda^k \mathbf{T}^* M.
\]
\end{remark}

\subsection{Coisotropic submanifolds}

\begin{definition}[\textsc{$\ell$-coisotropic submanifold of a multisymplectic manifold}]
Let $(M, \omega)$ be a multisymplectic manifold.
%%%
Let $N \subset M$ be an immersed submanifold, and let $\ell \in \{1, \dots, k\}$.
%%%
We say that $N$ is an \textbf{$\ell$-coisotropic submanifold} of $(M, \omega)$ if, for every point $n \in N$,
\[
\left( \Lambda^\ell T_n N \right)^{\perp_\omega} \subseteq \Lambda^\ell T_n N,
\]
where the $\ell$-orthogonal is defined as
\[
\left( \Lambda^\ell T_n N \right)^{\perp_\omega} := \left\{\, X \in T_n M \;:\; i_X \omega_n(\xi) = 0 \quad \forall\, \xi \in \Lambda^\ell T_n N \,\right\}.
\]
\end{definition}

\begin{remark}
When $k\, =\, 1$, the notion of $\ell$-coisotropic submanifold coincides with the usual definition of coisotropic submanifold in symplectic and pre-symplectic geometry.
%%%
\end{remark}

\subsection{Coisotropic embeddings}

\subsubsection{Existence}
\label{subsection:multisymplectic_existence}
\begin{theorem}[\textsc{Multisymplectic thickenings of pre-multisymplectic manifolds}] \label{Thm: multisymplectic thickening}
Let $(M,\, \omega)$ be a pre-multisymplectic manifold, with $\omega$ a closed $k$-form on $M$.
%%%
Assume $\operatorname{ker}\,\omega$ to be of constant rank and integrable.
%%%
There always exists a multisymplectic manifold $(\widetilde{M},\, \widetilde{\omega})$ and an embedding
\be
\mathfrak{i}  \colon  M \hookrightarrow \widetilde{M} \,,
\ee
such that $\mathfrak{i}(M)$ is a $(k-1)$-coisotropic submanifold of $\widetilde{M}$.
%%%
The multisymplectic manifold $(\widetilde{M},\, \widetilde{\omega})$ is referred to as \textbf{multisymplectic tickening} of $(M,\, \omega)$.
%%%
\begin{proof}
Since $\omega$ has been assumed to have constant rank $r$, there exists a system of local coordinates
\be
\left\{\, x^a,\, z^A \,\right\}_{a=1,...,r; A=1,...,l} \,,
\ee
with $l=d - r$, such that $\omega$ reads
\be
\omega \,=\, \omega_{a_1 ... a_k}(x) \dd x^{a_1} \wedge ... \wedge \dd x^{a_k} \,,
\ee
whose characteristic distribution $\mathcal{V}$ reads
\be
\mathcal{V} \,=\, \operatorname{span} \left\{\, \frac{\de}{\de z^A} \,\right\}_{A=1,...,l} \,.
\ee
%%%
Consider a complementary distribution $\mathcal{H}$ providing an almost product structure on $M$ associated with the $(1,\,1)$-tensor $P$ on $M$ that, in the system of local coordinates considered, reads
\be
P \,=\, \left(\, \dd z^A - {P_x}^A_a \dd x^a \,\right) \otimes \frac{\de}{\de z^A} \,=:\, P^A \otimes \frac{\de}{\de z^A} \,.
\ee
%%%
The opposite $(1,\,1)$-tensor $R$ reads
\be
R \,=\, \mathds{1} - P \,=\, \dd x^a \otimes \left(\, \frac{\de}{\de x^a} + {P_x}^A_a \frac{\de}{\de z^A} \,\right) \,.
\ee
%%%
Consider the bundle ${\Lambda^{k-1}}^\perp_R(M)$ over $M$.
%%%
Denote by $\tau$ the projection map onto $M$.
%%%
Its sections are differential $(k-1)$-forms that locally reads
\be
\begin{split}
\alpha \,=\, &\alpha_{A_1...A_k} P^{A_1} \wedge ... \wedge P^{A_k} + \\
&\alpha_{a_1 A_2 ... A_k} \dd x^{a_1} \wedge P^{A_2} \wedge ... \wedge P^{A_k} + \\
&\;\;\vdots \\
&\alpha_{a_1...a_{k-1}A_k} \dd x^{a_1} \wedge... \wedge \dd x^{a_{k-1}} \wedge P^{A_k} \,.
\end{split}
\ee
%%%
Thus, a system of adapted coordinates on ${\Lambda^{k-1}}^\perp_R(M)$, reads
\be
\left\{\, x^a,\, z^A,\, \mu_{A_1 \,...\, A_{k-1}},\, \mu_{a_1 \, A_2 \, ... \, A_{k-1}}, ...,\, \mu_{a_1\,...\,a_{k-1} \, A_{k-1}} \,\right\}_{a_j=1,...,r; A_j=1,...,l} \,.
\ee
%%%
The projection map $\tau$ thus reads
\be
\tau  \colon  (x^a,\, z^A,\, \mu_{A_1 \,...\, A_{k-1}},\, \mu_{a_1 \, A_2 \, ... \, A_{k-1}}, ...,\, \mu_{a_1\,...\,a_{k-1} \, A_{k-1}}) \mapsto (x^a,\, z^A) \,.
\ee
%%%
The bundle ${\Lambda^{k-1}}^\perp_R(M)$ has a distinguished $(k-1)$-form, which is the obvious analogous of the tautological $1$-form of $\mathbf{T}^* M$, defined as
\be
\vartheta^P_\alpha(X_1, ...,\, X_{k-1}) \,=\, \alpha(\tau_*(X_1),...,\, \tau_*(X_{k-1})) \,, \;\;\; \forall \,\, X_1,...,X_{k-1} \in \mathbf{T}_\alpha M \,,
\ee
where $\alpha$ has to be considered as a point in ${\Lambda^{k-1}}^\perp_R(M)$ on the left hand side, and as a differential $(k-1)$-form on $M$ on the right hand side.
%%%
In the system of local coordinates chosen, $\vartheta^P$ reads
\be
\begin{split}
\vartheta^P \,=\, &\mu_{A_1\,...\,A_{k-1}} P^{A_1} \wedge ... \wedge P^{A_{k-1}} + \\ 
&\mu_{a_1 \, A_2 \,...\, A_{k-1}} \dd x^{a_1} \wedge P^{A_2} \wedge ... \wedge P^{A_{k-1}} +\\
&\;\; \vdots \\
&\mu_{a_1 \,...\, a_{k-2} \, A_{k-1}} \dd x^{a_1} \wedge ... \wedge \dd x^{a_{k-2}} \wedge P^{A_{k-1}} \,.
\end{split}
\ee
%%%
Consider the differential $k$-form 
\be \label{Eq: multisymplectic structure tickening}
\widetilde{\omega} \,=\, \tau^* \omega + \dd \vartheta^P \,,
\ee
on ${\Lambda^{k-1}}^\perp_R(M)$.
%%%
It is closed by definition.
%%%
It is also non-degenerate on the whole ${\Lambda^{k-1}}^\perp_R(M)$.
%%%
Indeed, consider a vector field $X$ written in the basis
\be
\left\{\, H_a \,:=\, \frac{\de}{\de x^a} + {P_x}^A_a \frac{\de}{\de z^A},\, \frac{\de}{\de z^A},\, \frac{\de}{\de \mu_{A_1 \,...\, A_{k-1}}},\, \frac{\de}{\de \mu_{a_1\,A_2\,...\,A_{k-1}}},\,...\,,\, \frac{\de}{\de \mu_{a_1\,...\,a_{k-2}\,A_{k-1}}} \,\right\} \,,
\ee
as
\be
\begin{split}
X \,=\, X^a H_a + X^A \frac{\de}{\de z^A} + &X_{A_1 \,...\, A_{k-1}} \frac{\de}{\de \mu_{A_1 \,...\, A_{k-1}}} + \\
&X_{a_1 \, A_2 \,...\, A_{k-1}} \frac{\de}{\de \mu_{a_1\,A_2\,...\,A_{k-1}}} + \\
& \;\; \vdots \\
&X_{a_1\,...\,a_{k-2}\,A_{k-1}} \frac{\de}{\de \mu_{a_1\,...\,a_{k-2}\,A_{k-1}}} \,.
\end{split}
\ee
%%%
The contraction $i_X \widetilde{\omega}$ reads
\be \label{Eq: contraction multisymplectic tickening}
\begin{split}
i_X \widetilde{\omega} \,=\,  &{E^{(0)}}_{A_1 \,...\, A_{k-1}} \, P^{A_1} \wedge ... \wedge P^{A_{k-1}} + \\
&{E^{(1)}}_{a_1 \, A_2 \,...\, A_{k-1}} \, \dd x^{a_1} \wedge P^{A_2} \wedge \,...\, \wedge P^{A_{k-1}} + \\
&\;\; \vdots \\
&{E^{(k-1)}}_{a_1 \,...\, a_{k-1}} \, \dd x^{a_1} \wedge \,...\, \dd x^{a_{k-1}} \, + \\
& {F^{(0)}}^{A_1} \, \dd \mu_{A_1 \,...\, A_{k-1}} \wedge P^{A_2} \wedge \,...\, \wedge P^{A_{k-1}} + \\
&{F^{(1)}_0}^{a_1} \, \dd \mu_{a_1 \,A_2\,...\, A_{k-1}} \wedge P^{A_2} \wedge \,...\, \wedge P^{A_{k-1}} + \\ 
&{F^{(1)}_1}^{A_2} \dd \mu_{a_1 \, A_2 \, A_{k-1}} \wedge \dd x^{a_1} \wedge P^{A_3} \wedge \,...\, \wedge P^{A_{k-1}} + \\
& \;\; \vdots \\
&{F^{(k-1)}_{k-2}}^{a_1} \dd \mu_{a_1 \,...\, a_{k-2} A_{k-1}} \wedge \dd x^{a_2} \wedge ... \wedge \dd x^{a_{k-2}} \wedge P^{A_{k-1}} + \\
&{F^{(k-1)}_{k-1}}^{A_{k-1}} \dd \mu_{a_1 \,...\, a_{k-2} A_{k-1}} \wedge \dd x^{a_1} \wedge ... \wedge \dd x^{a_{k-2}} \,. 
\end{split}
\ee
%%%
The coefficient ${F^{(0)}}^{A_1}$ reads
\be
{F^{(0)}}^{A_1} \,=\, - X^{A_1} \,.
\ee
%%%
Since the form
\be
\dd \mu_{A_1 \,...\, A_{k-1}} \wedge P^{A_2} \wedge \,...\, \wedge P^{A_{k-1}} 
\ee
is independent of all the other forms appearing in the decomposition \eqref{Eq: contraction multisymplectic tickening}, imposing $i_X \widetilde{\omega} \,=\, 0\,,\;\;\; \forall \,\, X$ gives
\be \label{Eq: XA=0}
X^A \,=\, 0 \,,\;\; \forall \,\, A=1,...,l \,.
\ee
%%%
With this condition in mind one gets
\begin{align}
{F^{(1)}_1}^A_2 \,=\, -X^{A_2} \,=\, 0 \,,\\
{F^{(1)}_0}^{a_1} \,=\, -X^{a_1} \,.
\end{align}
%%%
Since the form
\be
\dd \mu_{a_1 \,A_2\,...\, A_{k-2}} \wedge P^{A_2} \wedge \,...\, \wedge P^{A_{k-1}}
\ee
is independent of all the other forms appearing in the decomposition \eqref{Eq: contraction multisymplectic tickening}, imposing $i_X \widetilde{\omega} \,=\, 0\,,\;\;\; \forall \,\, X$ gives
\be \label{Eq: Xa=0}
X^a \,=\, 0 \,,\;\; \forall \,\, a=1,...,r \,.
\ee
%%%
With conditions \eqref{Eq: XA=0} and \eqref{Eq: Xa=0} in mind, one gets that all the other coefficients ${F^{(j)}_k}$ vanish, whereas
\begin{align}
{E^{(0)}}_{A_1 \,...\, A_{k-1}} \,&=\, X_{A_1\,...\,A_{k-1}} \,, \\
{E^{(1)}}_{a_1\,A_2 \,...\,A_{k-1}} \,&=\, X_{a_1 \,A_2 \,...\,A_{k-1}} \,, \\
& \vdots \\
{E^{(k-2)}}_{a_1 \,...\, a_{k-2}\,A_{k-1}} \,&=\, X_{a_1 \,...\, a_{k-2}\,A_{k-1}} \,,\\
{E^{(k-1)}}_{a_1 \,...\, a_{k-1}} \,&=\, 0 \,.
\end{align}
%%%
Since the $E^{(j)}_k$ are coefficients of independent $(k-1)$-forms, imposing $i_X \widetilde{\omega} \,=\, 0 \,,\;\; \forall \,\, X$ gives
\begin{align}
X_{A_1\,...\,A_{k-1}} \,&=\, 0 \,, \label{Eq: Xmu0=0} \\
X_{a_1 \,A_2 \,...\,A_{k-1}}\,&=\,0 \,, \label{Eq: Xmu1=0} \\
& \vdots  \nonumber \\
X_{a_1 \,...\, a_{k-2}\,A_{k-1}} \,&=\, 0 \,. \label{Eq: Xmuk-2=0}
\end{align}
%%%
Conditions \eqref{Eq: XA=0}, \eqref{Eq: Xa=0}, \eqref{Eq: Xmu0=0}, \eqref{Eq: Xmu1=0}, ..., \eqref{Eq: Xmuk-2=0} are the proof that 
\be
i_X \widetilde{\omega} \,=\, 0 \,,\;\; \forall\,\, X \in \mathfrak{X}({\Lambda^{k-1}}^\perp_R(M)) \;\; \implies \;\; X \,=\,0 \,,
\ee
namely, that $\widetilde{\omega}$ is non-degenerate, and, thus, multisymplectic. 
%%%

\noindent We will prove that $M$ is a $(k-1)$-coisotropic submanifold ${\Lambda^{k-1}}^\perp_R(M)$, that will be our $\widetilde{M}$.
%%%

\noindent Consider the contraction $\widetilde{\omega}_m(X,\, W_1,...,\,W_{k-1})$ for $m \in M$, where $W_1,\, W_2,\,...,\, W_{k-1}$ are tangent vectors to $M$ at $m$
\be
W_j \,=\, {W_j}^a H_a\bigr|_m + {W_j}^A \frac{\de}{\de z^A}\bigr|_m \,.
\ee
%%%
Such a contraction reads
\be
\begin{split}
\widetilde{\omega}_m(X,\, W_1,...,\,W_{k-1}) \,=\, &\tau^* \omega_m(X,\, W_1,...,\,W_{k-1}) \,+ \\
& X_{A_1\,...\,A_{k-1}} \left(\,P^{A_1} \wedge ... \wedge P^{A_{k-1}}\,\right)(W_1,...,\,W_{k-1}) \,+ \\
& X_{a_1\,A_2\,...\,A_{k-1}} \left(\, \dd x^{a_1} \wedge P^{A_2} \wedge ... \wedge P^{A_{k-1}} \,\right)(\,W_1,...,\,W_{k-1}) \,+\\
&\;\; \vdots \\
& X_{a_1\,...\,a_{k-2}\,A_{k-1}} \left(\, \dd x^{a_1} \wedge ... \dd x^{a_{k-2}} \wedge P^{A_{k-1}} \,\right)(W_1,...,\,W_{k-1}) \,.
\end{split}
\ee
%%%
We want to prove that imposing this contraction to be zero for all $W_j$ implies that $X$ is tangent to $M$.
%%%

\noindent Let us start by considering all the $W_j$ to be vertical, namely such that ${W_j}^a \,=\, 0$.
%%%
In this case the only contribution to the previous contraction reads
\be
X_{A_1\,...\,A_{k-1}} \left(\,P^{A_1} \wedge ... \wedge P^{A_{k-1}}\,\right)(W_1,...,\,W_{k-1}) \,.
\ee
%%%
Imposing it to vanish for all the $W_j$ implies 
\be
X_{A_1\,...\,A_{k-1}} \,=\, 0 \,.
\ee
%%%
Consider now $W_1$ to be horizontal, namely of the type $W_1 \,=\, {W_1}^a H_a$, and all the other $W_j$ to be vertical.
%%%
In this case, the only contribution to the contraction above is
\be
X_{a_1\,A_2\,...\,A_{k-1}} \left(\, \dd x^{a_1} \wedge P^{A_2} \wedge ... \wedge P^{A_{k-1}} \,\right)(\,W_1,...,\,W_{k-1}) \,.
\ee
%%%
Imposing it to be zero for all the $W_j$ implies 
\be
X_{a_1\,A_2\,...\,A_{k-1}} \,=\, 0 \,.
\ee
%%%
This argument can be iterated by considering all the other possible choices for the vector fields $W_j$. 
%%%
Eventually, one ends up with the conditions
\be
\begin{split}
X_{a_1\,a_2\,...\,A_{k-1}} \,&=\, 0 \,,\\
& \vdots \\
X_{a_1\,...\,a_{k-2}\,A_{k-1}} \,&=\, 0 \,,
\end{split}
\ee
thus proving that $X$ is tangent to $M$.
%%%
This completes the proof that $M$ is a $(k-1)$-coisotropic submanifold of a ${\Lambda^{k-1}}^\perp_R(M)$.
%%%

\noindent Thus, the multisymplectic manifold $\widetilde{M}$ we were searching for is the bundle ${\Lambda^{k-1}}^\perp_R(M)$, and the embedding $\mathfrak{i}$ is the zero section of $\tau$
\be
\mathfrak{i} \,=\, \sigma_0^\tau \,.
\ee
%%%
\end{proof}
\end{theorem}

\begin{remark}
Note that in this pre-multisymplectic case, the zero-section condition is not necessary to prove the non-degeneracy of $\widetilde{\omega}$.
%%%
\end{remark}

\begin{remark}
Interestingly, the same co-symplectic thickening described in \cref{Subsubsec: Existence co-symplectic} can be found by directly working on the pre-multisymplectic manifold $(M,\, \omega \wedge \eta)$.
%%%
In this case, if one considers the whole ${\Lambda^2}^\perp_R(M)$, with the system of adapted local coordinates
\be
\left\{\, t,\, q^a,\,p_a,\, z^A,\, \mu_{AB},\, \mu_{Aa},\, \mu_A^a,\, \mu_{At},\, \mu_{ta},\, \mu_{t}^a \,\right\} \,,
\ee
one gets the following multisymplectic structure  
\be \label{Eq: multisymplectic 3form}
\widetilde{\omega \wedge \eta} \,=\, \tau^* \omega \wedge \eta + \dd \vartheta^P
\ee
for 
\be
\vartheta^P \,=\, \mu_{AB} P^A \wedge P^B + \mu_{Aa} P^A \wedge \dd q^a + \mu_A^a P^A \wedge \dd p_a + \mu_{At} P^A \wedge \dd t + \mu_{ta} \dd t \wedge \dd q^a + \mu_{t}^a \dd t \wedge \dd p_a \,,
\ee
following \cref{Thm: multisymplectic thickening}.
%%%
Even if it is non-degenerate, the $3$-form \eqref{Eq: multisymplectic 3form} can not be associated to a cosymplectic structure because it does not have, in general, the structure
\be
\widetilde{\omega \wedge \eta} \,=\, \widetilde{\omega} \wedge \widetilde{\eta} \,.
\ee
%%%
For this reason, one may restrict to the intersection
\be \label{Eq: intersection ideal eta}
{\Lambda^2}^\perp_R(M) \cap \mathcal{I}(\eta) \,,
\ee
where $\mathcal{I}(\eta)$ is the differential ideal generated by the differential $1$-form $\eta$.
%%%
The intersection \eqref{Eq: intersection ideal eta} reads the subbundle of $\Lambda^2(M)$ consisting of differential $2$-forms having only components along $\eta$ and $P^A$.
%%%
Denote it by $\widetilde{M} \,=\, {\Lambda^2}^\perp_R(M) \cap \mathcal{I}(\eta)$.
%%%
A system of adapted local coordinates on it can be written as
\be
\left\{\, t,\, q^a,\, p_a,\, z^A,\, \mu_{tA} \,\right\} \,,
\ee
and the tautological form here reads
\be \label{Eq: tautological form intersection cosymplectic}
\widetilde{\vartheta}^P \,=\, \mu_{tA} \tau^*\eta \wedge P^A \,.
\ee
%%%
The multisymplectic structure defined on $\widetilde{M}$ out of \eqref{Eq: tautological form intersection cosymplectic} is
\be \label{Eq: cosymplectic structure tickening}
\widetilde{\omega \wedge \eta} \,=\, \tau^* \omega \wedge \eta - \dd (\mu_{tA} P^A) \wedge \tau^*\eta \,=\, (\tau^* \omega - \dd (\mu_{tA} P^A)) \wedge \tau^* \eta \,.
\ee
%%%
Following the general theory presented in \cref{Thm: multisymplectic thickening}, $\widetilde{M}$ is a multisymplectic manifold and $M$ is a $2$-coisotropic submanifold of $\widetilde{M}$.
%%%
Furthermore, \eqref{Eq: cosymplectic structure tickening} now has the structure of a $2$-form associated with a cosymplectic structure.
%%%
In particular, it reads
\be
\widetilde{\omega \wedge \eta} \,=\, \widetilde{\omega} \wedge \widetilde{\eta} \,, 
\ee
for the cosymplectic structure
\begin{align}
\widetilde{\omega} \,&=\, \tau^* \omega + \dd \vartheta^P \,, \\
\widetilde{\eta} \,&=\, \tau^* \eta \,,
\end{align}
where $\vartheta^P$ reads the $2$-form \eqref{Eq: thetaP cosymplectic}, which coincides with the cosymplectic structure obtained in \cref{Thm: cosymplectic tickening}.
%%%
\end{remark}

\subsubsection{Uniqueness}

Also in this case, given the generality of multisymplectic structures, uniqueness can not be proven in general. What we offer is a general tool that allows us to characterize multisymplectic coisotropic embeddings, when the geometry of the multisymplectic manifold is \textit{fixed}. Again, we give two examples, one where topological uniqueness is found (here we also give sufficient conditions for these to be neighborhood equivalent), and one where we do not find uniqueness.

\noindent The main technique that we will use is the following to prove neighborhood equivalence is the following:

\begin{lemma} \label{thm:Extensions_of_forms}
Let $\mathfrak{i} \colon M \rightarrow \widetilde M$ be an embedding and $\widetilde \omega_1, \widetilde \omega_2\in \Omega^{k}(\widetilde M)$ two closed forms such that $\widetilde \omega_1 = \widetilde \omega_2$ on $\T \widetilde M |_M$. Denote by $\widetilde \omega_t:= \widetilde \omega_1 + t (\widetilde \omega_2 - \widetilde \omega_1)$. Suppose there exists a complete vector field $\Delta$ on a neighborhood of $M$ that vanishes on $M$ such that
\begin{itemize}
    \item For every $x$ in said neighborhood, $$\lim_{t \to - \infty} \psi_t^{\Delta}(x)  \in M.$$
    \item It satisfies $\Lie_\Delta i_{\Delta} \widetilde \omega \in \bigcap_{t \in [0, 1]} \Im \flat_t$, where $\widetilde \omega = \widetilde \omega_2 - \widetilde \omega_1$
\end{itemize}
Then, the pairs $(M, \widetilde M, \widetilde \omega_1)$ and $(M, \widetilde M, \widetilde \omega_2)$ are neighborhood equivalent.
\end{lemma}
\begin{proof} We just apply Moser's trick. Indeed, by the relative Poincaré Lemma (\cref{thm:Relative_Poincaré_Lemma}), and by taking a tubular neighborhood such that $\Delta$ is the induced Liouville vector field, we have that there exists a $(k-1)$-form $\theta$ vanishing on $M$ such that 
\[
\widetilde \omega_2 - \widetilde \omega_1 = \dd \theta\,,
\]
where $\theta$ is given by
\[
\theta := \int_0^1 \psi^\ast_t i_{\Delta_t} \left( \widetilde \omega_2 - \widetilde\omega_1\right) \dd t\,.
\]
Here, $\psi_t$ is multiplication by $t$ and $\Delta_t = \frac{\dd \psi_t}{\dd t}$. Then, in order to apply Moser's trick we need to guarantee the existence of a vector field $X_t$ such that 
\[
i_{X_t} \widetilde \omega_t = - \theta\,.
\]
This, in turn, will follow if we can show that for each $s$ there exists a smooth choice of vector field $X_{s,t}$ satisfying 
\begin{equation}
\label{eq:Moser's trick_ condition}
    i_{X_{s,t}} \widetilde \omega_t = -\psi^\ast_t i_{\Delta}(\widetilde \omega_2 - \widetilde \omega_1) = -\psi^\ast_s i_\Delta \widetilde \omega\,,
\end{equation}

since we may define $X_t := \int_0^1 s^k X_{s, t} \dd s$. Eq. \eqref{eq:Moser's trick_ condition} reads as $\psi^\ast_s i_\Delta \widetilde \omega \in \bigcap_{t \in [0,1]} \Im \flat_t$, for every $s$, which is equivalent to $\Lie_{\Delta_t} i_{\Delta} \widetilde \omega \in \bigcap _{t \in [0,1]} \Im \flat_t$ and, since $\Delta_t$ is proportional to $\Delta$, this last conditions follows by hypothesis, showing that we may apply Moser's trick, obtaining neighborhood equivalence.
\end{proof}

\noindent To illustrate the previous ideas, let us prove the uniqueness of a particular type of coisotropic embedding, the one presented in \cref{subsection:multisymplectic_existence}, employing \cref{thm:Extensions_of_forms}.

\noindent More particularly, let $(M, \omega)$ be an arbitrary pre-multisymplectic manifold with kernel of constant rank and let $\mathfrak{i} \colon (M, \omega) \longrightarrow (\widetilde M, \widetilde \omega)$ be a coisotropic embedding into a multisymplectic manifold that has the linear type of the embedding built in \cref{subsection:multisymplectic_existence}. That is, if we choose $\mathcal{H}$ a distribution complementary to $\mathcal{V} = \ker\, \omega$, $\T M = \mathcal{H} \oplus \mathcal{V}$, then for each $x \in M$ there is a multisymplectomorphism 
\[
f_x\colon \T _{\mathfrak{i}(x)} \widetilde M \rightarrow \T_x M \oplus \bigoplus_{j= 1}^{k-1} \left( \Lambda^{k-1-j}\mathcal{H}^\ast \otimes \Lambda^j \mathcal{V}^\ast \right)\,,
\]
that is the identity on $\T_{\mathfrak{i}(x)} M$,
where the multisymplectic form in the latter space is chosen as $\Omega + \omega|_x$, where $\Omega$ is the restriction of the canonical multisymplectic form in $\T_x M \oplus \Lambda^{k-1} (M)$\,. We are now going to find (very mild) conditions on $\widetilde M$ to guarantee that then $\widetilde M$ is locally multisymplectomorphic to the bundle of transversal forms on $M$ so that, in particular, it has fixed topology. We would like to remark that these conditions are local. For the time being, let us fix a distribution $\widetilde{ \mathcal{H}}$ in a neighborhood of $M$ in $\widetilde M$ that is complementary to $\mathfrak{i}_\ast(\T M)$ on $M$, namely,
\[\T \widetilde M |_M = \mathfrak{i}_\ast(\T M) \oplus \widetilde{\mathcal{H}} |_M\,.\]

\begin{proposition}
\label{prop:isomorphism_with_transversal_forms}
Let $\phi\colon \widetilde{\mathcal{H}}|_M\longrightarrow \Lambda^k (M)$ be given by 
$\phi(\widetilde h|_{\mathfrak{i}(x)}) := p \left(\mathfrak{i}^\ast \left(i_{\widetilde h} \Omega|_{\mathfrak{i}(x)}\right)\right)$, where \[
p\colon \Lambda^{k-1} (M) \longrightarrow {\Lambda^{k-1}}^\perp_R (M)
\] 
is the projection onto the transversal forms and $\widetilde h$ is an element of $\widetilde{\mathcal{H}}$. Then, $\phi$ defines a vector bundle isomorphism into ${\Lambda^{k-1}}^\perp_R (M)$. 
\end{proposition}
\begin{proof} Using the point-wise structure, that is, the existence of the multisymplectomorphism 
\[
f_x \colon  \T _{\mathfrak{i}(x)} \widetilde M \rightarrow \T_x M \oplus \bigoplus_{j= 1}^{k-1} \left( \Lambda^{k-1-j}\mathcal{H}^\ast \otimes \Lambda^j \mathcal{V}^\ast \right)\,,
\]
we only need to prove that the previous map is a bundle isomorphism into the space of transversal forms (with respect to the appropriate almost product structure), for any complementary $\widetilde{\mathcal{H}}$ of $\T_x M$ in the previous space. Now, any complementary $\widetilde{\mathcal{H}}$ is given by the image of $\bigoplus_{j= 1}^{k-1} \left( \Lambda^{k-1-j}\mathcal{H}^\ast \otimes \Lambda^j \mathcal{V}^\ast \right)$ under a map $A \oplus\operatorname{id}$, where
\[
A \colon  \bigoplus_{j= 1}^{k-1} \left( \Lambda^{k-1-j}\mathcal{H}^\ast \otimes \Lambda^j \mathcal{V}^\ast \right) \longrightarrow \T_x M
\]
is an arbitrary linear map. Hence, we only  need to check that the map 
\[
\phi \colon \bigoplus_{j= 1}^{k-1} \left( \Lambda^{k-1-j}\mathcal{H}^\ast \otimes \Lambda^j \mathcal{V}^\ast \right) \longrightarrow {\Lambda^{k-1}}^\perp_R (M) \, \quad :\;\; v \mapsto p \left(\mathfrak{i}^\ast(i_v \Omega + i_{A(v)} \omega ) \right)\,.
\]
defines an isomorphism, for an arbitrary linear map $A$.
Notice, however, that $i_u \omega$ is parallel and that $i_v \Omega$ is transversal so that, $p(i_u \omega) = 0$ and $p (\mathfrak{i}^\ast(i_v \Omega)) = \mathfrak{i}^\ast (i_v \Omega)$ and hence
\[
\phi(v) = \left( i_v \Omega\right)|_{\T_x M}\,
\]
which is easily checked to define a linear isomorphism into the space of transversal forms.
\end{proof}

\noindent \cref{prop:isomorphism_with_transversal_forms}, after employing the tubular neighborhood theorem, shows that any coisotropic embedding of $M$ that has the linear type of embeddings into the space of transversal forms is neighborhood \textit{diffeomorphic} (not necessarily multisymplectomorphic) to 
\[
M \hookrightarrow {\Lambda^{k-1}}^\perp_R (M)\,,
\]
given by the zero section. We now focus on giving conditions for it to be neighborhood \textit{multisymplectomorphic}. As it is reasonable to assume, it has to do with the flatness of the embedding:

\begin{theorem}
\label{thm:coisotropic_uniqueness_theorem}
Assume that $\widetilde{\mathcal{H}}$ can be chosen to be $1$-Lagrangian and integrable, at least in a neighborhood of $M$, which we denote by $\widetilde U$. Further assume that the fibered manifold that the distribution induces $\pi \colon \widetilde U \rightarrow M$ is such that $i_\Delta \widetilde \omega$ is transversal, for every vertical vector field $\Delta$, and for a particular almost product structure. Then, there are two neighborhoods $U$ and $V$ of $M$ in $\widetilde M$ and ${\Lambda^{k-1}}^\perp_R (M)$ ($M$ interpreted as the zero section), respectively, and a multisymplectomorphism
\[\psi\colon U \rightarrow V\,.\]
\end{theorem}
\begin{proof} First observe that the map defined in \cref{prop:isomorphism_with_transversal_forms} is actually a multisymplectomorphism of vector bundles on $\widetilde M$. Indeed, denoting by $\widetilde \omega$ the form on ${\Lambda^{k-1}}^\perp_R (M)$, for $x \in \T_x M$, $v_1, \dots, v_k \in \T_x M$ and $\widetilde{h}_1, \dots, \widetilde{h}_k \in \widetilde{\mathcal{H}} |_M$ we have
\begin{align*}
    \left(\phi^\ast  \widetilde \omega|_x\right) (v_1 + \widetilde{h}_1, \dots, v_k + \widetilde{h}_k) &= \widetilde\omega|_x \left(v_1 + \phi(\widetilde{h}_1), \dots, v_k + \phi(v_k) \right)\\
    &= \widetilde\omega|_x \left(v_1 + (i_{\widetilde{h}_1} \Omega)|_{\T_x M}, \dots, v_k + (i_{\widetilde{h}_k} \Omega)|_{\T_x M} \right)\\
    &= \sum_{j = 1}^{k} (-1)^{j-1} \left(i_{\widetilde{h}_j} \Omega\right) (v_1, \dots, \widehat{v_j}, \dots, v_k)\\
    &= \sum_{j= 1}^k \Omega(v_1,\dots, \widetilde{h}_j, \dots, v_k)= \Omega(v_1 + \widetilde{h}_1, \dots, v_k + \widetilde{h}_k)\,,
\end{align*}
where in the last equality we have used that $\widetilde{\mathcal{H}}$ is $1$-Lagrangian. Now, using tubular neighborhood theory \cite{abraham} we may choose a diffeomorphism between two neighborhoods $M \subset U \subseteq {\Lambda^{k-1}}^\perp_R (M)$ and $M \subset V \subseteq {\Lambda^{k-1}}^\perp_R (M)$, 
\[
\psi\colon U \longrightarrow V
\]
such that $\psi_\ast |_{M} = \phi$ and such that the fibers of ${\Lambda^{k-1}}^\perp_R (M) \rightarrow M$ are mapped into the leaves of $\widetilde{\mathcal{H}}$. Therefore we have $\psi^\ast \Omega = \widetilde \omega$ on $M$. Let $\Delta$ denote the Euler vector field. It only remains to show that $\Delta$ satisfies the hypotheses of \cref{thm:Extensions_of_forms} for the forms $\widetilde \omega$ and $\psi^\ast \Omega$. The first condition is clear. Regarding the second one, notice that we may write
\begin{align*}
    \pounds_{\Delta} i_\Delta \left( \widetilde \omega - \psi^\ast \Omega\right) &= i_{\Delta} \dd i_{\Delta}\ \widetilde \omega - i_\Delta \dd i_\Delta \psi^\ast \Omega\\
    &= i_{\Delta} \dd i_{\Delta}\ \widetilde \omega - \psi^\ast \left(i_{\psi_\ast\Delta} \dd i_{\psi_\ast\Delta} \Omega\right)\,.
\end{align*}
We will show that each of these forms is transversal, and then that $\flat_t$ takes values in the transversal forms, for a neighborhood small enough. Indeed, the first part follows from the following lemma:

\begin{lemma}
\label{lemma:transversal_semibasic}
Let $\tau\colon \widetilde M \rightarrow M$ be a fibered manifold (surjective submersion) onto $M$ equipped with an almost product structure $\T M = \mathcal{H} \oplus \mathcal{V}$. Suppose $\widetilde \omega \in \Omega^k(\widetilde{M})$ such that $i_\Delta \widetilde \omega$ is a semi-basic transversal form to $\mathcal{H}$, for every vertical vector field. Then, $i_\Delta \dd i_\Delta \widetilde \omega$ is semi-basic and transversal to $\mathcal{H}$ as well.
\end{lemma}
\begin{proof} The fact that $i_\Delta \dd i_\Delta \widetilde \omega$ is semi-basic is clear. Now to show that it is transversal, it is enough to show that $(i_\Delta \dd i_\Delta \widetilde \omega) (\widetilde{h}_1, \dots, \widetilde{h}_k) = 0$, for every $\widetilde{h}_1, \dots, \widetilde{h}_k \in \mathcal{H}$\,. Indeed:
\begin{align*}
    (i_\Delta \dd i_\Delta \widetilde \omega)(\widetilde{h}_1, \dots, \widetilde{h}_k) &= \dd (i_ \Delta \widetilde \omega) (\Delta, \widetilde{h}_1, \dots, \widetilde{h}_k)\\
    &= \Delta \left(\widetilde \omega(\Delta, \widetilde{h}_1, \dots, \widetilde{h}_k) \right) + \sum_{j} (-1)^{j+1} \widetilde \omega (\Delta, [\Delta, \widetilde{h}_j], \widetilde{h}_1, \dots, \widehat{\widetilde{h}_j},\dots, \widetilde{h}_k)\,.
\end{align*}
The first term is null by hypotheses, and the second because $[\Delta, \widetilde{h}_j]$ is vertical, and $i_\Delta \widetilde \omega$ is semi-basic. Hence, we have $(i_\Delta \dd i_\Delta \widetilde \omega) (\widetilde{h}_1, \dots, \widetilde{h}_k) = 0$, which finishes the proof.
\end{proof}

\noindent By \cref{lemma:transversal_semibasic}, the form $\Lie_{\Delta} i_\Delta \left( \widetilde \omega - \psi^\ast \Omega\right)$ is transversal and semi-basic. It only remains to show that $\flat_t$ is surjective onto the space of transversal and semi-basic forms when restricted to vertical vector fields. Indeed, it is clear that on $M$ $\flat_1 = \flat_0$ defines a surjective mapping
\[
\{\text{vertical vectors}\} \longrightarrow {\Lambda^{k-1}}^\perp_R (M)\,.
\]
Furthermore, since $\psi$ maps fibers into leaves of $\widetilde{\mathcal{H}}$ (which is $1$-Lagrangian), in general $\flat_t$ is a map $\flat_t\colon \{\text{vertical vectors}\} \longrightarrow {\Lambda^{k-1}}^\perp_R (M)$. Now, since being surjective is an open condition, there will be a neighborhood of $M$ in ${\Lambda^{k-1}}^\perp_R (M)$ such that $\flat_t$ defines a surjective vector bundle map onto the space of transversal forms and thus, we have 
\[\Lie_\Delta i_\Delta \widetilde \omega \in \bigcap_{t\in[0,1]} \flat_t\,,\]
and, by \cref{thm:Extensions_of_forms}, we conclude that both $U$ and $V$ (after maybe being shrunk further) are neighborhood multisymplectomorphic.
\end{proof}

Now let us study an example where we do not find uniqueness:

\begin{example}
Let $V$ be a vector space and $L \subseteq V$ be a vector subspace. Define $\left(V \oplus \Lambda^k V^\ast, \Omega \right)$ to be the $k$-multisymplectic vector space with \[
\Omega (v_1 + \alpha_1, \dots, v_{k+1} + \alpha_{k+1}) = \sum_{j = 1}^{k+1}(-1)^{j+1} \alpha_j(v_1, \dots, \hat{v}_j, \dots, v_{k+1}),
\] 
and let $\left(L \oplus \Lambda^k V^\ast, \omega \right)$ be the $k$-pre-multisymplectic vector space given by $\omega = \mathfrak{i}^\ast \Omega$, where $\mathfrak{i}$ denotes the inclusion of the previous vector space into the former. Notice that $L \oplus \Lambda^k V^\ast$ defines a $k$-coisotropic subspace, with kernel $K = \{\alpha \in \Lambda^k V^ \ast \colon \alpha |_{L} = 0\} = \left(V \oplus \Lambda^k V^\ast \right)^{\perp_\Omega,\, k}$. Notice that the canonical embedding \[
\phi \colon L \oplus \Lambda^k V^\ast \hookrightarrow V \oplus \Lambda^k V^\ast
\]
is, in fact, a coisotropic embedding. Let us study coisotropic embeddings of a $k$-premultisymplectic manifold $(M, \omega)$ with linear type $(L \oplus \Lambda^k V^\ast, \omega)$ into a $k$-multisymplectic manifold $(\widetilde M, \widetilde \omega)$ with linear type $(V \oplus \Lambda^k V^\ast, \widetilde \omega)$. Furthermore, let us assume that point-wise, the coisotropic embedding is given by $\phi$, that is, we assume that the corresponding coisotropic embedding $\mathfrak{i} \colon M \hookrightarrow \widetilde M$ satisfies that, for each $x \in M$ there is a couple of (pre-)multisymplectic isomorphisms 
\begin{align*}
    f_x\colon \T_xM \longrightarrow L \oplus \Lambda^k V^\ast\,,\\
    g_x\colon \T_{f(x)} \widetilde M \longrightarrow V \oplus \Lambda^k V^\ast \,.
\end{align*}
that make the following diagram commutative:

% https://q.uiver.app/#q=WzAsNCxbMCwwLCJcXFRfeCBOIl0sWzIsMCwiXFxUX3tqKHgpfSBNIl0sWzAsMSwiTCBcXG9wbHVzIFxcYmlnd2VkZ2VeayBWXlxcYXN0Il0sWzIsMSwiVlxcb3BsdXMgXFxiaWd3ZWRnZV5rIFZeXFxhc3QiXSxbMCwyLCJmX3giXSxbMSwzLCJnX3giXSxbMCwxLCJqX1xcYXN0IiwxXSxbMiwzLCJcXHBoaSIsMV1d
\[\begin{tikzcd}[cramped]
	{\T_x M} && {\T_{\mathfrak{i}(x)} \widetilde M} \\
	{L \oplus \Lambda^k V^\ast} && {V\oplus \Lambda^k V^\ast}
	\arrow["{\mathfrak{i}_\ast}"{description}, from=1-1, to=1-3]
	\arrow["{f_x}", from=1-1, to=2-1]
	\arrow["{g_x}", from=1-3, to=2-3]
	\arrow["\phi"{description}, from=2-1, to=2-3]
\end{tikzcd}\,,\]
Let us prove that this sort of coisotropic embedding is not unique, by providing a pre-multisymplectic manifold $(M, \omega)$ with local type $(L \oplus  \Lambda^k V^\ast)$ and two different coisotropic embeddings into two different multisymplectic manifolds $(\widetilde M_i, \widetilde \omega_i)$, $i = 1, 2$ with linear type $(V \oplus \Lambda^k V^\ast, \widetilde \omega)$. Indeed, let us perform a general study to characterize, at least point-wise, these sort of embeddings. Let $\mathfrak{i} \colon M \hookrightarrow \widetilde M$ be such a map. Let $\mathcal{W}$ denote the $1$-Lagrangian distribution on $M$ given by $\mathcal{W}|_x = f_x ^{-1}\left(\Lambda^k V^\ast \right)$. It is well known that this distribution does not depend on the chosen multisymplectomorphism $f_x$ (see \cite{mar1988lett.math.phys.}). Let $\mathcal{L}$ be any complementary distribution to $\mathcal{W}$, so that $\T M = \mathcal{L} \oplus \mathcal{W}$ (we may even request $\mathcal{L}$ to be $k$-Lagrangian, but this is not necessary). Let \[
\mathcal{H} := \T \widetilde M \big |_{M} \big / \T M\,.
\]
\begin{lemma}
\label{lemma:lemma_isomorphism}
There is a vector bundle isomorphism 
\[
\phi\colon \mathcal{W} \longrightarrow\bigoplus_{j= 0}^{k-1}\Lambda^j \mathcal{H}^\ast \otimes  \Lambda^{k-1-j} \mathcal{L}^\ast\,.
\]
\end{lemma}
\begin{proof} Notice that 
\[
\bigoplus_{j= 0}^{k-1}\Lambda^j \mathcal{H}^\ast \otimes  \Lambda^{k-1-j} \mathcal{L}^\ast \cong \Lambda^{k-1} \left( \mathcal{H} \oplus \mathcal{L} \right)^\ast\,,
\]
and that we have 
\[
\T\widetilde M|_M \cong \mathcal{W} \oplus \mathcal{L} \oplus \mathcal{H}\,.
\]
Let $\psi \colon \T \widetilde M |_M \rightarrow \mathcal{W} \oplus \mathcal{L} \oplus \mathcal{H}$ denote a vector bundle isomorphism and define 
\[
\phi(w)(h_1+l_1, \dots, h_{k-1} + l_{k-1}) := \widetilde \omega(w, h_1+l_1, \dots, h_{k-1} + l_{k-1})\,.
\]
Given the local model, it is easy to check that this map defines a vector bundle isomorphism.
\end{proof}

\noindent Now, if $\mathcal{L}$ is $k$-Lagrangian, for each pair $(\mathcal{H}, \phi)$ of vector bundle over $M$ and isomorphism as in \cref{lemma:lemma_isomorphism}, we may build an embedding of the previous type such that $\T_{\mathfrak{i}(x)} \widetilde M / \T_{\mathfrak{i}(x)} M \cong \mathcal{H}$, so that if we find two different $\mathcal{H}$ with two linear isomorphisms, we find two different coisotropic embeddings:
\begin{theorem} Let $\mathcal{H} \rightarrow M$ be a vector bundle together with a vector bundle isomorphism 
\[
\phi\colon W \longrightarrow \Lambda^{k-1}\left( \mathcal{H} \oplus \mathcal{L}\right)^\ast\,.
\]
Then, there is a multisymplectic structure on a neighborhood of the zero section with the local type $\mathcal{V} \oplus \Lambda^k \mathcal{V}^\ast$ along $M$ such that $M$ (identified as the zero section) is a $(k-1)$-coisotropic submanifold.
\end{theorem}

\begin{proof} 
We have that 
\[
\T \mathcal{H} |_M \cong \mathcal{H} \oplus \mathcal{L} \oplus \mathcal{W}\,
\]
and we may define a $k$-form by 
\begin{align*}
    \widetilde \omega (h_1+l_1+w_1,\dots, h_k+l_k+w_k) = \sum_{j= 1}^{k} (-1)^{j+1} \phi(w_j)(h_1+ l_1, \dots, h_k + l_k)\,.
\end{align*}
This form clearly satisfies $\mathfrak{i}^\ast \widetilde \omega = \omega$, where $\mathfrak{i} \colon \T M \rightarrow \T \mathcal{H} |_{M}$ denotes the inclusion. Hence, it may be extended to a closed form on a neighborhood of $M$ in $\mathcal{H}$, proving the result.
\end{proof}

\noindent Hence, it is enough to give two different (non-isomorphic) vector bundles that satisfy the hypotheses of the previous theorem. Indeed, let 
\[
M = \mathbb{S}^1 \times \mathbb{R}^3 \times \Lambda^2 (\mathbb{R}^4) \cong \mathbb{S}^1 \times \mathbb{R}^3 \times \mathbb{R}^4 \times \Lambda^2 \mathbb{R}^{4 \ast},
\]
with pre-multisymplectic form given by the pullback of the natural multisymplectic form on $\Lambda^2 (\mathbb{R}^4)$. Since $M$ is parallelizable, making the identificacion $\T M  = M \times \mathbb{R}^{\dim M}$, we clearly have that $W$ is a trivial subbundle of dimension $10$,
\[
W \cong M \times \mathbb{R}^{10}\,,
\]
and that $L$ may be chosen as a trivial subbundle of dimension $4$. Hence, multisymplectic embeddings of the previous type are given by $1$-dimensional vector bundles over $M$, $\mathcal{H} \longrightarrow M$ such that
\[
\Lambda^2\left( \mathcal{L} \oplus \mathcal{H}\right) \cong\Lambda^2\left( \mathbb{R}^{4} \oplus \mathcal{H}\right)  \cong M \otimes \mathbb{R}^{10}\,.
\]
It turns out that $\mathcal{H}$ may be chosen both as the trivial bundle and as the pullback of the Möbius band bundle over $\mathbb{S}^1$, proving that there are \textit{at least} two non-isomorphic ways of embedding $M$ as a coisotropic submanifold.
\end{example}

\begin{example} As a more elementary example where uniqueness is not guaranteed is a coisotropic embedding of a smooth manifold with the zero premultisymplectic form, $M$, into a manifold $\widetilde M$ together with a volume form. Indeed, if $M$ is orientable (res. non-orientable), any orientable (res. non-orientable) line bundle $E \rightarrow M$ yields such an embedding and, in fact, any two embeddings of this kind are represented by certain orientable line bundle, as the coisotropic submanifolds of a manifold together with a volume form are the codimension one submanifolds.
\end{example}

\section{Conclusions and Future work}
%\addcontentsline{toc}{section}{Conclusions and Outlooks}

We have extended the coisotropic embbeding theorem for pre-symplectic manifolds to other geometric scenarios, that are fundamental to describe time-dependent mechanical systems, Lagrangian systems depending on the action and classical field theories. To do this, we have used a generic methodology, so that we are now in a good position to apply the obtained results in the so-called regularization problem for singular Lagrangian systems. Furthermore, we have studied the uniqueness of coisotropic embeddings in general and found that neighborhood equivalence only holds in the symplectic setting. However, if we relax the notion of equivalence of coisotropic embeddings, we obtain the classification summarized in \cref{table:uniqueness}, where we have added whether coisotropic embeddings in the different geometries considered are unique, where uniqueness is considered up to neighborhood diffeomoprhism, neighborhood diffeomorphisms that preserve the structure on $M$, and neighborhood diffeomorphisms that preserve the structure.

\begin{table}[!ht]
\caption{Summary of uniqueness of embeddings}
\centering
\renewcommand{\arraystretch}{2}
\begin{tabular}{p{3.5cm}|p{3.5cm}|p{3.5cm}|p{3.5cm}|p{3cm}}
  Type of \newline geometry & Topological & Topological and\newline preserves structure on $M$ & Neighborhood \newline equivalent\\
  \cline{1-4}
  Symplectic & Yes & Yes & Yes & \multirow{4}{*}{$\left.\begin{array}{c}\\ \\ \\ \\ \\ \\ \\\end{array}\right\}\,\makecell{\text{Classical} \\ \text{Mechanics}}$}
 \\
  Cosymplectic & Yes & If Reeb vector fields coincide on $M$ & If previous and Reeb vector fields share orbits& \\
  Contact & Yes & If Reeb vector fields coincide on $M$ & If previous and Reeb vector fields share orbits& \\
  Cocontact & Yes & If Reeb vector fields coincide on $M$ pair-wise &If previous and Reeb vector fields share orbits pair-wise& \\
  \cline{1-4}
  $k$-symplectic & Depends on the fixed model & Depends on the fixed model &Depends on the fixed model& \multirow{4}{*}{$\left.\begin{array}{c} \\ \\ \\ \\ \\ \\ \end{array}\right\}\,\makecell{\text{Classical} \\ \text{Field}\\ \text{Theories} }$}
\\
  $k$-cosymplectic & Depends on the fixed model & Depends on the fixed model & Depends on the fixed model& \\
  $k$-contact & Depends on the fixed model & Depends on the fixed model &Depends on the fixed model& \\
  Multisymplectic & Depends on the fixed model & Depends on the fixed model & Depends on the fixed model& \\
\end{tabular}
\label{table:uniqueness}
\end{table}

\noindent So, in an ongoing paper, we will apply the above results in these directions:

\begin{itemize}

\item Consider singular Lagrangian systems defined on $\mathbb{R} \times \T Q$, where $Q$ is the configuration manifolds. A constraint algorithm has been developed in \cite{chinea}, and a classification of these type of Lagrangians can be found in \cite{herminia,ibortmarin2}. Combining both results one could extend the results in \cite{Ibort-Solano}.

\item The regularization problem for singular Lagrangian systems depending on the action relies in two previous results; first, the corresponding constraint algorithm developed in \cite{mdlmlvsingular} and the classification of this type of Lagrangians \cite{inprepar}. The main difference with the previous case and the pre-symplectic case is that we are dealing with a Jacobi bracket and not a Poisson bracket.

\item The coisotropic embedding theorem also received several reformulations over the years. We mention the approach of \textit{V.
Guilleming} and \textit{S. Sternberg} that proved an equivariant (with respect to the action of a Lie group) version of the original theorem (see \cite{guillemin_sternberg}). Our plan is to extend that construction to the rest of the geometric scenarios for mechanics and Classical Field Theories.

\item Finally, singular classical field theories have been considered in many recent papers, and we will only mention a few of them, sending the reader to the actual references in these papers \cite{aitor,arturonarciso,singularfields0,singularfields1}. We should first obtain a classification of Lagrangians on the space of 1-jets of the configuration bundle using the operator introduced by \textit{D. Saunders} \cite{saunders} using a volume form on the space-time manifold (the base of the configuration bundle). Then, our intention is to study the regularization problem on these premises.
    
\end{itemize}

\section*{Acknowledgements}
M. de L., R. I. and P.S. acknowledge financial support from the Spanish Ministry of Science, Innovation and Universities under grants PID2022-137909NB-C21 and the Severo Ochoa Program for Centers of Excellence in R\&D (CEX2023-001347-S).
L.S. acknowledges financial support from Next Generation EU through the project 2022XZSAFN – PRIN2022 CUP: E53D23005970006.
%%%
L.S. is a member of the GNSGA (Indam).
%%%

\bibliographystyle{plain}

\bibliography{references}

\end{document}